\documentclass[11pt]{amsart}

\usepackage{geometry}                
\usepackage{natbib}

\usepackage{graphicx}

\usepackage{epstopdf}
\usepackage{amsmath,amssymb,amsthm}

\usepackage{dsfont}
\usepackage{mathrsfs}

\usepackage{tikz}
\usetikzlibrary{arrows,positioning}
\usepackage[cmtip,arrow]{xy}
\usepackage{pb-diagram,pb-xy}

\usepackage{algorithm}
\usepackage{algorithmic}
\usepackage{enumerate}

\usepackage{verbatim}

\usepackage{subfigure}

\theoremstyle{plain}  
\newtheorem{thm}{Theorem}[section]
\newtheorem{lem}[thm]{Lemma}
\newtheorem{prop}[thm]{Proposition}
\newtheorem{cor}[thm]{Corollary}

\theoremstyle{definition}  
\newtheorem{defn}[thm]{Definition}

\theoremstyle{remark}  
\newtheorem{rem}[thm]{Remark}

\newcommand{\setof}[1]{\left\{ {#1}\right\}}

\newcommand{\N}{{\mathbb{N}}}

\newcommand{\R}{{\mathbb{R}}}

\newcommand{\cA}{{\mathcal A}}
\newcommand{\cB}{{\mathcal B}}

\newcommand{\cE}{{\mathcal E}}
\newcommand{\cF}{{\mathcal F}}

\newcommand{\cK}{{\mathcal K}}
\newcommand{\cL}{{\mathcal L}}
\newcommand{\cM}{{\mathcal M}}
\newcommand{\cN}{{\mathcal N}}

\newcommand{\cT}{{\mathcal T}}

\newcommand{\cV}{{\mathcal V}}
\newcommand{\cW}{{\mathcal W}}

\newcommand{\sA}{{\mathsf A}}

\newcommand{\sD}{{\mathsf D}}

\newcommand{\sH}{{\mathsf H}}

\newcommand{\sJ}{{\mathsf J}}

\newcommand{\sL}{{\mathsf L}}

\newcommand{\sN}{{\mathsf N}}
\newcommand{\sO}{{\mathsf O}}
\newcommand{\sP}{{\mathsf P}}
\newcommand{\sQ}{{\mathsf Q}}

\newcommand{\sU}{{\mathsf U}}

\newcommand{\sMG}{\mathsf{MG}}

\newcommand{\mvmap}{\rightrightarrows}

\newcommand{\sAtt}{{\mathsf{ Att}}}

\newcommand{\sMD}{{\mathsf{ MD}}}
\newcommand{\sInvset}{{\mathsf{ Invset}}}

\newcommand{\sANbhd}{{\mathsf{ ANbhd}}}

\newcommand{\pred}[1]{\stackrel{\leftarrow}{#1}}

\newcommand{\down}{\mathop{\downarrow}}

\newcommand{\Inv}{\mathop{\mathrm{Inv}}\nolimits}

\newcommand{\Int}{\mathop{\mathrm{int}}\nolimits} 
\newcommand{\cl}{\mathop{\mathrm{cl}}\nolimits}

\definecolor{gray85}{gray}{0.85} 
\definecolor{gray8}{gray}{0.8} 
\definecolor{gray7}{gray}{0.7} 
\definecolor{gray6}{gray}{0.6} 
\definecolor{gray5}{gray}{0.5} 
\definecolor{gray4}{gray}{0.4} 
\definecolor{gray35}{gray}{0.35} 

\title{Global Dynamics for Steep Sigmoidal Nonlinearities in Two Dimensions}
\author{Tom\'{a}\v{s} Gedeon}
\address{Department of Mathematical Sciences \\
Montana State University\\
Bozeman, MT 59715\\
tel. (406)-994-5359\\
fax. (406)-994-1789\\
{\rm gedeon@math.montana.edu}}
\author{Shaun Harker}
\address{Department of Mathematics, Hill Center-Busch Campus\\
Rutgers, The State University of New Jersey \\
Piscataway, NJ  08854-8019, USA}
\author{Hiroshi Kokubu}
\address{Department of Mathematics\\ 
Kyoto University \\
Kyoto, 606-8502, Japan}
\author{Konstantin Mischaikow}
\address{Department of Mathematics, Hill Center-Busch Campus\\
Rutgers, The State University of New Jersey \\
Piscataway, NJ  08854-8019, USA}
\author{Hiroe Oka}
\address{Department of Applied Mathematics and Informatics \\
Ryukoku University \\
Seta, Otsu 520-2194, Japan}

\date{June 30, 2015}                                           

\begin{document}
\maketitle

\section{Introduction}

This paper discusses a novel approach to obtaining mathematically rigorous results on the global dynamics of ordinary differential equations. The motivation is twofold.
The first arises from applications, in particular the study of regulatory networks.
Because of their centrality in the study of systems biology anything beyond the most cursory comments are beyond the scope of this introduction. Instead we refer the interested reader to \cite{albert:collins:glass}.
The second arises from ongoing work of the authors to develop a mathematical framework, which we call a Database for Dynamics, that provides a computationally efficient and mathematically rigorous analysis of global dynamics of multiparameter nonlinear systems.
Reasonable success has been obtained in the context of nonlinear systems generated by maps \cite{siam,chaos, bush:mischaikow, newton}.
However, extending these methods to ordinary differential equations is proving to be technically challenging since the rigorous evaluation of a map needs to be replaced by the rigorous evaluation of solutions to a differential equation \cite{mischaikow:mrozek:weilandt, miyaji:pilarczyk:gameiro:kokubu:mischaikow}.
We return to this topic in the conclusion.

A regulatory network can be represented as an annotated directed graph.
The vertices  represent a regulatory object, e.g. a protein, and an edge from node $m$ to node $n$ indicates that $m$ directly regulates $n$. 
The annotation on this edge indicates whether $m$ activates (up regulates) or represses (down regulates) $n$.
It is natural to think of modeling the dynamics of this system via an ordinary differential equation.
In the context of gene regulatory networks, the associated proteins have a natural decay rate, and thus given a network with $N$ genes one is led to a system ordinary differential equation of the form
\begin{equation}
\label{eq:abstract}
\dot{x}_n = -\gamma_n x_n + f_n(x),\quad n=1,\ldots, N
\end{equation}
where $\gamma_n >0$ and $f_n$ is independent of $x_m$ if there is no edge from node $m$ to node $n$.
In the context of applications at this level of generality there is little that one can hope to say.
However, it is typically assumed that the interactions have switch like behavior and thus the nonlinearities used to model  interactions between individual nodes are often assumed to have a sigmoidal shape.
A common construction is based on Hill functions
\begin{equation}
\label{eq:hill}
\frac{x^k}{\theta^k + x^k}\quad\text{or}\quad\frac{\theta^k}{\theta^k + x^k}
\end{equation}
where the former and latter expressions are used to model activation and  repression, respectively.
However, it is important to keep in mind that there are other models that are probably more representative of the underlying biochemistry \cite{bintuA, bintuB}.
Even for $N$ of moderate size the analysis of \eqref{eq:abstract} with arbitrary Hill functions is intractable.
A standard simplification is to let $k\to \infty$ in which case one obtains nonlinearities that take the form of piecewise constant functions.
We refer to this as a {\em switching system} (for a precise definition see Section~\ref{sec:switchSystem}) and denote it by
\begin{equation}
\label{eq:switchN}
\dot{x} = -\Gamma x + \Lambda(x),\quad x\in \R^N.
\end{equation}
We view \eqref{eq:switchN} as a computational model, the purpose of which is to give us insight into the behavior of the biologically motivated model \eqref{eq:abstract} in which the nonlinearities $f_n$ are Lipschitz continuous though a particular analytic form is not known.
In particular, we are not concerned with identifying solutions to \eqref{eq:switchN}.

We recall \cite{albert:collins:glass} that there is a long tradition of associating {\em state transition diagrams}, which take the form of directed graphs, to regulatory networks. 
The paths through the state transition diagram are then used to represent the dynamics of the network.
As is described in Section~\ref{sec:transitionDiagram} we use \eqref{eq:switchN} to define a particular choice of state transition diagram.
Because our focus is on dynamics we find it convenient to represent this state transition diagram as a combinatorial multivalued map $\cF\colon\cV \mvmap \cV$.  
In this notation, $\cV$ denotes the set of vertices in the state transition diagram and there exists a directed edge $u\to v$ in the state transition diagram if and only if $v\in \cF(u)$.

The number of elements in $\cV$ can grow rapidly as a function of the size of the regulatory network.  
In particular, in the approach we take here, if $O(n)$ denotes the number of out edges at node $n$, then the size of $\cV$ is of order
\[
\prod_{n=1}^N (O(n)+1).
\]
Cataloguing all the paths in a graph of this size is not practical.
However, there are efficient (both in time and memory) graph algorithms that allow one to identify essential dynamical structures: the {\em recurrent dynamics}, i.e.\ the nontrivial strongly connected components of $\cF$; and the {\em gradient-like dynamics}, i.e.\ the reachability, defined by paths in $\cF$, between the recurrent components (see \cite{siam,chaos} and references therein and \cite{thieffry}
for an application of these techniques in the context of state transition diagrams).
We encode this information in the form of a {\em Morse graph}, $\sMG(\cF)$.
This is the minimal directed acyclic graph such that each nontrivial strongly connected component  is represented by a distinct node  and the edges indicate the reachability information inherited from $\cF$ between the nodes.

The switching system \eqref{eq:switchN} has an $N$ dimensional phase space, but an $N+3E$ dimensional parameter space ($E$ denotes the number of edges).
In \cite{paper2} we describe a set of algorithms that allows us to decompose parameter space into semi-algebraic sets, such that on each set the state transition diagram $\cF$ is  constant and hence the global dynamics as described by the Morse graph $\sMG(\cF)$ is valid for each parameter value in the set.
The algorithms have been implemented \cite{databaseSoftware} and thus for moderate sized $N$ we have the capability of  describing via the Morse graphs the global dynamics associated with the state transition diagrams for {\em all} parameter values.

As is emphasized above we believe that from the biological perspective \eqref{eq:abstract} provides a more realistic model than \eqref{eq:switchN}. 
Therefore, to justify the biological relevance of the combinatorial computations  described above that are based on \eqref{eq:switchN}, it is important to be able to demonstrate that $\sMG(\cF)$ provides correct and meaningful information about the dynamics of \eqref{eq:abstract}.
Since a Morse graph  is a directed acyclic graph it generates a poset. 
Our goal is to translate the poset structure associated with $\sMG(\cF)$ into information about the structure of invariant sets for a flow generated by a smooth differential equation.
For this we make use of Morse decompositions as defined by Conley \cite{cbms}. 
Recall that given a continuous flow $\varphi\colon \R\times X\to X$ defined on a compact metric space $X$ a \emph{Morse decomposition} of $X$ consists of a finite collection of mutually disjoint compact invariant sets called \emph{Morse sets} indexed by a partially ordered set $(\sP,\leq)$ with the property that if 
\[
x\in X\setminus \bigcup_{p\in\sP}M(p),
\]
where $M(p)$ denotes the Morse set indexed by $p$, then
\[
\alpha(x,\varphi) \subset M(p)\quad\text{and}\quad \omega(x,\varphi) \subset M(q)
\]
where $q\leq p$.

The primary goal of this paper is Theorem~\ref{thm:controlledPert}, which roughly states that given a regulatory network we can use \eqref{eq:switchN} to construct a state transition diagram for which we can efficiently compute a Morse graph $\sMG(\cF)$ from which we can determine a Morse decomposition for the dynamics defined by a smooth system of the form \eqref{eq:abstract}.
The following outline indicates the tools and constructions that are used to obtain a proof of Theorem~\ref{thm:controlledPert}.

We begin in Section~\ref{sec:conley} with a brief description of  Conley theory.
Section~\ref{sec:posetLattice} contains elementary ideas from lattice and poset theory and a statement of Birkhoff's theorem that relates finite distributive lattices and finite posets.
Sections~\ref{sec:conleyComb} and \ref{sec:conleyCont} presents the necessary definitions of Conley theory in the settings of combinatorial and continuous dynamics, respectively.
Section~\ref{sec:translation} contains Theorem~\ref{thm:translate}, which provides the theoretical framework by which we translate the  information from the combinatorial dynamics to the continuous dynamics.

As indicated above Section~\ref{sec:switchSystem} presents the definition of a switching system \eqref{eq:switchN} and associated notation and definitions.
The material in this section is restricted to two-dimensional systems, but can straightforwardly be extended to systems of arbitrary finite dimension.
In Section~\ref{sec:controlP} $\delta$-constrained continuous switching systems
\begin{equation}
\label{eq:dconstrained}
\dot{x} = -\Gamma x + f^{(\delta)}(x)
\end{equation}
are defined.  The definition begins with a given switching system, which determines $\Gamma$  and a positive number $\delta$,
which indicates the width of a collar around the lines of discontinuity of $\Lambda$.  The function $f^{(\delta)}$ is obtained by replacing  $\Lambda$  by a continuous function on this collar.
Again, this construction is done in the setting of $\R^2$, but can be extended to $\R^n$.

Theorem~\ref{thm:controlledPert} states that the $\sMG(\cF)$ derived from the switching system \eqref{eq:switchN} determines a Morse decomposition for any associated $\delta$-constrained continuous switching system \eqref{eq:dconstrained} for any $0<\delta <\delta^*$, where $\delta^*$ is explicitly determined by  $\Gamma$ and $\Lambda$ \eqref{eq:delta}.
The strategy of the proof is as follows.
We use the results of Section~\ref{sec:conley} and in particular Birkhoff's theorem to pass from the poset structure induced by $\sMG(\cF)$ to a lattice of attractors for the state transition graph $\cF$. 
We use Theorem~\ref{thm:translate} to guide the construction of a lattice of forward invariant sets for the $\delta$-constrained continuous switching system.
Birkhoff's theorem is then once more employed to identify the poset structure of the Morse decomposition for $\delta$-constrained continuous switching system.

The state transition graph $\cF$ associated with a switching system \eqref{eq:switchN} is defined in Section~\ref{sec:transitionDiagram}. This construction is presented in the setting of $\R^2$, but has been extended to $\R^n$ \cite{paper2}. 

The major technical work of this paper is to construct a lattice of trapping regions for the $\delta$-constrained continuous switching system \eqref{eq:dconstrained} that is isomorphic to the lattice of attractors for $\cF$. 
This is done in two steps, both of which are restricted to $\R^2$. 
The first is to construct  elementary regions in $\R^2$ that we call tiles and chips, which are used to construct the trapping regions. 
Tiles are related to the definition of $f^{(\delta)}$ and defined in Section~\ref{sec:controlP}.
Chips are defined in Section~\ref{sec:chips} along with a proof that if the constraints on $\delta$ imposed by \eqref{eq:delta} are satisfied, then the flow of \eqref{eq:dconstrained} is transverse along the edges of interest of tiles and chips.
The second step involves the construction of the desired trapping regions. 
This is done in Section~\ref{sec:proof} in slightly more generality than needed as we construct a trapping regions associated with any forward invariant set of $\cF$.
There are two important remarks that need to be made about the content of Section~\ref{sec:proof}.  
First, it is a tedious case by case local analysis of the constructing and verification of the trapping regions and hence too cumbersome to generalize to higher dimensions.
Second, we  do not know of counter examples to the construction for higher dimensional systems, thus we believe that with an alternate proof it might be possible to generalize Theorem~\ref{thm:controlledPert} to $n$-dimensional systems.

The formal statement and proof of Theorem~\ref{thm:controlledPert} along with related results are presented in Section~\ref{sec:theorem}.

\section{Conley Theory}
\label{sec:conley}

Our proof that  efficient graph theoretic computations can lead to rigorous mathematical results for smooth switching systems is based on new developments in Conley theory as presented in a series of papers \cite{kmv0,kmv1,kmv2,kmv3}. 
We review the essential ideas of these results along the lines in which we employ these ideas in this paper.
We begin with a brief review of posets and lattices, describe Conley theory first in the context of directed graphs (combinatorial dynamics), and then in the setting of continuous flows on compact metric spaces. Finally, we state  a theorem using this language that provides the framework by which the state transition graph leads to mathematically rigorous statements about the global dynamics for a smooth switching system.

\subsection{Posets and Lattices}
\label{sec:posetLattice}

We assume the reader is familiar with the concepts of partially ordered sets (posets) and lattices (see \cite{roman, davey:priestley}), but review some fundamental concepts as a means of establishing notation.

Given a poset $(\sP,\leq)$, $\sD\subset \sP$ is a \emph{down set} of $\sP$ if $p\in \sP$, $q\in \sD$, and $p\leq q$ implies that $p\in\sD$.  
The set of down sets of $\sP$ is denoted by $\sO(\sP)$ and is a lattice under the operations of union and intersection.
In fact, $\sO$ defines a contravariant functor from the category of posets to the category of lattices.

Let $(\sL,\vee,\wedge,{\bf 0},{\bf 1})$ be a bounded distributive lattice where ${\bf 0}$ and ${\bf 1}$ denote the minimal and maximal elements, i.e.\
\[
{\bf 0}\wedge \sU = {\bf 0}\quad\text{and}\quad {\bf 1}\wedge \sU = \sU
\]
for all $U\in\sL$.
The lattice algebra induces a partial order on its elements as follows. Given $U,V\in\sL$
\begin{equation}
\label{eq:po}
\text{if}\quad U \wedge V = U\quad\text{then}\quad U \leq V.
\end{equation}

Recall that a nonzero element $\sU\in\sL$ is \emph{join irreducible} if $U = A\vee B$ implies that $U = A$ or $U = B$.
We denote the set of join irreducible elements of $\sL$ by $\sJ^\vee(\sL)$.
Observe that in a finite lattice if $U$ is a join irreducible element, then it has a unique predecessor with regard to the partial order \eqref{eq:po}.
We denote this unique predecessor by
\begin{equation}
\label{eq:!pred}
\pred{U}.
\end{equation}
Using the partial order \eqref{eq:po} $(\sJ^\vee(\sL),\leq)$ is a poset and, more generally, $\sJ^\vee$ defines a contravariant functor from the category of bounded distributive lattices to the category of finite posets.

\begin{thm}
[Birkhoff's theorem]
Let $(\sL,\vee,\wedge)$ be a finite distributive lattice and let $(\sP,\leq)$ be a finite poset.  Then,
\begin{enumerate}
\item[(i)] $\sO(\sJ^\vee(\sL))$ is lattice isomorphic to $\sL$.
\item[(ii)] $\sJ^\vee(\sO(\sP))$ is poset isomorphic to $\sP$.
\end{enumerate}
\end{thm}

As is made clear below our interest in Birkhoff's theorem lies in the fact that it guarantees that one can represent the same information either in a poset or a lattice and that there is a well defined transformation between the two.

\subsection{Combinatorial Conley Theory}
\label{sec:conleyComb}
 
 As indicated in the Introduction, to emphasize the fact that we are interested in dynamical structures we represent a directed graph as a combinatorial multivalued map $\cF\colon \cV \mvmap \cV$, where $\cV$ is the finite set of vertices and there is a directed edge $\nu \rightarrow \nu'$ if and only if $\nu'\in\cF(\nu)$.
Note that we allow self edges in our directed graph thus it is possible that $\nu\in\cF(\nu)$.
We use the  notation $\nu\leadsto \nu'$ to indicate the existence of a path from $\nu$ to $\nu'$. 
Using the multivalued map notation, $\nu\leadsto \nu'$ is equivalent to the statement that there exists $n\in\N$ such that $\nu'\in\cF^n(\nu)$.
Backward paths in the graph can be associated with reversal of time. With this in mind define $\cF^{-1}\colon \cV\to \cV$ by
\[
\nu'\in\cF^{-1}(\nu)\quad\text{if and only if}\quad \nu\in\cF(\nu').
\]

Elements $\nu,\nu'\in\cV$ belong to the same \emph{strongly connected path component} of $\cF$ if $\nu\leadsto \nu'$ and $\nu'\leadsto \nu$.
Since we allow self edges it is possible that a strongly connected path component consists of a single vertex with a self edge.
We refer to a strongly connected path component of $\cF$ as a {\em Morse set} of $\cF$ and denote it by $\cM\subset \cV$.
The collection of all strongly connected path components of $\cF$ is denoted by 
\[
\sMD(\cF) :=\setof{\cM(p)\subset \cV\mid p\in \sP}
\]
and forms a  {\em Morse decomposition} of $\cF$.
We impose a partial order on the indexing set $\sP$ of $\sMD(\cF)$  by defining
\[
q \leq p\quad \text{if there exists a path in $\cF$ from an element of $\cM(p)$ to an element of $\cM(q)$}.
\]

\begin{defn}
The \emph{Morse graph} of $\cF$, $\sMG(\cF)$, is the Hasse diagram of the poset $(\sP,\leq)$.
We refer to the elements of $\sP$ as the \emph{Morse nodes} of the graph.
\end{defn}

As discussed in the introduction given $\cF\colon \cV\mvmap \cV$ identification of Morse sets and the Morse graph is computationally feasible.
Thus, these are the objects that we extract from the state transition diagram.
However, we know of no direct means of transferring knowledge of the Morse graph to an associated continuous system of the form \eqref{eq:switchN}.
As indicated above Birkhoff's theorem guarantees that we will not lose information by considering the lattice of down sets
$\sO(\sP)$, where $(\sP,\leq)$ is the down set that defines the Morse graph.
To identify this lattice in the directed graph $\cF\colon \cV\mvmap \cV$ recall the following concept.

\begin{defn}
A set $\cN\subset \cV$ is \emph{forward invariant} under $\cF$ if $\cF(\cN)\subset \cN$.
\end{defn}

The collection of all forward invariant sets of $\cF$ is denoted by $\sInvset^+(\cF)$ and as discussed in \cite[Section 2]{kmv2} is a bounded distributive lattice where ${\bf 0} := \emptyset$ and ${\bf 1} =\cV$.
Given $\cN\in \sInvset^+(\cF)$ define
\[
\cN^0 := \setof{\nu\in \cN\mid \cF^{-1}(\nu)\cap \cN \neq \emptyset}.
\]

Of central interest is the following special type of forward invariant set.
\begin{defn}
A set $\cA\subset \cV$ is an \emph{attractor} for $\cF$ if $\cF(\cA) = \cA$.
\end{defn}
Observe that given an attractor $\cA$, $\cA^0 = \cA$.

The collection of all attractors in $\cV$ under $\cF$ is denoted by  $\sAtt(\cF)$ and as discussed in \cite[Section 2]{kmv2} is a bounded distributive lattice where ${\bf 0} := \emptyset$ and ${\bf 1} = \max\setof{\cA\mid \cA\in\sAtt(\cF)}$.
Furthermore, given $\cA_0,\cA_1\in\sAtt(\cF)$ the lattice operations are defined by
\begin{equation}
\label{eq:veeAtt}
\cA_0\vee \cA_1 := \cA_0\cup \cA_1
\end{equation}
and
\begin{equation}
\label{eq:wedgeAtt}
\cA_0\wedge \cA_1 := \max\setof{\cA\in\sAtt(\cF)\mid \cA\subset \cA_0\cap \cA_1}.
\end{equation}

Given a Morse set $\cM(p)$ let 
\[
\down(\cM(p))= \setof{\nu\in \cV\mid \exists \nu'\in \cM(p)\ \text{such that}\ \nu'\leadsto \nu}.
\]
Note that $\down(\cM(p))\in \sAtt(\cF)$ and, in fact, $\setof{\down(\cM(p))\mid p\in \sP}$ generates $\sAtt(\cF)$.
Thus $\sAtt(\cF) \cong \sO(\sP)$ (see \cite{kmv3} for details).

We  make use of the following result, which follows via a finite induction argument from the definition of a forward invariant set.

\begin{prop}
\label{prop:vinA}
Assume $\cN\in\sInvset^+(\cF)$.
If $\nu\in \cN$ and  $\nu\leadsto \nu'$, then $\nu'\in\cN$.
\end{prop}


\subsection{Continuous Conley Theory}
\label{sec:conleyCont}

We assume the reader is familiar with basic concepts from the theory of dynamical systems \cite{robinson}.
Let $\varphi\colon \R\times X\to X$ be a continuous flow defined on a compact metric space.
Let $\sInvset(X,\varphi)$ denote the collection of invariant sets in $X$ under $\varphi$.

\begin{defn}
A compact set $N\subset X$ is an \emph{attracting neighborhood} if $\omega(N,\varphi)\subset \Int(N)$ where $\omega(N,\varphi)$ denotes the $\omega$-limit set of $N$ under $\varphi$.
\end{defn}

The set of all attracting neighborhoods in $X$ under $\varphi$ is denoted  by $\sANbhd(X,\varphi)$ and as shown in \cite{kmv1} is a bounded distributive lattice (but in general is not finite).
In this paper we make use of a special class of attracting neighborhoods.

\begin{defn}
Assume the flow $\varphi$ is generated by a differential equation $\dot{x} = f(x)$, $x\in\R^2$.  
Let $N\subset \R^2$ be a regular closed set \cite{walker} whose boundary is made up of a finite number of straight line segments. 
For each closed boundary edge $e$ of $N$ let $n_e$ denote the outward normal vector.
We say that $N$ is an \emph{attracting block} if $f(x)\cdot n_e <0$ for every $x\in e$.
\end{defn}

Recall that an invariant set $A\in\sInvset(X,\varphi)$ is an \emph{attractor} for  $\varphi$ if there exists an attracting neighborhood such that $A = \omega(N,\varphi)$.  
The \emph{dual repeller} of an attractor $A$ is defined to be
\[
A^* := \setof{x\in X\mid \omega(x,\varphi) \cap A = \emptyset}.
\]

In what follows we assume that $\sA$ is a finite bounded sublattice of $\sANbhd(X,\varphi)$.
Recall that $\sJ^\vee(\sA)$ forms a poset.
Let $\iota\colon \sQ\to \sJ^\vee(\sA)$ be a poset isomorphism, i.e.\ we are using the poset $(\sQ,\leq)$ to index the elements of $\sJ^\vee(\sA)$.
Define $\mu \colon \sQ\to \sInvset(X,\varphi)$ by
\[
\mu(q) := \Inv(\iota(q),\varphi) \cap \left(\Inv\left( \stackrel{\leftarrow}{\iota(q)} ,\varphi \right)\right)^*.
\]
As is shown in \cite{kmv3} $\mu(\sQ)$ is a \emph{Morse decomposition} of $X$ under $\varphi$, i.e.\ it is a collection of mutually disjoint compact invariant sets with the following property:  if 
\[
x\in X\setminus \bigcup_{q\in\sQ} \mu(q),
\]
then there exists $q,q'\in \sQ$ such that
\[
\alpha(x,\varphi)\subset \mu(q)\quad\text{and}\quad \omega(x,\varphi)\subset \mu(q')
\]
and furthermore $q' < q$ under the partial order on $\sQ$.
The Hasse diagram of $\sQ$ is the \emph{Morse graph} associated to the Morse decomposition. 

\subsection{The Translational Theorem}
\label{sec:translation}

\begin{defn}
\label{defn:barP}
Given a poset $(\sP,\leq_\sP)$ define $(\bar{\sP},\leq_{\bar{\sP}})$ to be the poset where $\bar{\sP} = \sP \cup \setof{\bar{p}}$ and $\leq_{\bar{\sP}}$ restricted to $\sP$ equals $\leq_\sP$ along with the additional relations $p\leq_{\bar{\sP}} \bar{p}$ for all $p\in\sP$.
\end{defn}

The primary result of this paper is a corollary of the following theorem \cite{kmv3}.

\begin{thm}
\label{thm:translate}
Let $\varphi\colon \R\times X\to X$ be a flow on a compact metric space.
Let $(\sN,\wedge,\vee,{\bf 0},{\bf 1})$ be a finite distributive lattice that satisfies the following properties:
\begin{enumerate}
\item  $\sN\subset  \sInvset^+(X,\varphi)$.  	
\item  ${\bf 0}=\emptyset$,
\item  $N\wedge N' \subset N\cap N'$ and $N\vee N' = N\cup N'$.
\end{enumerate}
Let the poset $(\sP,\leq) \cong (\sJ^\vee(\sN),\leq)$  be an indexing set for $\sJ^\vee(\sN)$.
For each $p\in\bar{\sP}$ define
\[
M(p) := \begin{cases}
 \Inv(N(p),\varphi) \cap \left(\Inv\left(\pred{N(p)},\varphi \right)\right)^*  & \text{if $p\in \sP$}\\
 \left(\Inv\left({\bf 1},\varphi \right)\right)^* & \text{if $p=\bar{p}$}.
 \end{cases}
\]
Then the collection of invariant sets $M(p)$, $p\in \bar{\sP}$ defines a Morse decomposition of  $\varphi$ and $\leq_{\bar{\sP}}$ is an admissible order.
\end{thm}

\section{Two-dimensional Switching Systems}
\label{sec:switchSystem}
In this section we provide a formal definition of a general two-dimensional switching system and provide elementary results about the associated dynamics.  
We begin with two sets of  non-negative real numbers $\Xi := \setof{\xi_i\mid i=0,\ldots,I+1}$ and $\sH := \setof{\eta_j\mid j=0,\ldots, J+1}$ that we refer to as \emph{threshold values}, with the property that
\begin{align*}
0 & =\xi_0 < \xi_1 < \ldots < \xi_I < \xi_{I+1} = \infty \\
0 & =\eta_0 < \eta_1 < \ldots < \eta_J < \eta_{J+1} = \infty.
\end{align*}
Let 
\begin{equation}
\label{eq:gridpoints}
\Pi := \setof{(\xi_i,\eta_j) \mid i = 0,\ldots, I,\ j=0,\ldots, J}\subset [0,\infty)^2.
\end{equation}
In addition we assume that we are given partitions of the sets of threshold values
\[
\Xi = \Xi^1\cup \Xi^2\quad\text{and}\quad \sH = \sH^1\cup \sH^2.
\]
Of primary importance is the following collection of open rectangles, called \emph{cells}, defined in terms of the thresholds as follows:
\[
\cK : = \setof{ \kappa(i,j) : = (\xi_i,\xi_{i+1})\times (\eta_j,\eta_{j+1})\subset (0,\infty)^2 \mid i = 0,\ldots, I,\ j=0,\ldots, J}.
\]

\begin{defn}
\label{defn:switchingsystem}
The \emph{switching system} $\Sigma = \Sigma(\Gamma,\Lambda,\Xi^1,\Xi^2,\sH^1,\sH^2)$ is defined to be
the system of differential equations
\begin{equation}
\label{eq:switching}
\dot{x} = -\Gamma x + \Lambda(x), \qquad x\in \bigcup_{\kappa \in \cK} \kappa\subset (0,\infty)^2
\end{equation}
where
\[
\Gamma = \begin{bmatrix}\gamma_1 & 0 \\ 0 & \gamma_2 \end{bmatrix},\qquad \gamma_i >0
\]
and
\[
\Lambda(x) = \begin{bmatrix}\Lambda_1(x) \\ \Lambda_2(x)\end{bmatrix}
\]
is constant on the cells $\cK$.
Furthermore, $\Lambda$ satisfies the following constraints for all $i\in\setof{0,\ldots, I}$, $j\in\setof{0,\ldots, J}$,
\begin{equation}
\label{eq:Liconstant}
\begin{aligned}
\xi_i\in \Xi^2 \quad &\Rightarrow \quad \Lambda_1(\kappa(i,j)) = \Lambda_1(\kappa(i+1,j)) \\ 
\xi_i\in \Xi^1 \quad &\Rightarrow \quad \Lambda_2(\kappa(i,j)) = \Lambda_2(\kappa(i+1,j)) \\ 
\eta_j\in \sH^1 \quad &\Rightarrow \quad \Lambda_2(\kappa(i,j)) = \Lambda_2(\kappa(i,j+1)) \\ 
\eta_j\in \sH^2 \quad &\Rightarrow \quad \Lambda_1(\kappa(i,j)) = \Lambda_1(\kappa(i,j+1)).
\end{aligned}
\end{equation} 
For the sake of simplicity  we assume that  if $\kappa_i \neq \kappa_j$, then 
\[
\Lambda(\kappa_i) \neq \Lambda(\kappa_j).
\]
Furthermore, setting
\begin{equation}
\label{eq:PHI}
\Phi(\kappa) := \Gamma^{-1}\Lambda(\kappa)
\end{equation}
we assume that for all $\kappa\in\cK$
\begin{equation}
\label{eq:notTvalue}
\Phi_1(\kappa) \not\in \Xi\quad\text{and}\quad \Phi_2(\kappa) \not\in \sH.
\end{equation}
\end{defn}

A consequence of assumption \eqref{eq:notTvalue} is that for a fixed switching system $\Sigma$ 
\begin{equation}
\label{eq:muSigma}
\mu =\mu(\Sigma) :=  \min_{\kappa\in\cK,\ \xi\in\Xi,\ \eta\in\sH}\setof{|\Phi_1(\kappa)-\xi |,|\Phi_2(\kappa)-\eta |} > 0.
\end{equation}
Two other positive constants that are used later in the paper are 
\begin{equation}
\label{eq:PointDisplacement}
\rho = \rho(\Sigma) :=  \max_{\kappa =\kappa(i,j)\in \cK}\setof{ | \Phi_1(\kappa)-\xi_i|,| \Phi_1(\kappa)-\xi_{i+1}|,  | \Phi_2(\kappa)-\eta_j|,| \Phi_2(\kappa)-\eta_{j+1}|},
\end{equation}
which provides a measurement of the maximal displacement of the attracting fixed point $\Phi(\kappa)$ from the cell $\kappa$, and
\begin{equation}
\label{eq:gammaRatio}
\bar{\gamma} = \bar{\gamma}(\Sigma) :=  \min\setof{ \frac{\gamma_1}{\gamma_2}, \frac{\gamma_2}{\gamma_1}}\leq 1.
\end{equation}

Given a particular cell $\kappa = \kappa(i,j)\in \cK$ there is an associated affine vector field, which we call the \emph{$\kappa$-equation},  given by
\begin{equation}
\label{eq:Kaffine}
\dot{x} = -\Gamma x + \Lambda(\kappa).
\end{equation}
We denote the flow generated by the $\kappa$-equation by $\psi_\kappa$.
Observe that $\Phi(\kappa)$, as defined in \eqref{eq:PHI}, is an attracting fixed point for the $\kappa$-equation.

We begin by labeling the cells according to the behavior of the associated $\kappa$-equation on the cell.
In particular, 
\begin{equation}
\label{eq:coarseType}
\text{$\kappa(i,j)$ is of type}\ 
\begin{cases}
N & \text{if $\xi_i < \Phi_1(\kappa) < \xi_{i+1}$ and $\eta_{j+1} < \Phi_2(\kappa)$} \\
NE & \text{if $\xi_{i+1} <\Phi_1(\kappa) $ and $\eta_{j+1} < \Phi_2(\kappa)$} \\
E & \text{if $\xi_{i+1} <\Phi_1(\kappa) $ and $\eta_{j} < \Phi_2(\kappa)< \eta_{j+1}$}  \\
SE & \text{if $\xi_{i+1} <\Phi_1(\kappa) $ and $ \Phi_2(\kappa)$}< \eta_{j}   \\
S & \text{if $\xi_i < \Phi_1(\kappa) < \xi_{i+1}$ and $ \Phi_2(\kappa)$}< \eta_{j}   \\
SW & \text{if $ \Phi_1(\kappa) < \xi_{i}$ and $ \Phi_2(\kappa)< \eta_{j}$}   \\
W & \text{if $ \Phi_1(\kappa) < \xi_{i}$ and $\eta_{j} < \Phi_2(\kappa)< \eta_{j+1}$}  \\
NW & \text{if $ \Phi_1(\kappa) < \xi_{i}$ and $\eta_{j+1} < \Phi_2(\kappa)$}  \\
A & \text{if $\xi_i < \Phi_1(\kappa) < \xi_{i+1}$ and $\eta_{j} < \Phi_2(\kappa)< \eta_{j+1}$}
\end{cases}
\end{equation}

We remark that  assumption \eqref{eq:notTvalue} implies that every cell $\kappa$ is of the type indicated above.
A cell of type $A$ is called an \emph{attracting cell}; a cell of type $N$, $E$, $S$, or $W$ is called a \emph{focussing cell}; and a cell of type $NE$, $SE$, $SW$, or $NW$ is called a \emph{translating cell}.
Furthermore, for bookkeeping purposes we find it convenient to place a star in the cell to indicate its type (see Figure~\ref{fig:stars}).

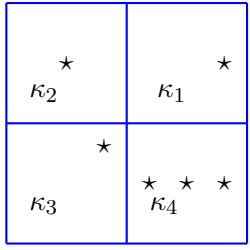
\begin{figure}
\begin{center}
\begin{tikzpicture}
[scale=0.4]

	\draw[blue, thick](0,0) -- (8,0);
	\draw[blue, thick](0,4) -- (8,4);
	\draw[blue, thick](0,0) -- (0,8);
	\draw[blue, thick](8,0) -- (8,8);
	\draw[blue, thick](4,0) -- (4,8);
	\draw[blue, thick](0,8) -- (8,8);
	
	\draw(1.25,5) node{$\kappa_2$};

	\draw(2,6) node{$\star$};
	
	\draw(5.5,5) node{$\kappa_1$};
	
	\draw(7.25,6) node{$\star$};

	\draw(1.25,1.25) node{$\kappa_3$};
	
	\draw(3.25,3.25) node{$\star$};
	
	\draw(5.25,1.25) node{$\kappa_4$};
	
	\draw(4.75,2) node{$\star$};
	\draw(6,2) node{$\star$};
	\draw(7.25,2) node{$\star$};
	
\end{tikzpicture}
\end{center}
\caption{The figure indicates four cells labeled $\kappa_i$, $i=1,2,3,4$. 
The star in the center of cell $\kappa_2$ indicates that it is of type $A$. 
The star in the center right of cell $\kappa_1$ indicates that it is of type $E$. 
The star in the upper right corner of cell $\kappa_3$ indicates that it is of type $NE$.
The multiple stars in $\kappa_4$ indicate that it is of type $W$, $A$, or $E$.}
\label{fig:stars}
\end{figure}

As indicated in the introduction we use the switching system  to construct the state transition diagram.
The boundaries of the cells are used in the definition of the set of vertices of the state transition diagram. 
The precise definition is as follows.

\begin{defn}
\label{defn:faces}
Let $\kappa = \kappa(i,j)\in\cK$.
The set of \emph{faces} of $\kappa$ is denoted by  $\cV(\kappa)$ and consists of
\begin{align*}
v_{*,\bar{j}} &:= \setof{\xi_*}\times (\eta_j,\eta_{j+1}), \quad *=i,i+1 \\
v_{\bar{i},*} &:= (\xi_i,\xi_{i+1})\times\setof{\eta_*}, \quad *=j,j+1 
\end{align*}
and, additionally, 
\[
w_\kappa   \quad\text{if $\kappa$ is of type $A$.}
\]
The complete set of faces is given by
\[
\cV := \bigcup_{\kappa\in\cK} \cV(\kappa)
\]
\end{defn}
We remark that elements of $\cV(\kappa)$ have two interpretations: faces of $\kappa$ and hence subsets of $[0,\infty)^2$, and 
vertices as elements of the state transition diagram.
Whether we are employing them as topological or combinatorial objects should be clear from the context.

\begin{defn}
Given a cell $\kappa = \kappa(i,j)$ we label the faces $\cV(\kappa)$ as follows
\begin{align*}
v_{i,\bar{j}}\ &\ \text{is an} \begin{cases} \text{ \emph{entrance} face} & \text{if $\Phi_1(\kappa) >\xi_i$} \\
\text{ \emph{absorbing} face} & \text{if $\Phi_1(\kappa) <\xi_i$} 
\end{cases} \\
v_{i+1,\bar{j}}\ &\ \text{is an} \begin{cases} \text{ \emph{entrance} face} & \text{if $\Phi_1(\kappa) <\xi_{i+1}$} \\
\text{ \emph{absorbing} face} & \text{if $\Phi_1(\kappa) >\xi_{i+1}$} 
\end{cases}\\
v_{\bar{i},j}\ &\ \text{is an} \begin{cases} \text{ \emph{entrance} face} & \text{if $\Phi_2(\kappa) >\eta_j$} \\
\text{ \emph{absorbing} face} & \text{if $\Phi_2(\kappa) <\eta_j$} 
\end{cases} \\
v_{\bar{i},j+1}\ &\ \text{is an} \begin{cases} \text{ \emph{entrance} face} & \text{if $\Phi_2(\kappa) <\eta_{i+1}$} \\
\text{ \emph{absorbing} face} & \text{if $\Phi_2(\kappa) >\eta_{i+1}$} 
\end{cases} \\
w_\kappa\ &\ \text{is an \emph{absorbing} face.} 
\end{align*}
\end{defn}

Let $\cV_e(\kappa)$ and $\cV_a(\kappa)$ denote the entrance and absorbing faces of $\cV(\kappa)$.
Observe that
\[
\cV_e(\kappa)\cap\cV_a(\kappa) =\emptyset.
\]

As the following proposition indicates, given a pair of adjacent cells that share a face $v$ there are constraints on the possible cell types.
Using the notation of Definition~\ref{defn:faces} we can assume $v= v_{i,\bar{j}} = \setof{\xi_i}\times (\eta_j,\eta_{j+1})$ or $v=v_{\bar{i},j}=(\xi_i,\xi_{i+1})\times \setof{\eta_j}$. 
In an abuse of notation we write 
\[
v_{i,\bar{j}}\in\Xi^n\quad\text{if}\quad  \xi_i\in\Xi^n
\]
and 
\[
v_{\bar{i},j} \in\sH^n\quad\text{if}\quad\eta_j\in\sH^n
\] 
for $n=1,2$.

\begin{prop}
\label{prop:adjacentCells}
The following figures show the possible types of two cells that share a face $v$ where it is indicated whether $v\in \sH^n$ or $v\in \Xi^n$ for $n=1,2$.
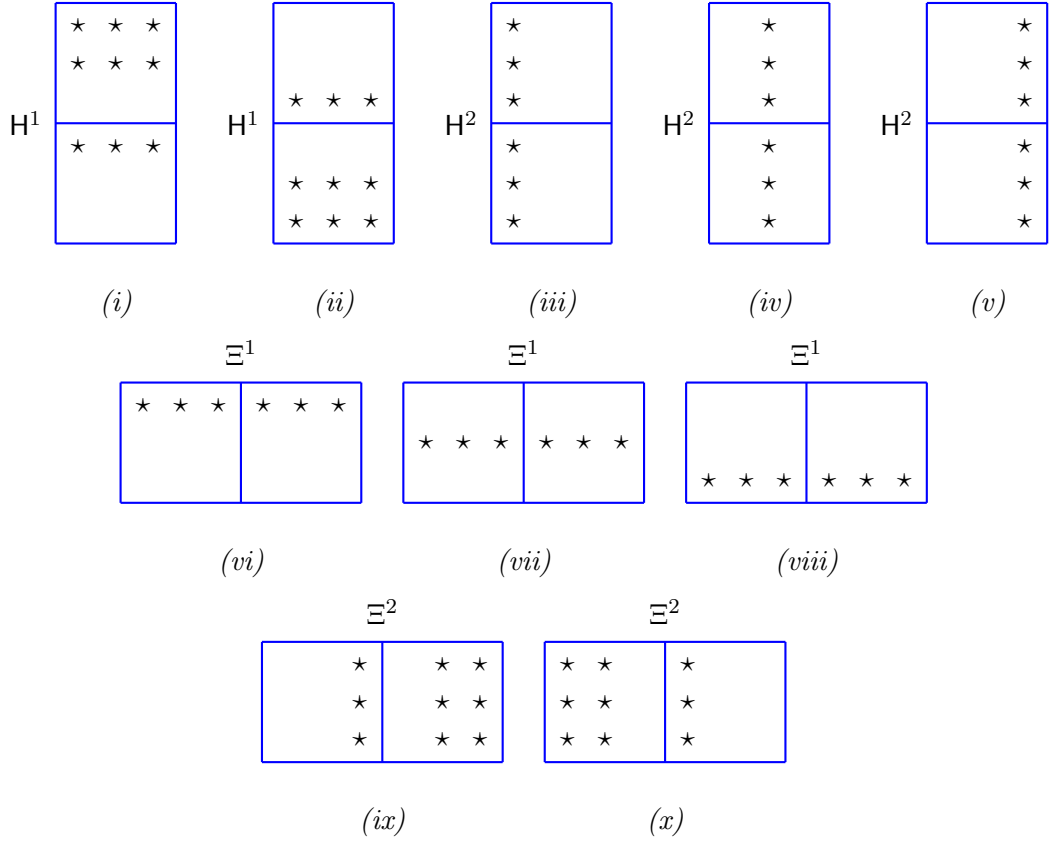
\begin{center}
\begin{tikzpicture}
[scale=0.4]

	\draw[blue, thick](0,0) -- (4,0);
	\draw[blue, thick](0,4) -- (4,4);
	\draw[blue, thick](0,0) -- (0,8);
	\draw[blue, thick](4,0) -- (4,8);
	\draw[blue, thick](0,8) -- (4,8);
	
	\draw(-1.0,4) node{$\sH^1$};

	\draw(0.75,7.25) node{$\star$};
	\draw(2,7.25) node{$\star$};	
	\draw(3.25,7.25) node{$\star$};
	\draw(0.75,6) node{$\star$};
	\draw(2,6) node{$\star$};
	\draw(3.25,6) node{$\star$};

	
	\draw(0.75,3.25) node{$\star$};
	\draw(2,3.25) node{$\star$};
	\draw(3.25,3.25) node{$\star$};
		
	\draw(2,-2) node{(i)};	
\end{tikzpicture}
\quad
\begin{tikzpicture}
[scale=0.4]

	\draw[blue, thick](0,0) -- (4,0);
	\draw[blue, thick](0,4) -- (4,4);
	\draw[blue, thick](0,0) -- (0,8);
	\draw[blue, thick](4,0) -- (4,8);
	\draw[blue, thick](0,8) -- (4,8);
	
	\draw(-1.0,4) node{$\sH^1$};

	\draw(0.75,4.75) node{$\star$};	
	\draw(2,4.75) node{$\star$};
	\draw(3.25,4.75) node{$\star$};
	
	
	\draw(0.75,2) node{$\star$};
	\draw(2,2) node{$\star$};
	\draw(3.25,2) node{$\star$};
	\draw(0.75,0.75) node{$\star$};
	\draw(2,0.75) node{$\star$};
	\draw(3.25,0.75) node{$\star$};
	
	\draw(2,-2) node{(ii)};	
\end{tikzpicture}
\quad
\begin{tikzpicture}
[scale=0.4]

	\draw[blue, thick](0,0) -- (4,0);
	\draw[blue, thick](0,4) -- (4,4);
	\draw[blue, thick](0,0) -- (0,8);
	\draw[blue, thick](4,0) -- (4,8);
	\draw[blue, thick](0,8) -- (4,8);
	
	\draw(-1.0,4) node{$\sH^2$};


	\draw(0.75,7.25) node{$\star$};
	\draw(0.75,6) node{$\star$};
	\draw(0.75,4.75) node{$\star$};	

	
	\draw(0.75,3.25) node{$\star$};
	\draw(0.75,2) node{$\star$};
	\draw(0.75,0.75) node{$\star$};
	
	\draw(2,-2) node{(iii)};	
\end{tikzpicture}
\quad
\begin{tikzpicture}
[scale=0.4]

	\draw[blue, thick](0,0) -- (4,0);
	\draw[blue, thick](0,4) -- (4,4);
	\draw[blue, thick](0,0) -- (0,8);
	\draw[blue, thick](4,0) -- (4,8);
	\draw[blue, thick](0,8) -- (4,8);
	
	\draw(-1.0,4) node{$\sH^2$};


	\draw(2,7.25) node{$\star$};	
	\draw(2,6) node{$\star$};
	\draw(2,4.75) node{$\star$};

	
	\draw(2,3.25) node{$\star$};
	\draw(2,2) node{$\star$};
	\draw(2,0.75) node{$\star$};
	
	\draw(2,-2) node{(iv)};	
\end{tikzpicture}
\quad
\begin{tikzpicture}
[scale=0.4]

	\draw[blue, thick](0,0) -- (4,0);
	\draw[blue, thick](0,4) -- (4,4);
	\draw[blue, thick](0,0) -- (0,8);
	\draw[blue, thick](4,0) -- (4,8);
	\draw[blue, thick](0,8) -- (4,8);
	
	\draw(-1.0,4) node{$\sH^2$};


	\draw(3.25,7.25) node{$\star$};
	\draw(3.25,6) node{$\star$};
	\draw(3.25,4.75) node{$\star$};

	
	\draw(3.25,3.25) node{$\star$};
	\draw(3.25,2) node{$\star$};
	\draw(3.25,0.75) node{$\star$};
	
	\draw(2,-2) node{(v)};	
\end{tikzpicture}
\\
\begin{tikzpicture}
[scale=0.4]

	\draw[blue, thick](0,0) -- (8,0);
	\draw[blue, thick](0,4) -- (8,4);
	\draw[blue, thick](0,0) -- (0,4);
	\draw[blue, thick](8,0) -- (8,4);
	\draw[blue, thick](4,0) -- (4,4);
	
	\draw(4,5) node{$\Xi^1$};

	
	\draw(0.75,3.25) node{$\star$};
	\draw(2,3.25) node{$\star$};
	\draw(3.25,3.25) node{$\star$};
	
	
	\draw(4.75,3.25) node{$\star$};
	\draw(6,3.25) node{$\star$};
	\draw(7.25,3.25) node{$\star$};
		
	\draw(4,-2) node{(vi)};	
\end{tikzpicture}
\quad
\begin{tikzpicture}
[scale=0.4]

	\draw[blue, thick](0,0) -- (8,0);
	\draw[blue, thick](0,4) -- (8,4);
	\draw[blue, thick](0,0) -- (0,4);
	\draw[blue, thick](8,0) -- (8,4);
	\draw[blue, thick](4,0) -- (4,4);
	
	\draw(4,5) node{$\Xi^1$};

	
	\draw(0.75,2) node{$\star$};
	\draw(2,2) node{$\star$};
	\draw(3.25,2) node{$\star$};
	
	
	\draw(4.75,2) node{$\star$};
	\draw(6,2) node{$\star$};
	\draw(7.25,2) node{$\star$};
		
	\draw(4,-2) node{(vii)};	
\end{tikzpicture}
\quad
\begin{tikzpicture}
[scale=0.4]

	\draw[blue, thick](0,0) -- (8,0);
	\draw[blue, thick](0,4) -- (8,4);
	\draw[blue, thick](0,0) -- (0,4);
	\draw[blue, thick](8,0) -- (8,4);
	\draw[blue, thick](4,0) -- (4,4);
	
	\draw(4,5) node{$\Xi^1$};

	
	\draw(0.75,0.75) node{$\star$};
	\draw(2,0.75) node{$\star$};
	\draw(3.25,0.75) node{$\star$};
	
	
	\draw(4.75,0.75) node{$\star$};
	\draw(6,0.75) node{$\star$};
	\draw(7.25,0.75) node{$\star$};
		
	\draw(4,-2) node{(viii)};	
\end{tikzpicture}
\\
\begin{tikzpicture}
[scale=0.4]

	\draw[blue, thick](0,0) -- (8,0);
	\draw[blue, thick](0,4) -- (8,4);
	\draw[blue, thick](0,0) -- (0,4);
	\draw[blue, thick](8,0) -- (8,4);
	\draw[blue, thick](4,0) -- (4,4);
	
	\draw(4,5) node{$\Xi^2$};

	
	\draw(3.25,3.25) node{$\star$};
	\draw(3.25,2) node{$\star$};
	\draw(3.25,0.75) node{$\star$};
	
	
	\draw(6,3.25) node{$\star$};
	\draw(7.25,3.25) node{$\star$};
	\draw(6,2) node{$\star$};
	\draw(7.25,2) node{$\star$};
	\draw(6,0.75) node{$\star$};
	\draw(7.25,0.75) node{$\star$};
		
	\draw(4,-2) node{(ix)};	
\end{tikzpicture}
\quad
\begin{tikzpicture}
[scale=0.4]

	\draw[blue, thick](0,0) -- (8,0);
	\draw[blue, thick](0,4) -- (8,4);
	\draw[blue, thick](0,0) -- (0,4);
	\draw[blue, thick](8,0) -- (8,4);
	\draw[blue, thick](4,0) -- (4,4);
	
	\draw(4,5) node{$\Xi^2$};

	
	\draw(0.75,3.25) node{$\star$};
	\draw(2,3.25) node{$\star$};
	\draw(0.75,2) node{$\star$};
	\draw(2,2) node{$\star$};
	\draw(0.75,0.75) node{$\star$};
	\draw(2,0.75) node{$\star$};
	
	
	\draw(4.75,3.25) node{$\star$};
	\draw(4.75,2) node{$\star$};
	\draw(4.75,0.75) node{$\star$};
		
	\draw(4,-2) node{(x)};	
\end{tikzpicture}
\end{center}
\end{prop}

\begin{proof}
The arguments for all the cases are essentially the same so we only provide explicit proofs in two cases.

(i) Without loss of generality assume the lower cell is $\kappa(i,j)$.  
Then the adjacent cell is $\kappa(i,j+1)$.
The stars indicate that $\kappa(i,j)$ is of type $NW$, $N$ or $NE$.  
This implies that $\Phi_2(\kappa(i,j)) > \eta_{j+1}$.  
By assumption $v\in H^1$. 
Thus $\Phi_2(\kappa(i,j+1) = \Phi_2(\kappa(i,j)) > \eta_{j+1}$, and therefore, $\kappa(i,j+1)$ is of type $NW$, $N$, $NE$, $W$, $A$, or $E$, as indicated.

(vii)  Without loss of generality assume the left cell is $\kappa(i,j)$.  
Then the adjacent cell is $\kappa(i+1,j)$.
The stars indicate that $\kappa(i,j)$ is of type $W$, $A$ or $E$.  
This implies that $\eta_j < \Phi_2(\kappa(i,j)) < \eta_{j+1}$.  
By assumption $v\in\Xi^1$. 
Thus $\Phi_2(\kappa(i+1,j) = \Phi_2(\kappa(i,j))$, and therefore, $\kappa(i+1,j)$ is of type $W$, $A$ or $E$.  
\end{proof}

\section{$\delta$-Constrained Continuous Switching Systems}
\label{sec:controlP}

As indicated in the introduction the goal of this section as follows: given a fixed switching system, $\Sigma = \Sigma(\Gamma,\Lambda,\Xi^1,\Xi^2,\sH^1,\sH^2)$, construct $f^{(\delta)}$  to define an associated $\delta$-constrained continuous switching system \eqref{eq:dconstrained}.

Starting with the observation that
\begin{equation}
\label{eq:cellWidth}
\lambda = \lambda(\Sigma) := \frac{1}{2} \min
\setof{ \min_{i=1,\ldots, I}\setof{\xi_{i}-\xi_{i-1}},\min_{j=1,\ldots, J } \setof{\eta_{j}-\eta_{j-1}} },
\end{equation}
is half of the minimal width of any cell in $\Sigma$ we choose
\[
0<\delta < \lambda.
\]

For each $\kappa = \kappa(i,j) = (\xi_i,\xi_{i+1})\times (\eta_j,\eta_{j+1})\in\cK$,  $i=1,\ldots,I$, $j=1,\ldots,J$ define 
\begin{equation}
\label{eq:G2}
\begin{aligned}
G^2(\kappa)=G^2(i,j)  & := [\xi_i + \delta, \xi_{i+1}-\delta]\times [\eta_j+\delta,\eta_{j+1}-\delta]\\
G^2(0,0)   &:= (0, \xi_{1}-\delta]\times (0,\eta_{1}-\delta] \\
G^2(i,0) &:= [\xi_i + \delta, \xi_{i+1}-\delta]\times (0,\eta_{1}-\delta] \\
G^2(0,j) &:= (0,\xi_1-\delta)\times  [\eta_j+\delta,\eta_{j+1}-\delta].
\end{aligned}
\end{equation}
For $i=1,\ldots, I$ and $j=1,\ldots, J$ define
\begin{equation}
\label{eq:G1}
\begin{aligned}
G^1(v_{i,\bar{j}}) = G^1(i,\bar{j}) &: = [\xi_i - \delta, \xi_{i}+\delta]\times [\eta_j+\delta,\eta_{j+1}-\delta] \\
G^1(v_{\bar{i},j}) = G^1(\bar{i},j) &: = [\xi_i + \delta, \xi_{i+1}-\delta]\times [\eta_j-\delta,\eta_{j}+\delta] \\
G^1(v_{i,\bar{0}}) = G^1(i,\bar{0}) &: = [\xi_i - \delta, \xi_{i}+\delta]\times (0,\eta_{1}-\delta] \\
G^1(v_{\bar{0},j}) = G^1(\bar{0},j) &: = (0, \xi_{1}-\delta]\times [\eta_j-\delta,\eta_{j}+\delta]
\end{aligned}
\end{equation}
and let
\begin{equation}
\label{eq:G0}
G^0(\pi)  :=  [\xi_i - \delta, \xi_{i}+\delta]\times [\eta_j-\delta,\eta_{j}+\delta]
\end{equation}
where $\pi \in\Pi$ as defined in \eqref{eq:gridpoints}.
We refer to these compact sets shown in Figure~\ref{fig:tiles} as \emph{tiles}, and more precisely $G^i$, $i=0,1,2$ is called an \emph{$i$-tile}.

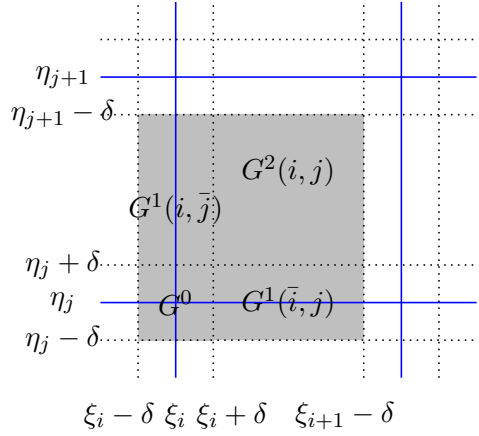
\begin{figure}
\begin{center}
\begin{tikzpicture}

	\fill[lightgray] (2.0,2.0) -- (4.0,2.0) -- (4.0,4.0) -- (2.0,4.0) -- (2.0,2.0);
	\fill[lightgray] (1.0,2.0) -- (2.0,2.0) -- (2.0,4.0) -- (1.0,4.0) -- (1.0,2.0);
	\fill[lightgray] (1.0,2.0) -- (2.0,2.0) -- (2.0,1.0) -- (1.0,1.0) -- (1.0,2.0);
	\fill[lightgray] (2.0,2.0) -- (4.0,4.0) -- (4.0,1.0) -- (2.0,1.0) -- (2.0,2.0);

	\draw[dotted, semithick] (0.5,5) -- (5.5,5);
	\draw[blue, semithick](0.5,4.5) -- (5.5,4.5);
	\draw[dotted, semithick] (0.5,4) -- (5.5,4);
	
	\draw[dotted, semithick] (0.5,2) -- (5.5,2);
	\draw[blue, semithick](0.5,1.5) -- (5.5,1.5);
	\draw[dotted, semithick] (0.5,1) -- (5.5,1);

	\draw[dotted, semithick] (1,0.5) -- (1,5.5);
	\draw[blue, semithick] (1.5,0.5) -- (1.5,5.5);
	\draw[dotted, semithick] (2,0.5) -- (2,5.5);
	
	\draw[dotted, semithick] (4,0.5) -- (4,5.5);
	\draw[blue, semithick] (4.5,0.5) -- (4.5,5.5);
	\draw[dotted, semithick] (5,0.5) -- (5,5.5);

	\draw (0.75,0) node {$\xi_i-\delta$};
	\draw (1.5,0) node {$\xi_i$};
	\draw (2.25,0) node {$\xi_i+\delta$};
	\draw (3.75,0) node {$\xi_{i+1}-\delta$};
	
	\draw (3,3.25) node {$G^2(i,j)$};
	\draw (1.5,2.75) node {$G^1(i,\bar{j})$};
	\draw (1.5,1.5) node {$G^0$};
	\draw (3,1.5) node {$G^1(\bar{i},j)$};
	
	\draw (0,4.5) node {$\eta_{j+1}$};
	\draw (0,4.0) node {$\eta_{j+1}-\delta$};
	\draw (0,2.0) node {$\eta_j+\delta$};
	\draw (0,1.5) node {$\eta_j$};
	\draw (0,1.0) node {$\eta_j-\delta$};

\end{tikzpicture}
\end{center}
\caption{The shaded regions indicate the four tiles $G^2(i,j)$, $G^1(i,\bar{j})$, $G^1(\bar{i},j)$, and $G^0=G^0(\pi)$ where $\pi = (\xi_i,\eta_j)$.}
\label{fig:tiles}
\end{figure}

We define the continuous nonlinearity $f^{(\delta)}\colon (0,\infty)^2 \to (0,\infty)^2$ in steps. 
Observe that $\Lambda$ is constant on any given  2-tile $G^2$, and hence, $\Lambda(G^2)$ is a unique well defined vector.
Thus we define $f^{(\delta)}\colon \cup_{i,j}G^2(i,j)\to (0,\infty)^2$ by
\begin{equation}
\label{eq:fG2}
f^{(\delta)}(x) := \Lambda(G^2(i,j))\quad x\in G^2(i,j).
\end{equation}

To define the action of $f^{(\delta)}$ on the 1-tiles we consider four cases.
\begin{itemize}
\item If $\xi_i\in \Xi^1$, then $\Lambda_1(\kappa(i-1,j) \neq \Lambda_1(\kappa(i,j))$ and $\Lambda_2(\kappa(i-1,j)) =\Lambda_2(\kappa(i,j))$.
Thus we define
\begin{equation}
\label{eq:f2G1X1}
f^{(\delta)}_2(x) := \Lambda_2(\kappa(i,j))\quad x\in  G^1(i,\bar{j}).
\end{equation}
and choose $f^{(\delta)}_1$ to be a continuous function on $G^1(i,\bar{j})$ which agrees with $f^{(\delta)}_1$ as defined by \eqref{eq:fG2} on $G^2(i-1,j)\cap G^1(i,\bar{j})$ and $G^1(i,\bar{j})\cap G^2(i,j)$ with the constraint that for $x\in G^1(i,\bar{j})$
\begin{equation}
\label{eq:f1G1X1}
\min\setof{\Lambda_1(\kappa(i-1,j)), \Lambda_1(\kappa(i,j))}\leq f^{(\delta)}_1(x) \leq 
\max\setof{\Lambda_1(\kappa(i-1,j)), \Lambda_1(\kappa(i,j))}.
\end{equation}

\item If $\xi_i\in \Xi^2$, then $\Lambda_1(\kappa(i-1,j) = \Lambda_1(\kappa(i,j))$ and $\Lambda_2(\kappa(i-1,j)) \neq\Lambda_2(\kappa(i,j))$.
Thus we define
\begin{equation}
\label{eq:f1G1X2}
f^{(\delta)}_1(x) := \Lambda_1(\kappa(i,j))\quad x\in  G^1(i,\bar{j}).
\end{equation}
and choose $f^{(\delta)}_2$ to be a continuous function on $G^1(i,\bar{j})$ which agrees with $f^{(\delta)}_2$ as defined by \eqref{eq:fG2} on $G^2(i-1,j)\cap G^1(i,\bar{j})$ and $G^1(i,\bar{j})\cap G^2(i,j)$ with the constraint that for $ x\in G^1(i,\bar{j})$
\begin{equation}
\label{eq:f2G1X2}
\min\setof{\Lambda_2(\kappa(i-1,j)), \Lambda_2(\kappa(i,j))}\leq f^{(\delta)}_2(x) \leq 
\max\setof{\Lambda_2(\kappa(i-1,j)), \Lambda_2(\kappa(i,j))}.
\end{equation}

\item If $\eta_j\in \sH^1$, then $\Lambda_1(\kappa(i,j-1) \neq \Lambda_1(\kappa(i,j))$ and $\Lambda_2(\kappa(i,j-1)) =\Lambda_2(\kappa(i,j))$.
Thus we define
\begin{equation}
\label{eq:f2G1H1}
f^{(\delta)}_2(x) := \Lambda_2(\kappa(i,j))\quad x\in G^1(\bar{i},j).
\end{equation}
and choose $f^{(\delta)}_1$ to a continuous function on $G^1(\bar{i},j)$ which agrees with $f^{(\delta)}_1$ as defined by \eqref{eq:fG2} on $G^2(i,j-1)\cap G^1(\bar{i},j)$ and $G^1(\bar{i},j)\cap G^2(i,j)$ with the constraint that for $x\in G^1(\bar{i},j)$
\begin{equation}
\label{eq:f1G1H1}
\min\setof{\Lambda_1(\kappa(i,j-1)), \Lambda_1(\kappa(i,j))}\leq f^{(\delta)}_1(x) \leq 
\max\setof{\Lambda_1(\kappa(i,j-1)), \Lambda_1(\kappa(i,j))}.
\end{equation}

\item If $\eta_j\in \sH^2$, then $\Lambda_1(\kappa(i,j-1) = \Lambda_1(\kappa(i,j))$ and $\Lambda_2(\kappa(i,j-1)) \neq\Lambda_2(\kappa(i,j))$.
Thus we define
\begin{equation}
\label{eq:f1G1H2}
f^{(\delta)}_1(x) := \Lambda_1(\kappa(i,j))\quad x\in  G^1(\bar{i},j).
\end{equation}
and choose $f^{(\delta)}_2$ to a continuous function on $G^1(\bar{i},j)$ which agrees with $f^{(\delta)}_2$ as defined by \eqref{eq:fG2} on $G^2(i,j-1)\cap G^1(\bar{i},j)$ and $G^1(\bar{i},j)\cap G^2(i,j)$ with the constraint that for $x\in G^1(\bar{i},j)$
\begin{equation}
\label{eq:f2G1H2}
\min\setof{\Lambda_2(\kappa(i,j-1)), \Lambda_2(\kappa(i,j))}\leq f^{(\delta)}_2(x) \leq 
\max\setof{\Lambda_2(\kappa(i,j-1)), \Lambda_2(\kappa(i,j))}.
\end{equation}
\end{itemize}

At this point we have defined $f^{(\delta)}\colon (0,\infty)^2 \setminus \left(\bigcup_{\pi\in\Pi} \Int\left(G^0(\pi)\right)\right)\to (0,\infty)^2$.
For each $\pi\in\Pi$ we extend $f^{(\delta)}$ to $G^0(\pi)$ continuously. 


\section{Cells, Chips, and Transversality}
\label{sec:chips}

As described in the introduction we construct the lattice of trapping regions using tiles and chips.
Tiles are introduced in the previous section.
In this section we define chips, which are closed right triangular regions, and prove results concerning the transversality of the vector field of a $\delta$-constrained switching system on the edges of tiles and the hypothenuse of chips.
The transversality results are used in Section~\ref{sec:proof} to verify the construction of trapping regions. 

A \emph{chip} is a closed right triangular subset of a 1-tile $G^1(v)$ and is uniquely identified by the 2-tile $G^2(\kappa)$ and the 0-tile $G^0(\pi)$ that it intersects.
With this in mind we denote a chip by
\[
C^n(\kappa, v, \pi)\quad\text{or}\quad C^w(\kappa, v, \pi)
\]
depending on whether it is a \emph{narrow chip} or \emph{wide chip}, respectively. 
If $v = v_{i,\bar{j}}$ or $v = v_{\bar{i},j}$, then the lengths of the edges of a narrow chip are $\delta$ and $(\eta_{j+1}-\eta_j)/2$ or $(\xi_{i+1} - \xi_i)/2$, respectively, while the lengths of the edges of a wide chip are $2\delta$ and $(\eta_{j+1}-\eta_j)/2$ or $(\xi_{i+1} - \xi_i)/2$, respectively.  Representative wide and narrow chips are shown in Figure~\ref{fig:chips}.

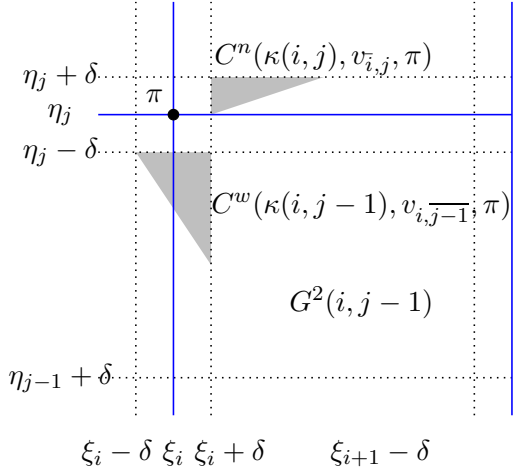
\begin{figure}
\begin{center}
\begin{tikzpicture}

	\fill[lightgray] (1.0,4.0) -- (2.0,4.0) -- (2.0,2.5) -- (1.0,4.0);
	\fill[lightgray] (2.0,5.0) -- (2.0,4.5) -- (3.5,5.0) -- (2.0,5.0);

	\draw[dotted, semithick] (0.5,5) -- (6,5);
	\draw[blue, semithick](0.5,4.5) -- (6,4.5);
	\draw[dotted, semithick] (0.5,4) -- (6,4);
	
	\draw[dotted, semithick] (0.5,1) -- (6,1);

	\draw[dotted, semithick] (1,0.5) -- (1,6);
	\draw[blue, semithick] (1.5,0.5) -- (1.5,6);
	\draw[dotted, semithick] (2,0.5) -- (2,6);
	
	\draw[dotted, semithick] (5.5,0.5) -- (5.5,6);
	\draw[blue, semithick] (6,0.5) -- (6,6);

	\draw (0.75,0) node {$\xi_i-\delta$};
	\draw (1.5,0) node {$\xi_i$};
	\draw (2.25,0) node {$\xi_i+\delta$};
	\draw (4.25,0) node {$\xi_{i+1}-\delta$};
	
	\draw (4,2.0) node {$G^2(i,j-1)$};
	\draw (4,3.25) node {$C^w(\kappa(i,j-1),v_{i,\overline{j-1}},\pi)$};
	\draw (3.5,5.25) node {$C^n(\kappa(i,j),v_{\bar{i},j},\pi)$};
	\draw (1.25,4.75) node {$\pi$};
	\filldraw(1.5,4.5) circle (2pt);

	\draw (0,5.0) node {$\eta_{j}+\delta$};
	\draw (0,4.5) node {$\eta_{j}$};
	\draw (0,4.0) node {$\eta_{j}-\delta$};
	\draw (0,1.0) node {$\eta_{j-1}+\delta$};

\end{tikzpicture}
\end{center}
\caption{The shaded triangular regions indicate the wide chip $C^w(\kappa(i,j-1),v_{i,\overline{j-1}},\pi)$ and the narrow chip $C^n(\kappa(i,j),v_{\bar{i},j},\pi)$.  
Observe that $C^n(\kappa(i,j),v_{\bar{i},j},\pi)\subset G^1(v_{\bar{i},j})$ and its edges lie in $G^2(\kappa(i,j))$ and $G^0(\pi)$ where $\pi = (\xi_i,\eta_j)$.}
\label{fig:chips}
\end{figure}
 
As indicated above, we use the tiles and chips to construct trapping regions for $\delta$-constrained switching systems. 
For this we need to know that the vector field associated with the $\delta$-constrained switching system are transverse to the appropriate edges of the tiles and chips.
We catalogue this information in the following sequence of propositions.

For the remainder of this section $\Sigma = \Sigma(\Gamma,\Lambda,\Xi^1,\Xi^2,\sH^1,\sH^2)$ is assumed to be a fixed switching system.
Choose $\delta >0$ such that
\begin{equation}
\label{eq:delta}
\delta < \delta_*(\lambda,\mu,\rho,\bar{\gamma}):=
\min\left\{\frac{\lambda\mu\bar\gamma}{\sqrt{2}(2\lambda+3\rho)},\ 
\sqrt{\frac{\lambda\mu\bar\gamma}{32}}\right\},
\end{equation}
where $\mu = \mu(\Sigma)$, $\lambda = \lambda(\Sigma)$, $\rho = \rho(\Sigma)$, and $\bar{\gamma} = \bar{\gamma}(\Sigma)$ are  defined as in \eqref{eq:muSigma}, \eqref{eq:cellWidth}, \eqref{eq:PointDisplacement}, and \eqref{eq:gammaRatio}, respectively.
Note that the above condition implies $\delta<\lambda$, as $\bar{\gamma}\leq 1$ and $\mu\leq\rho$.

\begin{prop}
\label{prop:TG2}
Consider $\kappa \in\cK$. If $v\in\cV_e(\kappa)$, then the vector field \eqref{eq:switching} is transverse in to $G^2(\kappa)$ along the associated edge.
\end{prop}

\begin{proof}
Let $\kappa = \kappa(i,j)$. Assume $v = v(i,\bar{j})$.  Since $v\in\cV_e(\kappa)$, $\kappa$ is of type $N$, $NE$, $A$, $E$, $S$, or $SE$ and $\Phi_1(\kappa) \geq \xi_i + \mu > \xi_i + \delta$. Thus the vector field \eqref{eq:switching} is transverse in to $G^2(\kappa)$ along the edge $\setof{\xi_i+\delta}\times [\eta_j+\delta, \eta_{j+1}-\delta]$.
The arguments for the remain faces are similar.
\end{proof}

Similar arguments lead to the following propositions.

\begin{prop}
\label{prop:TG1j}
Consider $G^1(v_{i,\bar{j}})$.
\begin{enumerate}
\item[(i)]  If $v_{\bar{i},j}\in \cV_e(\kappa(i,j))$, then the vector field \eqref{eq:switching} is transverse in to $G^1(v_{i,\bar{j}})$ along the edge $[\xi_i,\xi_i + \delta]\times \setof{\eta_j + \delta}$.
\item[(ii)]  If $v_{\overline{i-1},j}\in \cV_e(\kappa(i-1,j))$, then the vector field \eqref{eq:switching} is transverse in to $G^1(v_{i,\bar{j}})$ along the edge $[\xi_i -  \delta,\xi_i]\times \setof{\eta_j + \delta}$.
\item[(iii)]  If $v_{\bar{i},j+1}\in \cV_e(\kappa(i,j))$, then the vector field \eqref{eq:switching} is transverse in to $G^1(v_{i,\bar{j}})$ along the edge $[\xi_i,\xi_i + \delta]\times \setof{\eta_{j+1} - \delta}$.
\item[(iv)]  If $v_{\overline{i-1},j+1}\in \cV_e(\kappa(i-1,j))$, then the vector field \eqref{eq:switching} is transverse in to $G^1(v_{i,\bar{j}})$ along the edge $[\xi_i - \delta,\xi_i ]\times \setof{\eta_{j+1} - \delta}$.
\end{enumerate}
\end{prop}

\begin{prop}
\label{prop:TG1i}
Consider $G^1(v_{\bar{i},j})$.
\begin{enumerate}
\item[(i)]  If $v_{i,\bar{j}}\in \cV_e(\kappa(i,j))$, then the vector field \eqref{eq:switching} is transverse in to $G^1(v_{\bar{i},j})$ along the edge $\setof{\xi_i + \delta}\times [\eta_j ,\eta_j+ \delta]$.
\item[(ii)]  If $v_{i,\overline{j-1}}\in \cV_e(\kappa(i,j-1))$, then the vector field \eqref{eq:switching} is transverse in to $G^1(v_{\bar{i},j})$ along the edge $\setof{\xi_i + \delta}\times [\eta_j - \delta ,\eta_j ]$.
\item[(iii)]  If $v_{i+1,\bar{j}}\in \cV_e(\kappa(i,j))$, then the vector field \eqref{eq:switching} is transverse in to $G^1(v_{\bar{i},j})$ along the edge $\setof{\xi_{i+1} - \delta}\times [\eta_j ,\eta_j+ \delta]$.
\item[(iv)]  If $v_{i+1,\overline{j-1}}\in \cV_e(\kappa(i,j-1))$, then the vector field \eqref{eq:switching} is transverse in to $G^1(v_{\bar{i},j})$ along the edge $\setof{\xi_{i+1} - \delta}\times [\eta_j -  \delta ,\eta_j]$.
\end{enumerate}
\end{prop}

\begin{prop}
\label{prop:TG0}
Consider $G^0(i,j)$.
\begin{enumerate}
\item[(i)]  If $\kappa(i,j)$ is of type $NW$, $W$, or $SW$, then the vector field \eqref{eq:switching} is transverse in to $G^0(i,j)$ along the edge $\setof{\xi_i + \delta}\times [\eta_j ,\eta_j+ \delta]$.
\item[(ii)]  If $\kappa(i,j)$ is of type $SW$, $S$, or $SE$, then  the vector field \eqref{eq:switching} is transverse in to $G^0(i,j)$ along the edge $[\xi_i,\xi_i + \delta]\times \setof{\eta_j + \delta}$.
\item[(iii)]  If $\kappa(i-1,j)$ is of type $SW$, $S$, or $SE$, then the vector field \eqref{eq:switching} is transverse in to $G^0(i,j)$ along the edge $[\xi_i-\delta,\xi_i]\times \setof{\eta_j+ \delta}$.
\item[(iv)]  If $\kappa(i-1,j)$ is of type $NE$, $E$, or $SE$, then the vector field \eqref{eq:switching} is transverse in to $G^0(i,j)$ along the edge $\setof{\xi_i - \delta}\times [\eta_j ,\eta_j+ \delta]$.
\item[(v)]  If $\kappa(i-1,j-1)$ is of type $NE$, $E$, or $SE$, then the vector field \eqref{eq:switching} is transverse in to $G^0(i,j)$ along the edge $\setof{\xi_i - \delta}\times [\eta_j -\delta,\eta_j]$.
\item[(vi)]  If $\kappa(i-1,j-1)$ is of type $NW$, $N$, or $NE$, then the vector field \eqref{eq:switching} is transverse in to $G^0(i,j)$ along the edge $[\xi_i-\delta,\xi_i]\times \setof{\eta_j- \delta}$.
\item[(vii)]  If $\kappa(i,j-1)$ is of type $NW$, $N$, or $NE$, then the vector field \eqref{eq:switching} is transverse in to $G^0(i,j)$ along the edge $[\xi_i,\xi_i+\delta]\times \setof{\eta_j- \delta}$.
\item[(viii)]  If $\kappa(i,j-1)$ is of type $NW$, $W$, or $SW$, then the vector field \eqref{eq:switching} is transverse in to $G^0(i,j)$ along the edge $\setof{\xi_i + \delta}\times [\eta_j - \delta ,\eta_j]$.
\end{enumerate}
\end{prop}

We now give a proposition which shows the transversality on the hypotenuse of a chip.  
Consider a $\delta$-constrained continuous switching system \eqref{eq:dconstrained}  associated with 
a switching system $\Sigma = \Sigma(\Gamma,\Lambda,\Xi^1,\Xi^2,\sH^1,\sH^2)$, and suppose 
a narrow chip $C^n(\kappa, v, \pi)$ or a wide chip $C^w(\kappa, v, \pi)$ is introduced for $\Sigma$.
Recall that $\delta>0$ is chosen to satisfy the condition \eqref{eq:delta}.   

Without loss of generality, one can assume, by applying rotation if necessary, that a chip $C$ appears associated with $\kappa=\kappa(i,j),\ v=v_{\bar{i},j},\ \pi=(\xi_i,\eta_j)$.  
Let $H$ be the hypotenuse of the chip $C$.  
Under these circumstance, the transversality results can be formulated as follows:

\begin{prop}
\label{prop:TC}
\begin{enumerate}
\item[(i)] If $C=C^n(\kappa,v,\pi)$ for $\kappa=\kappa(i,j),\ v=v_{\bar{i},j},\ \pi=(\xi_i,\eta_j)$,
$\kappa=\kappa(i,j)$ is of type $A$, $N$, $NE$, $E$, and $\kappa'=\kappa(i,j-1)$ is of type $NW$, 
then the vector field \eqref{eq:dconstrained} is transverse in to $C$ along the hypotenuse $H$.
\item[(ii)] If $C=C^w(\kappa,v,\pi)$ for $\kappa=\kappa(i,j),\ v=v_{\bar{i},j},\ \pi=(\xi_i,\eta_j)$,
$\kappa=\kappa(i,j)$ is of type $A$, $N$, $NE$, or $E$, and $\kappa'=\kappa(i,j-1)$ is of type $N$
or $NE$, then the vector field \eqref{eq:dconstrained} is transverse in to $C$ along the hypotenuse $H$.
\end{enumerate}
\end{prop}

To prove the above proposition, recall that the vector field \eqref{eq:dconstrained} is rewritten as
\[
\dot{x} = -\Gamma x + f^{(\delta)}(x) = -\Gamma(x - \Phi^{(\delta)}(x)).
\]
From the above definition of $f^{(\delta)}(x)$, we see that $\Phi^{(\delta)}(x)=\Phi(\kappa)$ for any $x\in G^2(\kappa)$ with $\kappa=\kappa(i,j)$ or $\kappa(i,j-1)$.  
We first show that the transversality of the vector field \eqref{eq:dconstrained} on $H$ is reduced to the transversality of an affine vector field.  
Let $p(t)=(1-t)H_0+tH_1\ (0\leq t\leq 1)$ where $H_0$ and $H_1$ are the end points of $H$, and let $\theta(t)$ is the angle (measured counter-clockwise from the positive direction of the $x_1$-axis) of the vector 
$V(x)=-\Gamma x + f^{(\delta)}(x)$ evaluated at $p(t)$. 
In other words,
\[
V(p(t))=R(t)\cdot\begin{pmatrix}\cos\theta(t) \\ \sin\theta(t)\end{pmatrix}\quad (R(t)>0).
\]
Observe that the types of $\kappa, \kappa'$ and the definition of $f^{(\delta)}(x)$ on 
$G^1(v)$ imply that the angle $\theta(t)$ must satisfy $0<\theta(t)<\pi$ for any $t\in[0,1]$.

Let $\nu=r\cdot\begin{pmatrix}\cos\varphi \\ \sin\varphi\end{pmatrix}$ be a normal vector to
$H$ pointing upward, hence $\pi/2 < \varphi < \pi$.  
The vector field \eqref{eq:dconstrained} is
transverse in to the chip $C$ along $H$ if and only if the inner product $V(p(t))\cdot\nu$ 
at any point of $p(t)\in H\ (t\in[0,1]$ is positive.   Let $\theta_{\rm min}$ be the minimum of 
$\theta(t)$, and let $t_{\rm min}\in [0,1]$ be such that $\theta(t_{\rm min})=\theta_{\rm min}$.

\begin{lem}
\label{lem:TC}
Let $V_{\rm min}=V(p(t_{\rm min}))$.
If $V_{\rm min}\cdot\nu>0$, then $V(p(t))\cdot\nu>0$ for any $t\in [0,1]$.
\end{lem}

\begin{proof}
Observe that
\[
V(p(t))\cdot\nu = R(t)\cdot\begin{pmatrix}\cos\theta(t) \\ \sin\theta(t)\end{pmatrix}\cdot
r\cdot\begin{pmatrix}\cos\varphi \\ \sin\varphi\end{pmatrix} = R(t)r\cos(\theta(t)-\varphi),
\]
where $-\pi<\theta_{\rm min}-\varphi\leq \theta(t)-\varphi<\pi/2$.  If $0\leq \theta(t)-\varphi<\pi/2$,
then we immediately obtain the positivity of $V(p(t))\cdot\nu$.  So we consider the case 
$-\pi<\theta_{\rm min}-\varphi\leq \theta(t)-\varphi<0$.  In this range, the cosine function is 
monotone increasing, and hence $\cos(\theta_{\rm min}-\varphi)\leq \cos(\theta(t)-\varphi)$.
Therefore, if $V_{\rm min}\cdot\nu>0$, then we have $0<\cos(\theta_{\rm min}-\varphi)\leq 
\cos(\theta(t)-\varphi)$, and therefore $V(p(t))\cdot\nu>0$.
\end{proof}

Now define $\Phi^*=-\Gamma^{-1}V_{\rm min}$ and consider the affine vector field $V^*(x)=
-\Gamma(x-\Phi^*)$.  Observe that, from the definition of $f^{(\delta)}(x)$ on $G^1(v)$, 
$\Phi^*=\begin{pmatrix} \Phi_1^* \\ \Phi_2^*\end{pmatrix}$ satisfies
\[
|\Phi_1^* - \xi|<\rho, \quad |\Phi_2^*-\eta|<\rho
\]
for any $\xi\in\Xi,\ \eta\in\sH$.  Similarly,  it also satisfies
\[
\mu < |\Phi_2^*-\eta|
\]
for $\eta=\eta_j$.

In order to prove $V_{\rm min}\cdot\nu>0$, we first consider the narrow chip $C=C^n(\kappa,v,\pi)$.

In this case, we can use coordinate system with origin at $(\xi_i, \eta_j)$ to evaluate
\[
p(t)=(1-t)\begin{pmatrix}\delta \\ 0\end{pmatrix}+t\begin{pmatrix}a/2 \\ \delta\end{pmatrix}
=\begin{pmatrix}\delta+(a/2-\delta)t \\ \delta\end{pmatrix}
\]
and one can choose $\nu=\begin{pmatrix}-\delta \\ a/2 - \delta\end{pmatrix}$, where 
$a =\xi_{i+1} - \xi_i(>2\lambda)$.  Therefore the inner product $V^*(p(t))\cdot\nu$ defines a 
function $T(t,\delta)$ which is affine in $t\in[0,1]$ and quadratic in $\delta$.   More explicitly, 
we can obtain $T(t,\delta)=K(t)\delta^2+L(t)\delta+M(t)$, where
\begin{eqnarray}
K(t)&=& (1-t)\gamma_1 + t\gamma_2\\
L(t)&=& (\gamma_1-\gamma_2)(a/2)t-\gamma_1(\Phi_1^*-\xi_i)-\gamma_2(\Phi_2^*-\eta_j)\\
M(t)&=& (a/2)\gamma_2(\Phi_2^*-\eta_j)
\end{eqnarray}
In fact, $M$ does not depend on $t$, $M(t)=M$, and is strictly positive.  Also $K(t)$ is strictly positive
because $\gamma_1, \gamma_2>0$ and $t\in[0,1]$.  So the quadratic function $T(t,\delta)=K(t)\delta^2+L(t)\delta+M$
is always positive for $\delta>0$ if the discriminant $D=D(t)=L(t)^2-4K(t)M$ is negative, or if $D\geq 0$ 
and $L>0$.  So the only case where $T(t,\delta)$ can take a negative value is the case $D\geq 0$ and 
$L<0$.  In this case, we can easily see that $T(t,\delta)>0$ for any $\delta\in (0,\delta_-(t))$, where 
$\delta_-(t)$ is a smaller root of the quadratic function $T(t,\delta)=0$, namely 
\[
\delta_-(t) = \frac{-L(t)-\sqrt{D(t)}}{2K(t)}.
\]
In order to obtain an estimate which does not depend on $t$, we observe that $\delta_-(t)>-M/L(t)$ in 
case $L(t)<0$.  This easily follows by comparing the graph of $T(t,\delta)$ and its tangent line at 
$\delta=0$.  Substituting above and using the estimates \eqref{eq:muSigma}, 
\eqref{eq:PointDisplacement}, \eqref{eq:gammaRatio}, \eqref{eq:cellWidth}, we obtain
\begin{eqnarray*}
-\frac{M}{L(t)}&=&\frac{(a/2)\gamma_2(\Phi_2^*-\eta_j)}
{-(\gamma_1-\gamma_2)(a/2)t+\gamma_1(\Phi_1^*-\xi_i)+\gamma_2(\Phi_2^*-\eta_j)}\\
&\geq&\frac{(a/2)\gamma_2(\Phi_2^*-\eta_j)}
{|\gamma_1-\gamma_2|(a/2)+\gamma_1(\Phi_1^*-\xi_i)+\gamma_2(\Phi_2^*-\eta_j)}\\
&=&\frac{|\Phi_2^*-\eta_j|}
{|(\gamma_1/\gamma_2)-1|+(2/a)\{(\gamma_1/\gamma_2)|\Phi_1^*-\xi_i|+|\Phi_2^*-\eta_j|\}}\\
&\geq&\frac{\mu}{|(1/\bar\gamma)-1|+(1/\lambda)\{(1/\bar\gamma)\rho+\rho\}}\\
&=&\frac{\lambda\mu\bar\gamma}{(1-\bar\gamma)\lambda+(1+\bar\gamma)\rho}\\
&\geq&\frac{\lambda\mu\bar\gamma}{\lambda+2\rho}
\end{eqnarray*}
Therefore, in the case of the narrow chip, if we choose $\delta$ so that 
$0<\delta<\lambda\mu\bar\gamma/(\lambda+2\rho)$, then we can conclude that the vector 
field $V^*(x)$, and hence $V^{(\delta)}(x)$ as well, is transverse in to the chip along its hypotenuse.

For the wide chip, we can argue in the same manner, but the estimates become more complicated.
Using 
\[
p(t)=(1-t)\begin{pmatrix}\delta \\ -\delta\end{pmatrix}+t\begin{pmatrix}a/2 \\ \delta\end{pmatrix}
=\begin{pmatrix}\delta+(a/2-\delta)t \\ (2t-1)\delta\end{pmatrix}
\]
and $\nu=\begin{pmatrix}-2\delta \\ a/2 - \delta\end{pmatrix}$, one can similarly define
\[
T(t,\delta):=V^*(p(t))\cdot\nu=K(t)\delta^2+L(t)\delta+M
\]
where
\begin{eqnarray}
K(t)&=& 2\gamma_1-\gamma_2 - 2(\gamma_1-\gamma_2)t,\\
L(t)&=& (\gamma_1-\gamma_2)at +(a/2)\gamma_2 -2\gamma_1(\Phi_1^*-\xi_i)
-\gamma_2(\Phi_2^*-\eta_j),\\
M(t)&=& (a/2)\gamma_2(\Phi_2^*-\eta_j).
\end{eqnarray}
As in the case of narrow chip, $M$ does not depend on $t$ and is positive.  However, $K(t)$ can change
its sign.  In case $K(t)>0$, the same argument works, and we obtain that $-M/L(t)$ gives a bound for 
$\delta$ in case $L(t)<0$, otherwise $T(t,\delta)>0$ for any $\delta>0$.  

Observe that $K(t)$ vanishes at some $t_0\in [0,1]$ only when $\gamma_2\geq 2\gamma_1$.  
In case $K(t_0)=0$,  $T(t_0,\delta)=0$ when $\delta=-M/L(t_0)$ and 
\[
L(t_0)=(\gamma_1-\gamma_2)at_0 +(a/2)\gamma_2 -2\gamma_1(\Phi_1^*-\xi_i)
-\gamma_2(\Phi_2^*-\eta_j)=(\gamma_1-\gamma_2)a -2\gamma_1(\Phi_1^*-\xi_i)
-\gamma_2(\Phi_2^*-\eta_j)<0, 
\]
so the upper bound of $\delta$ for $T(t_0,\delta)>0$ is $\delta<-M/L(t_0)$ which is the same as 
before, or more precisely,
\begin{eqnarray*}
-\frac{M}{L(t_0)}&=&\frac{(a/2)\gamma_2(\Phi_2^*-\eta_j)}
{(\gamma_2-\gamma_1)a+2\gamma_1(\Phi_1^*-\xi_i)+\gamma_2(\Phi_2^*-\eta_j)}\\
&=&\frac{|\Phi_2^*-\eta_j|}
{2|1-(\gamma_1/\gamma_2)|+(2/a)\{2(\gamma_1/\gamma_2)|\Phi_1^*-\xi_i|+|\Phi_2^*-\eta_j|\}}\\
&>&\frac{\lambda\mu\bar\gamma}{2(1-\bar\gamma)\lambda+(2+\bar\gamma)\rho}\\
&\geq&\frac{\lambda\mu\bar\gamma}{2\lambda+3\rho}.
\end{eqnarray*}

In case $K(t)<0$, which occurs when $\gamma_2>2\gamma_1$ and $0<t<t_0$, 	we have $T(t,\delta)>0$
if $0<\delta<\delta_+(t)=(-L(t)-\sqrt{D(t)})/2K(t)$, where $\delta_+(t)$ is a larger root of the quadratic
equation $T(t,\delta)=0$.  Using an inequality
\[
\sqrt{x+h}\geq \sqrt{x}+\frac12\frac{h}{\sqrt{x+h}}\quad (x>0, h>0)
\]
we have, independent of whether $L(t)>0$ or not,
\[
\delta_+(t)\geq\frac{L(t)+|L(t)|+\frac12\frac{4(-K(t))M}{\sqrt{L(t)^2+4(-K(t))M}}}{2(-K(t))}
\geq\frac{M}{\sqrt{L(t)^2+4(-K(t))M}}.
\]
Since $L(t)^2+4(-K(t))M<\max\{2L(t)^2, 2\times 4(-K(t))M\}$, we finally have an estimate
\[
\delta_+(t)\geq \min\left\{\frac{M}{\sqrt{2}|L(t)|},\ \frac{1}{2\sqrt{2}}\sqrt{\frac{M}{-K(t)}}\right\}
\]
The first term in the min can be treated exactly the same way, up to the constant $1/\sqrt{2}$.  The 
second term can be treated as follows:
\[
\frac{M}{-K(t)}>\frac{(a/2)\gamma_2(\Phi_2^*-\eta_j)}
{|2\gamma_1-\gamma_2| + 2|\gamma_1-\gamma_2|}
=\frac{(a/2)(\Phi_2^*-\eta_j)}{|2(\gamma_1/\gamma_2)-1|+2|(\gamma_1/\gamma_2)-1|}
>\frac{\lambda\mu\bar\gamma}{(2-\bar\gamma)+2(1-\bar\gamma)}
>\frac{\lambda\mu\bar\gamma}{4}
\]

Putting all the above estimates together, for both the narrow chip case and the wide chip case, we 
define
\begin{equation}
\label{eq:mindelta}
\delta_*=\delta_*(\lambda,\mu,\rho,\bar{\gamma}):=
\min\left\{\frac{\lambda\mu\bar\gamma}{\sqrt{2}(2\lambda+3\rho)},\ 
\sqrt{\frac{\lambda\mu\bar\gamma}{32}}\right\}.
\end{equation}
The above argument proves that, for any $\delta$ with $0<\delta<\delta_*$, the vector field 
$V^*(x)$, and hence $V^{(\delta)}(x)$ as well, is transverse in to the (both narrow and wide) chip 
along its hypotenuse.  This completes the proof of Proposition \ref{prop:TC}.

\section{State Transition Diagram}
\label{sec:transitionDiagram}

\begin{defn}
Given a cell $\kappa$ the directed \emph{$\kappa$-graph}, $\cF_\kappa\colon \cV_e(\kappa)\mvmap \cV_a(\kappa)$, is defined by
\[
v \in \cF(u)\quad\text{if and only if}\quad v\in\cV_a(\kappa)\ \text{and}\ u\in \cV_e(\kappa).
\]
In addition, if $w_\kappa\in\cV(\kappa)$, then $\cF_\kappa(w_\kappa) = w_\kappa$.
\end{defn}

To give some geometric perspective to these definitions the reader is referred to Figure~\ref{fig:CoarseFkappaMap} where the directed graphs $\cF_\kappa$ are shown for various types of cells $\kappa$. 
The following proposition, whose proof is left to the reader, relates the definitions of entrance and absorbing faces to the dynamics of 
 the associate $\kappa$-equation.

\begin{prop}
Let $v\in\cV_e(\kappa)$. For every $x\in v$, there exists a unique $t = t(x) \geq 0$ and a $v'\in\cV_a(\kappa)$ such that
\[
\psi_\kappa(t,x)\in\cl(v')\quad\text{or}\quad \lim_{t\to \infty}\psi_\kappa(t,v) = \Phi(\kappa).
\]
In the latter case, $\kappa$ is an attracting cell and $\cF_\kappa(v) = w$.
\end{prop}

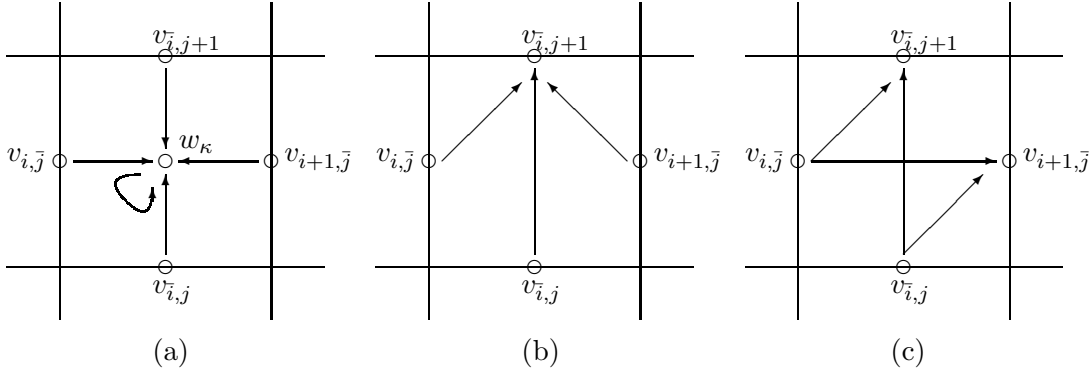
\begin{figure}[t]
\begin{picture}(125,140)(0,-20)
\multiput(0,20)(0,80){2}{\line(1,0){120}}
\multiput(20,0)(80,0){2}{\line(0,1){120}}
\multiput(60,20)(0,80){2}{\circle{5}}
\multiput(20,60)(80,0){2}{\circle{5}}
\put(60,60){\circle{5}}
\put(25,60){\vector(1,0){30}}
\put(95,60){\vector(-1,0){30}}
\put(60,25){\vector(0,1){30}}
\put(60,95){\vector(0,-1){30}}
\put(65,65){$w_{\kappa}$}
\put(0,60){$v_{i,\bar{j}}$}
\put(105,60){$v_{i+1,\bar{j}}$}
\put(55,10){$v_{\bar{i},j}$}
\put(55,105){$v_{\bar{i},j+1}$}

\qbezier(50,55)(35,55)(45,45)
\qbezier(55,50)(55,35)(45,45)
\put(55,45){\vector(0,1){5}}

\put(55,-15){(a)}
\end{picture}
\quad
\begin{picture}(125,140)(0,-20)
\multiput(0,20)(0,80){2}{\line(1,0){120}}
\multiput(20,0)(80,0){2}{\line(0,1){120}}
\multiput(60,20)(0,80){2}{\circle{5}}
\multiput(20,60)(80,0){2}{\circle{5}}
\put(25,60){\vector(1,1){30}}
\put(60,25){\vector(0,1){70}}
\put(95,60){\vector(-1,1){30}}
\put(0,60){$v_{i,\bar{j}}$}
\put(105,60){$v_{i+1,\bar{j}}$}
\put(55,10){$v_{\bar{i},j}$}
\put(55,105){$v_{\bar{i},j+1}$}

\put(55,-15){(b)}
\end{picture}
\quad
\begin{picture}(125,140)(0,-20)
\multiput(0,20)(0,80){2}{\line(1,0){120}}
\multiput(20,0)(80,0){2}{\line(0,1){120}}
\multiput(60,20)(0,80){2}{\circle{5}}
\multiput(20,60)(80,0){2}{\circle{5}}
\put(25,60){\vector(1,0){70}}
\put(25,60){\vector(1,1){30}}
\put(60,25){\vector(0,1){70}}
\put(60,25){\vector(1,1){30}}
\put(0,60){$v_{i,\bar{j}}$}
\put(105,60){$v_{i+1,\bar{j}}$}
\put(55,10){$v_{\bar{i},j}$}
\put(55,105){$v_{\bar{i},j+1}$}

\put(55,-15){(c)}
\end{picture}
\caption{The directed graphs or multivalued maps $\cF_\kappa\colon \cV_e(\kappa)\mvmap \cV_a(\kappa)$ for $\kappa = \kappa(i,j)$.
(a) $\kappa$ is attracting and hence of type $A$. (b) $\kappa$ is focusing of type $N$. (c) $\kappa$ is translating cell of type $NE$. }
\label{fig:CoarseFkappaMap}
\end{figure}

\begin{defn}
\label{defn:stateTransitionGraph}
Given a switching system $\Sigma$ the associated state transition diagram $\cF \colon \cV\mvmap \cV$ is the directed graph with vertices 
\[
\cV :=\left ( \bigcup_{\kappa\in\cK} \cV(\kappa)\right) / \sim
\]
where $\sim$ is equivalence relation $v \sim w$ when $v \in \cV(\kappa)$, $w \in \cV(\kappa')$ and $\kappa \cap \kappa'$ intersect in a 1-dimensional interval that correspond to both $v$ and 
$w$.  
\end{defn}

We label the possible configurations of $\cF_\kappa$ and $\cF_{\kappa'}$ under the assumption that $\kappa$ and $\kappa'$ share a common face. 

\begin{defn}
Consider the state transition diagram $\cF\colon\cV\mvmap\cV$  for a switching system.
Let $v\in\cV$.  
If $v=w_{i,j}\in\cV(\kappa(i,j))$, then $v$ is called an \emph{minimal vertex}.  
Assume $v\in \cV(\kappa)\cap \cV(\kappa')$.
If there exists $u\in \cV(\kappa)$ and $w\in \cV(\kappa')$ such that $u\to v \to w$, then $v$ is a \emph{transparent vertex}.
If there exists $u\in \cV(\kappa)$ and $w\in \cV(\kappa')$ such that $u\to v$ and $w\to v$, then $v$ is a \emph{black vertex}.
If there exists $u\in \cV(\kappa)$ and $w\in \cV(\kappa')$ such that $v\to u$ and $v\to w$, then $v$ is a \emph{white vertex}.
The sets of minimal, transparent, black, and white vertices are denoted by $\cM$, $\cT$, $\cB$, and $\cW$, respectively.
\end{defn}

\begin{prop}
\label{prop:MTBWpart}
Consider a state transition diagram $\cF\colon\cV\mvmap\cV$  for a switching system.
The minimal, transparent, black, and white vertices partition $\cV$.
\end{prop}

\begin{proof}
We need to demonstrate that $\cM \cup \cT \cup \cB \cup \cW = \cV$ and that $\cM$, $\cT$, $\cB$, and $\cW$ are mutually disjoint.
By definition $\cM \cap (\cT \cup \cB \cup \cW) = \emptyset$.
Thus we can restrict our attention to vertices that are associated with faces of cells.

Observe that from Figure~\ref{fig:CoarseFkappaMap} if $v\in \cV(\kappa)$, then there is either an edge to $v$ or an edge from $v$ in $\cF_\kappa$, but not both.  
The first statement follows from the existence of the edges. 
The second statement follows from the fact that it is not possible to have both types of edges within one cell.
\end{proof}

For the remainder of the paper we make the following assumption
\begin{description}
\item[$\not$B]
The state transition diagram $\cF\colon\cV\mvmap\cV$  does not contain a black vertex.
\end{description}

We hasten to add that this is not an unreasonable assumption.
To understand why, consider the following result that is easily checked.
\begin{lem}
\label{lem:trivial}
If $v(i,\bar{j})$ is a black vertex, then
\[
\Lambda_1(\kappa(i-1,j)) > \gamma_1\xi_i > \Lambda_1(\kappa(i,j).
\]
\end{lem}
Observe that starting with a regulatory network this implies that $x_1$ has a self-edge that corresponds to repression.
For many biological applications this type of self-regulation is better modeled by two nodes $x_1$ and $y$ where $x_1$ activates $y$ and $y$ represses $x_1$.

From a mathematical perspective, Lemma~\ref{lem:trivial} can be used to show that given a system of the form \eqref{eq:abstract} it is possible to approximate the nonlinear functions $f_1$ using piecewise constant functions in such a way that the state transition graph for resulting switching system does not contain black walls. 
This approach is discussed in \cite{edwards2015} and in future work \cite{CFG}. 

We conclude this section with three results concerning the structure of forward invariant sets under $\cF$.

\begin{prop}
\label{prop:uvw}
Consider $\cN\in \sInvset^+(\cV,\cF)$  under assumption {\bf $\not$B}.
If $v\in \cN^0$, then $v\in \cM\cup \cT$. 
Furthermore, if $v\in \cT$, then there exist distinct $u\in \cN$ and $w\in\cN^0$ such that $u\to v\to w$.
\end{prop}

\begin{proof}
By Proposition~\ref{prop:MTBWpart} and  assumption {\bf $\not$B}, it is sufficient to show that $v\not\in\cW$.
If $v\in\cW$, then $\cF^{-1}(v) = \emptyset$, hence $v\not\in\cN^0$, a contradiction.

Since $v\in\cN^0$, by definition there  exists $u\in \cF^{-1}(v)\cap \cN$.
If $v\in \cT$, then $v$ does not have a self-edge hence $u\neq v$.
By definition, if $v\in \cT$, then there exists $w\in\cV$ such that $v\to w$. 
This implies that $w\in\cF(\cN)$.
By definition of a forward invariant set $\cF(\cN) \subset \cN$ and hence $w\in\cN$.
Since $v\to w$, $w\in\cN^0$.  
\end{proof}

\begin{prop}
\label{prop:uvwCell}
Let $\cN\in\sInvset^+(\cF)$.  
If $\cV(\kappa)\cap \cN^0 \cap \cT\neq \emptyset$, then $\cV_e(\kappa) \cap \cN \neq  \emptyset$ and $\cV_a(\kappa) \cap \cN\subset \cN^0$.
\end{prop}

\begin{proof}
Let $v\in \cV(\kappa)\cap \cN^0 \cap \cT$.  
By Proposition~\ref{prop:uvw} there exist exist distinct $u\in \cN$ and $w\in\cN^0$ such that $u\to v\to w$.
If $v\in \cV_e(\kappa)$, then $\cV_a(\kappa)= \cF_\kappa(v)$ and hence $\cV_a(\kappa)\subset \cN^0$.
If $v\in \cV_a(\kappa)$, then $u\in\cN\cap \cV_e(\kappa)$ and $\cF_\kappa(u) = \cV_a(\kappa)$.
Therefore, $\cV_a(\kappa)\subset \cN^0$.
\end{proof}

\begin{prop}
\label{prop:cellAtt_Rev}
Consider $\cN\in \sInvset^+(\cF)$ under assumption {\bf $\not$B}.
If $v\in \cN^0\cap \cV(\kappa)$, then 
\[
\cV_a(\kappa)\subset \cN^0.
\]
Furthermore, if $\cN^0\cap \cV(\kappa)\neq \setof{w_\kappa}$, then $\cV_e(\kappa)\cap \cN \neq\emptyset$.
\end{prop}

\begin{proof}   
First, assume $v\in\cV_e(\kappa)$. Then, $\cV_a(\kappa)\subset \cF_\kappa(v)$, and hence, by Proposition~\ref{prop:vinA} $\cV_a(\kappa)\subset \cA$.

Now assume $v\in\cV_a(\kappa)$. 
To indicate the line of reasoning we provide the proofs of two of the three cases.
Assume $\kappa$  is an attracting cell. 
This implies that  $\cV_a(\kappa)=\setof{w_\kappa} = \setof{v}\subset \cV_a(\kappa)$. 
Assume $\kappa$  is a translating cell.  
By Proposition~\ref{prop:uvw} there exists $u\in \cN$ and $w\in\cN^0$ such that $u\to v\to w$.
Since $v\in\cV_a(\kappa)$, by {\bf $\not$B} $u\in\cV(\kappa)$.
Furthermore, $u\to v$ implies that $u\in \cV_e(\kappa)$.
Thus by Proposition~\ref{prop:vinA} $\cF_\kappa(u) =\cV_a(\kappa) \subset \cN^0$.
\end{proof}

\begin{figure}
\begin{picture}(140,140)
\multiput(0,20)(0,80){2}{\line(1,0){120}}
\multiput(20,0)(80,0){2}{\line(0,1){120}}

\put(65,25){\vector(1,1){30}}
\put(60,25){\vector(0,1){70}}

\put(60,20){\circle{5}}
\put(60,100){\circle*{5}}
\put(20,60){\circle*{5}}
\put(100,60){\circle*{5}}

{\thicklines
\put(25,65){\vector(1,1){30}}
\put(25,60){\vector(1,0){70}}
}
\end{picture}
\begin{picture}(140,140)
\multiput(0,20)(0,80){2}{\line(1,0){120}}
\multiput(20,0)(80,0){2}{\line(0,1){120}}
\put(25,65){\vector(1,1){30}}
\put(25,60){\vector(1,0){70}}

\put(60,20){\circle*{5}}
\put(60,100){\circle*{5}}
\put(20,60){\circle{5}}
\put(100,60){\circle*{5}}

{\thicklines
\put(65,25){\vector(1,1){30}}
\put(60,25){\vector(0,1){70}}
}
\end{picture}
\begin{picture}(140,140)
\multiput(0,20)(0,80){2}{\line(1,0){120}}
\multiput(20,0)(80,0){2}{\line(0,1){120}}

\put(60,20){\circle*{5}}
\put(60,100){\circle*{5}}
\put(20,60){\circle*{5}}
\put(100,60){\circle*{5}}

{\thicklines
\put(25,65){\vector(1,1){30}}
\put(65,25){\vector(1,1){30}}
\put(60,25){\vector(0,1){70}}
\put(25,60){\vector(1,0){70}}
}
\end{picture}
\caption{Possible intersections of  $\cN^0$ with a translation cell.  Vertices of $\cN^0$ are indicated as solid circles.}
\label{fig:AcapT_Rev}
\end{figure}
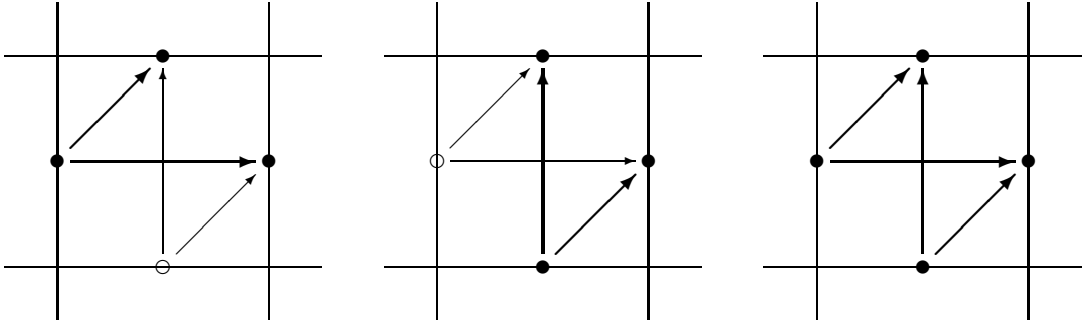

\section{Constructing Trapping Regions for $\delta$-Continuous Switching Systems}
\label{sec:proof}

Let $\cN\in\sInvset^+(\cF)$.  
Our goal is to identify a compact set $N_\cN\subset (0,\infty)^2$ such that $N_\cN$ is a trapping region for
\begin{equation}
\label{eq:threshPerturbed}
\dot{x} = -\Gamma x + \Lambda^{(\delta)}(x).
\end{equation}
The region $N_\cN$ is constructed using the above defined tiles and chips.
Furthermore, as is made clear below the identification of $N_\cN$ is done locally. 

\begin{defn}
For each $(i,j)$, $i = 0,\ldots, I$, $j = 0,\ldots, J$ the associated \emph{elementary domain} is defined to be 
\[
E(i,j) := \left(\xi_{i-1} +\delta,\xi_{i+1}-\delta \right) \times \left(\eta_{j-1}+\delta,\eta_{j+1}-\delta\right).
\]
The set of associated vertices is denoted by $\cE(i,j)\subset \cV$ and is defined to be the union of vertices for which the  associated face $v$ satisfies $v\cap E(i,j)\neq \emptyset$ or a face $w_\kappa$ where $\kappa\cap E(i,j)\neq \emptyset$.
We use $\bar{\cE}(i,j)$ to denote the union of vertices in $\cE(i,j)$ and the vertices associated with the faces of the cells in 
$\left[ \xi_{i-1},\xi_{i+1} \right] \times \left[\eta_{j-1},\eta_{j+1} \right]$.
\end{defn}

We leave it to the reader to check that
\begin{equation}
\label{eq:domainCovers}
(\delta,\infty)^2 = \bigcup_{\begin{array}{c}i = 0,\ldots, I \\ j = 0,\ldots, J \end{array}} E(i,j)
\end{equation}

Fix an elementary domain  $E(i,j)$. 
In the applying the following rules, unless otherwise specified, it is assumed that the cells $\kappa$, the faces $v$, and $\pi = (\xi_i,\eta_j)\in\Pi$ all intersect $E(i,j)$ nontrivially.
\begin{description}
\item[Rule 0]
If $w_{\kappa}\in \cN^0\cap\cM\cap \cE(i,j)$, then
\[
G^2(\kappa)  \subset N_\cN.
\]
\item[Rule 1]
Let $u\in \cN^0\cap \cT \cap \cE(i,j)$ and let $r,s\in\bar{\cE}(i,j)$ such that $r\to u\to s$. 
Consider the grid elements $\kappa_r$ and $\kappa_s$ such that $r,u\in \cV(\kappa_r)$ and  $u,s\in\cV(\kappa_s)$.  
Then
\[
G^2(\kappa_r) \cup G^2(\kappa_s) \cup G^1(u) \subset N_\cN.
\]
\item[Rule 2]
Assume 
\[
G^2(\kappa_\alpha) \cup G^2(\kappa_\beta) \cup G^2(\kappa_\gamma)\cup G^1(v_{\alpha \beta}) \cup G^1(v_{\beta \gamma}) \subset N_\cN,
\] 
where $v_{\alpha \beta} = \kappa_\alpha\cap \kappa_\beta$ and $v_{\beta \gamma} = \kappa_\beta \cap \kappa_\gamma$.
If $v_{\alpha \beta}\in \cV_a(\kappa_\beta)$ or $v_{\beta \gamma}\in \cV_a(\kappa_\beta)$, then
\[
G^0(\pi) \subset N_\cN.
\]
\item[Rule 3] 
Assume $G^2(\kappa) \cup G^1(v_\alpha)\cup G^0(\pi) \subset N_\cN$ where $v_\alpha\in \cV(\kappa) \cap \cE(i,j)$.
If $\cV(\kappa)\cap \cE(i,j) \subset \cV_e(\kappa)$, then 
\[
C^n(\kappa, v_\beta,\pi)\subset N_\cN
\] 
where $v_\beta\in \cV(\kappa)\cap \cE(i,j)$.
\item[Rule 4] 
Assume $G^0(\pi) \subset N_\cN$ and $G^2(\kappa_\alpha) \not\subset N_\cN$.
Let $\setof{v_{\alpha\beta},v_{\alpha\gamma}} = \cV(\kappa_\alpha)\cap \cE(i,j)$.
If $v_{\alpha\beta} \not\in\cV_a(\kappa)$, then 
\[
C^w(\kappa_\gamma,v_{\alpha\gamma},\pi)\subset N_\cN
\]
where $v_{\alpha\gamma} = \kappa_\alpha \cap \kappa_\gamma$.
\item[Rule 5]  Assume $G^2(\kappa_\alpha)\cup G^2(\kappa_\beta)\cup G^0(\pi)\subset N_\cN$ and 
$v_{\alpha\beta} = \kappa_\alpha\cap \kappa_\beta$.
If $\cV(\kappa_\alpha)\cap \cE(i,j) \subset \cV_e(\kappa_\alpha)$ and $\cV(\kappa_\beta)\cap \cE(i,j) \subset \cV_e(\kappa_\beta)$, then
\[
G^1(v_{\alpha\beta})\subset N_\cN.
\]
\end{description}

\begin{defn}
\label{defn:latticeMorphism}
For each $\cN\in\sInvset^+(\cF)$ define  $N_\cN\subset (0,\infty)^2$ to be the union of the minimal collection of tiles and chips that satisfies {\bf Rules 0} - {\bf 5} over all elementary domains $E(i,j)$, $i=0,\ldots I$, $j=0,\ldots, J$.
\end{defn}

For future reference we highlight the following remark.
\begin{lem}
\label{lem:trivialN}
If $\cN =\emptyset$, then $N_\cN = N_\emptyset = \emptyset$.
\end{lem}

\begin{prop}
\label{prop:trapN}
If $\cN\in \sInvset^+(\cF)$, then $N_\cN\in \sANbhd(\phi)$.
\end{prop}

The goal for the remainder of this section is to prove Proposition~\ref{prop:trapN}.
We will do this by proving that for each $\cN\in\sAtt(\cF)$, $N_\cN$ is a trapping region.
This in turn is done by considering all possible forms of intersection of $N_\cN$ with all possible  elementary domains $E(i,j)$ and checking for transversality as we proceed. 
For this we make use of the following objects.

\begin{defn}
Let $e$ denote a boundary edge of a 2-tile $G^2(\kappa)$, a 1-tile $G^1(v)$, or a 0-tile $G^0(i,j)$.
We say that $e$ is \emph{interior} to $E(i,j)$ if 
\[
e\cap E(i,j) \neq \emptyset
\] 
and \emph{exterior} to $E(i,j)$ if $e$ is not interior to $E(i,j)$, but 
\[
e\subset \cl(E(i,j)).
\]
A boundary edge $e$ of a chip is \emph{interior} to $E(i,j)$ if  $e\cap E(i,j)\neq \emptyset$.
\end{defn}
As the following lemma indicates it is sufficient show that for every $E(i,j)$  along any boundary edge of $N$ that is interior to $E(i,j)$ the vector field of \eqref{eq:threshPerturbed} is transverse in with respect to  $N$.

\begin{lem}
If $N$ is a union of tiles and chips and for every elementary domain $E(i,j)$, $i=1,\ldots, I$, $j=1,\ldots, J$, the interior edges of $N$ with respect to $E(i,j)$ are transverse in, then $N$ is a trapping region.
\end{lem}

\begin{proof}
By \eqref{eq:domainCovers} the collection of elementary domains covers the phase space $(0,\infty)^2$.
By definition $\delta <\lambda$ which is half the minimal cell width (see \ref{eq:cellWidth}) and thus every point $x\in (0,\infty)^2$ lies in the interior of some $E(i,j)$.
Thus if the interior edges of $N$ with respect to $E(i,j)$ are transverse in for all elementary domains, then every boundary point of $N$ is transverse in.
Therefore $N$ is a trapping region.
\end{proof}

To simplify the application of the rules we use the symmetry of the elementary domain.
By Proposition~\ref{prop:uvw}, if $\cN\in\sInvset(\cF)$, then $\cN^0\subset \cM\cup\cT$.
Assume that there exists $v\in \cN^0\cap \cT$.
This implies that  $v$ represents the intersection of two cells $\kappa_1$ and $\kappa_2$.
Since $v\in\cT$, there exists $u,w\in\cV$ such that $u\to v\to w$.
Without loss of generality we assume that $u\in\cV(\kappa_2)$ and $w\in \cV(\kappa_1)$.
Symmetry, i.e.\ a reflection or rotation, allows us without loss of generality to assume that $\kappa_2 = \kappa(i,j)$ and  $\kappa_1=\kappa(i+1,j)$ and thus we will refer to $v$ as a transparent face moving east. 
We will study the implications of {\bf Rules 0} - {\bf 5} on the elementary domain $E(i,j)$ shown in  Figure~\ref{fig:MinimalFourCells}
 where the neighboring cells of interest are $\kappa_3 = \kappa(i,j-1)$ and $\kappa_4=\kappa(i+1,j-1)$.

We begin by making simple observations concerning {\bf Rules 0} - {\bf 5}. 
If $\cN^0\cap\cT\cap \cE(i,j)\neq\emptyset$, then {\bf Rule 1}  implies that 
\begin{equation}
\label{eq:vTbasic}
G^2(\kappa_1) \cup G^2(\kappa_2) \cup G^1(v) \subset N_\cN
\end{equation}
as indicated in  Figure~\ref{fig:MinimalFourCells}(b).

{\bf Rule 2} determines if $G^0$ tiles belong to $N_\cN$. 
Therefore, a necessary condition for  $G^0(i,j)\subset N_\cN$ is that at least three of the four $G^2$ and at least two of the four $G^1$ tiles in the elementary domain $E(i,j)$ belong to $N_\cN$.

Observe that {\bf Rule 4} is only applicable if given an elementary domain $E$ exactly three of the four associated $G^2$ tiles belong to $N_\cN$.
Furthermore, if {\bf Rule 3} implies that $C^n(\kappa, v_\beta,(i,j))\subset N_\cN$ and {\bf Rule 4} implies that $C^w(\kappa, v_\beta,(i,j))\subset N_\cN$, then we can ignore {\bf Rule 3} since 
\[
C^n(\kappa, v_\beta,(i,j))\subset C^w(\kappa, v_\beta,(i,j))\subset N_\cN.
\]

Observe that if {\bf Rule 5} applies then, {\bf Rule 3} applies to $\kappa_1$ and $\kappa_4$. However
\[
C^n(\kappa_n, v_1,\pi)\subset G^1(v_1), \quad n=1,4
\]
and hence we can ignore {\bf Rule 3}.

\begin{prop}
\label{prop:looseG^1}
Let $\setof{v_{\alpha\beta}} = \cV(\kappa_\alpha)\cap\cV(\kappa_\beta)$ where $\kappa_n\cap E(i,j)\neq\emptyset$,
$n=\alpha,\beta$.
If $G^2(\kappa_\alpha)\not\subset N_\cN$, then $G^1(v_{\alpha\beta})\not\subset N_\cN$.
\end{prop}

\begin{proof}
Only {\bf Rule 1} and {\bf Rule 5} require that $G^1(v_{\alpha\beta})\subset N_\cN$.
Both these rules are based on  $G^2(\kappa_\alpha)\subset N_\cN$.
Thus the minimality of $N_\cN$ implies that $G^1(v_{\alpha\beta})\not\subset N_\cN$.
\end{proof}

\begin{prop}
\label{prop:N0}
Let $\cN_0,\cN_1\in \sInvset^+(\cF)$. If $\cN_0^0=\cN_1^0$, then 
\[
N_{\cN_0} = N_{\cN_1}.
\]
\end{prop}

\begin{proof}
The only rules that directly depend on the elements of $\cN$ are {\bf Rule 0} and {\bf Rule 1}.
These rules are given in terms of $\cN^0$.
\end{proof}

\begin{prop}
\label{prop:minimalN}
If $\cN = \setof{\kappa_k\mid k = 1,\ldots, K}\subset \cM$, then $\cN = \cN^0 \in \sAtt(\cF)$,
\[
N_\cN = \bigcup_{k = 1,\ldots, K} G^2(\kappa_k) ,
\]
and $N_\cN$ is a trapping region.
\end{prop}

\begin{proof}
By {\bf Rule 0}
\[
 \bigcup_{k = 1,\ldots, K} G^2(\kappa_k) \subset N_\cN.
\]
We now show that this is the minimal collection of tiles and chips that satisfy {\bf Rules 0} - {\bf 5}.
Given that $\cN\cap\cT =\emptyset$, the only way to require the existence of a $G^1$ tile is through {\bf Rule 5}.
However, {\bf Rule 5} requires the existence of $G^0(i,j)$.
The requirement for the existence of $G^0(i,j)$ follows from {\bf Rule 2}, which in turn requires the existence of a $G^1$ tile.
Thus,
\begin{equation}
\label{eq:Aminimal}
 N_\cN=  \bigcup_{k = 1,\ldots, K} G^2(\kappa_k).
\end{equation}

For each $\kappa_k$, $k = 1,\ldots, K$, each face is identified with a vertex that belongs to $\cV_e(\kappa_k)$.
Thus, by Proposition~\ref{prop:TG2}, $G^2(\kappa_k)$ is a trapping region and hence $N_\cN$ is a trapping region.
\end{proof}

The following result is an immediate application of Propositions~\ref{prop:N0} and \ref{prop:minimalN}.

\begin{cor}
Let $\cN\in\sInvset^+(\cF)$.
If $\cN^0 = \setof{\kappa_k\mid k = 1,\ldots, K}\subset \cM$, then
\[
N_\cN = \bigcup_{k = 1,\ldots, K} G^2(\kappa_k) ,
\]
and $N_\cN$ is a trapping region.
\end{cor}

For the remainder of the argument we assume that $\cN^0\not\subset \cM$ or equivalently that $\cN^0\cap \cT\neq\emptyset$.
Furthermore, taking advantage of the above described symmetry we always assume that $v = v_{i,\bar{j}}\in \cN^0\cap \cT$ is transparent east.

\begin{prop}
Under the assumption that $v= v_{i,\bar{j}} \in \cN^0\cap \cT$ the cells that intersect $E(i,j)$ must be of the types indicated in Figure~\ref{fig:MinimalFourCells}(a) and as indicated in Figure~\ref{fig:MinimalFourCells}(b) it must be the case that
\[
G^2(\kappa_1)\cup G^2(\kappa_2)\cup G^1(v)\subset N_\cN.
\]
\end{prop}

\begin{proof}
The assumption that $v = v_{i,\bar{j}}$ is transparent east implies that $\Phi_1(\kappa_1)>\xi_i$ and $\Phi_1(\kappa_2)>\xi_i$.
Thus, $\kappa_1$ is of type $N$, $A$, $S$, $NE$, $E$, or $SE$ and $\kappa_2$ is of type $NE$, $E$, or $SE$.

{\bf Rule 1} implies that $G^2(\kappa_1)\cup G^2(\kappa_2)\cup G^1(v)\subset N_\cN$.
\end{proof}

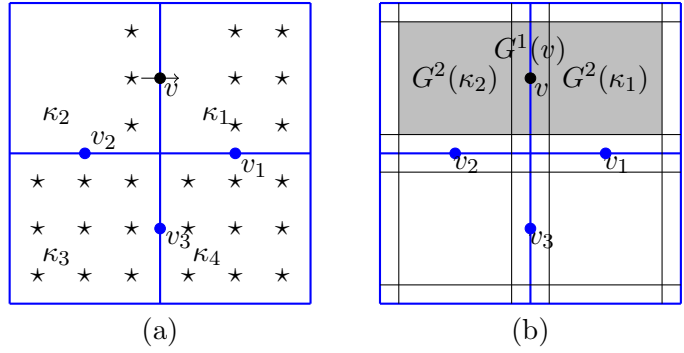
\begin{figure}
\begin{center}
\begin{tikzpicture}
[scale=0.5]

	\draw[blue, thick](0,0) -- (8,0);
	\draw[blue, thick](0,4) -- (8,4);
	\draw[blue, thick](0,0) -- (0,8);
	\draw[blue, thick](8,0) -- (8,8);
	\draw[blue, thick](4,0) -- (4,8);
	\draw[blue, thick](0,8) -- (8,8);

	\draw (4.30,5.70) node {$v$};		
	\filldraw(4,6) circle (4pt);
	\draw[->] (3.5,6) -- (4.5,6);
	\filldraw[blue](2,4) circle (4pt);
	\draw (2.50,4.35) node {$v_2$};
	\filldraw[blue](6,4) circle (4pt);
	\draw (6.50,3.60) node {$v_1$};
	\filldraw[blue](4,2) circle (4pt);
	\draw (4.5,1.70) node {$v_3$};
	
	\draw(5.5,5) node{$\kappa_1$};
	
	\draw(6,7.25) node{$\star$};
	\draw(6,6) node{$\star$};
	\draw(6,4.75) node{$\star$};
	\draw(7.25,7.25) node{$\star$};
	\draw(7.25,6) node{$\star$};
	\draw(7.25,4.75) node{$\star$};
	
	\draw(1.25,5) node{$\kappa_2$};
	
	\draw(3.25,7.25) node{$\star$};
	\draw(3.25,6) node{$\star$};
	\draw(3.25,4.75) node{$\star$};

	\draw(1.25,1.25) node{$\kappa_3$};
	
	\draw(0.75,3.25) node{$\star$};
	\draw(2,3.25) node{$\star$};
	\draw(3.25,3.25) node{$\star$};
	\draw(0.75,2) node{$\star$};
	\draw(2,2) node{$\star$};
	\draw(3.25,2) node{$\star$};
	\draw(0.75,0.75) node{$\star$};
	\draw(2,0.75) node{$\star$};
	\draw(3.25,0.75) node{$\star$};
	
	\draw(5.25,1.25) node{$\kappa_4$};
	
	\draw(4.75,3.25) node{$\star$};
	\draw(6,3.25) node{$\star$};
	\draw(7.25,3.25) node{$\star$};
	\draw(4.75,2) node{$\star$};
	\draw(6,2) node{$\star$};
	\draw(7.25,2) node{$\star$};
	\draw(4.75,0.75) node{$\star$};
	\draw(6,0.75) node{$\star$};
	\draw(7.25,0.75) node{$\star$};

	\draw(4,-0.75) node{(a)};	
\end{tikzpicture}
\qquad
\begin{tikzpicture}
[scale=0.5]

	\fill[lightgray] (0.5,4.5) -- (3.5,4.5) -- (3.5,7.5) -- (0.5,7.5) -- (0.5,4.5);
	\fill[lightgray] (4.5,4.5) -- (7.5,4.5) -- (7.5,7.5) -- (4.5,7.5) -- (4.5,4.5);
	\fill[lightgray] (3.5,4.5) -- (4.5,4.5) -- (4.5,7.5) -- (3.5,7.5) -- (3.5,4.5);
		
	\draw[blue, thick](0,0) -- (8,0);
	\draw[blue, thick](0,4) -- (8,4);
	\draw[blue, thick](0,0) -- (0,8);
	\draw[blue, thick](8,0) -- (8,8);
	\draw[blue, thick](4,0) -- (4,8);
	\draw[blue, thick](0,8) -- (8,8);
		
	\draw (0,0.5) -- (8,0.5);
	\draw (0,3.5) -- (8,3.5);
	\draw (0,4.5) -- (8,4.5);
	\draw (0,7.5) -- (8,7.5);
	\draw (0.5,0) -- (0.5,8);
	\draw (3.5,0) -- (3.5,8);
	\draw (4.5,0) -- (4.5,8);
	\draw (7.5,0) -- (7.5,8);
	
	
		
	\filldraw(4,6) circle (4pt);
	\draw (4.30,5.70) node {$v$};
	\filldraw[blue](2,4) circle (4pt);
	\draw (2.30,3.70) node {$v_2$};

	\draw (6.30,3.70) node {$v_1$};
	\filldraw[blue](6,4) circle (4pt);
	\draw (4.30,1.70) node {$v_3$};
	\filldraw[blue](4,2) circle (4pt);
		
	\draw(2,6) node{$G^2(\kappa_2)$};
	\draw(6,6) node{$G^2(\kappa_1)$};
	\draw(4,6.75) node{$G^1(v)$};
	
	\draw(4,-0.75) node{(b)};	
\end{tikzpicture}
\end{center}
\caption{Black filled dots indicate vertex that belongs to $\cN^0$. Blue filled dot indicates that vertex may or may not belong to $\cN^0$. (a) The vertex $v$ is a transparent face moving east.  
The associated four cells are $\kappa_i$, $i=1,\ldots, 4$. 
The possible cell types of $\kappa_1$ and $\kappa_2$ are indicated by the stars, e.g.\ center is $A$, upper right corner is $NE$, middle rights is $E$. 
Without further assumptions there are no restrictions on the cell types of $\kappa_3$ and $\kappa_4$.
(b) Shaded region indicates tiles that belong to $N_\cN$ under the assumption that $\cN^0\cap\cT\neq \emptyset$.
}
\label{fig:MinimalFourCells}
\end{figure}

\begin{proof}[Proof of Proposition~\ref{prop:trapN}]

\begin{enumerate}
\item {\em Assume $v_1\in\cN^0$ and $v_1$ is transparent north}.  
	By {\bf Rule 1}, $G^2(\kappa_4)\cup G^1(v_1) \subset N_\cN$.
	The possible cell types are indicated in Figure~\ref{fig:fourCellsNew}(a).  
	
\begin{figure}
\begin{tikzpicture}
[scale=0.5]

	\draw[blue, thick](0,0) -- (8,0);
	\draw[blue, thick](0,4) -- (8,4);
	\draw[blue, thick](0,0) -- (0,8);
	\draw[blue, thick](8,0) -- (8,8);
	\draw[blue, thick](4,0) -- (4,8);
	\draw[blue, thick](0,8) -- (8,8);

	\draw (4.30,5.70) node {$v$};		
	\filldraw(4,6) circle (4pt);
	\draw[->] (3.5,6) -- (4.5,6);
	\filldraw[blue](2,4) circle (4pt);
	\draw (2.50,4.35) node {$v_2$};
	
	\draw (6.50,3.60) node {$v_1$};
	\filldraw(6,4) circle (4pt);
	\draw[->] (6,3.5) -- (6,4.5);
	\filldraw[blue](4,2) circle (4pt);
	\draw (4.5,1.70) node {$v_3$};
	
	\draw(5.5,5) node{$\kappa_1$};
	
	\draw(6,7.25) node{$\star$};
	\draw(7.25,7.25) node{$\star$};
	\draw(6,6) node{$\star$};
	\draw(7.25,6) node{$\star$};
	
	\draw(1.25,5) node{$\kappa_2$};
	
	\draw(3.25,7.25) node{$\star$};
	\draw(3.25,6) node{$\star$};
	\draw(3.25,4.75) node{$\star$};

	\draw(1.25,1.25) node{$\kappa_3$};
	
	\draw(0.75,3.25) node{$\star$};
	\draw(2,3.25) node{$\star$};
	\draw(3.25,3.25) node{$\star$};
	\draw(0.75,2) node{$\star$};
	\draw(2,2) node{$\star$};
	\draw(3.25,2) node{$\star$};
	\draw(0.75,0.75) node{$\star$};
	\draw(2,0.75) node{$\star$};
	\draw(3.25,0.75) node{$\star$};
	
	\draw(5.25,1.25) node{$\kappa_4$};
	
	\draw(4.75,3.25) node{$\star$};
	\draw(6,3.25) node{$\star$};
	\draw(7.25,3.25) node{$\star$};

	\draw(4,-0.75) node{(a)};	
	\end{tikzpicture}
\qquad
\begin{tikzpicture}
[scale=0.5]

	\fill[lightgray] (0.5,4.5) -- (3.5,4.5) -- (3.5,7.5) -- (0.5,7.5) -- (0.5,4.5);
	\fill[lightgray] (4.5,4.5) -- (7.5,4.5) -- (7.5,7.5) -- (4.5,7.5) -- (4.5,4.5);
	\fill[lightgray] (3.5,4.5) -- (4.5,4.5) -- (4.5,7.5) -- (3.5,7.5) -- (3.5,4.5);
	\fill[lightgray] (4.5,3.5) -- (7.5,3.5) -- (7.5,4.5) -- (4.5,4.5) -- (4.5,4.5);
	\fill[lightgray] (4.5,0.5) -- (7.5,0.5) -- (7.5,3.5) -- (4.5,3.5) -- (4.5,0.5);	
	
	\draw[blue, thick](0,0) -- (8,0);
	\draw[blue, thick](0,4) -- (8,4);
	\draw[blue, thick](0,0) -- (0,8);
	\draw[blue, thick](8,0) -- (8,8);
	\draw[blue, thick](4,0) -- (4,8);
	\draw[blue, thick](0,8) -- (8,8);
	
	\draw (0,0.5) -- (8,0.5);
	\draw (0,3.5) -- (8,3.5);
	\draw (0,4.5) -- (8,4.5);
	\draw (0,7.5) -- (8,7.5);
	\draw (0.5,0) -- (0.5,8);
	\draw (3.5,0) -- (3.5,8);
	\draw (4.5,0) -- (4.5,8);
	\draw (7.5,0) -- (7.5,8);
	
	
	\filldraw(4,6) circle (4pt);
	\draw (4.30,5.70) node {$v$};
	\filldraw[blue](2,4) circle (4pt);
	\draw (2.30,3.70) node {$v_2$};
	\filldraw(6,4) circle (4pt);
	\draw (6.30,3.70) node {$v_1$};
	\filldraw[blue](4,2) circle (4pt);
	\draw (4.30,1.70) node {$v_3$};
	
	\draw(2,6) node{$G^2(\kappa_2)$};
	\draw(6,6) node{$G^2(\kappa_1)$};
	\draw(4,6.75) node{$G^1(v)$};
	\draw(6,2) node{$G^2(\kappa_4)$};
	\draw(7.25,4.5) node{$G^1(v_1)$};
	
	\draw(4,-0.75) node{(b)};	
\end{tikzpicture}	
\qquad
\begin{tikzpicture}
[scale=0.5]

	\draw[blue, thick](0,0) -- (8,0);
	\draw[blue, thick](0,4) -- (8,4);
	\draw[blue, thick](0,0) -- (0,8);
	\draw[blue, thick](8,0) -- (8,8);
	\draw[blue, thick](4,0) -- (4,8);
	\draw[blue, thick](0,8) -- (8,8);

	\draw (4.30,5.70) node {$v$};		
	\filldraw(4,6) circle (4pt);
	\draw[->] (3.5,6) -- (4.5,6);

	\draw (2.50,4.35) node {$v_2$};
	\draw(2,4) circle (4pt);
		
	\draw (6.50,3.60) node {$v_1$};
	\filldraw(6,4) circle (4pt);
	\draw[->] (6,3.5) -- (6,4.5);
	\filldraw[blue](4,2) circle (4pt);
	\draw (4.5,1.70) node {$v_3$};
	
	\draw(5.5,5) node{$\kappa_1$};
	
	\draw(6,7.25) node{$\star$};
	\draw(7.25,7.25) node{$\star$};
	\draw(6,6) node{$\star$};
	\draw(7.25,6) node{$\star$};
	
	\draw(1.25,5) node{$\kappa_2$};
	
	\draw(3.25,7.25) node{$\star$};
	\draw(3.25,6) node{$\star$};

	\draw(1.25,1.25) node{$\kappa_3$};
	
	\draw(0.75,3.25) node{$\star$};
	\draw(2,3.25) node{$\star$};
	\draw(3.25,3.25) node{$\star$};
	\draw(3.25,2) node{$\star$};
	\draw(3.25,0.75) node{$\star$};
	
	\draw(5.25,1.25) node{$\kappa_4$};
	
	\draw(4.75,3.25) node{$\star$};
	\draw(6,3.25) node{$\star$};
	\draw(7.25,3.25) node{$\star$};

	\draw(4,-0.75) node{(c)};	
	\end{tikzpicture}
\caption{Unfilled dot implies that vertex does not belong to $\cN^0$. (a) Possible cells types  in the setting of Case 1.
(b) Minimal set of tiles in $N_\cN$ in Case 1.
(c) Possible cell types in the setting of Case 1(a).
}
\label{fig:fourCellsNew}
\end{figure}
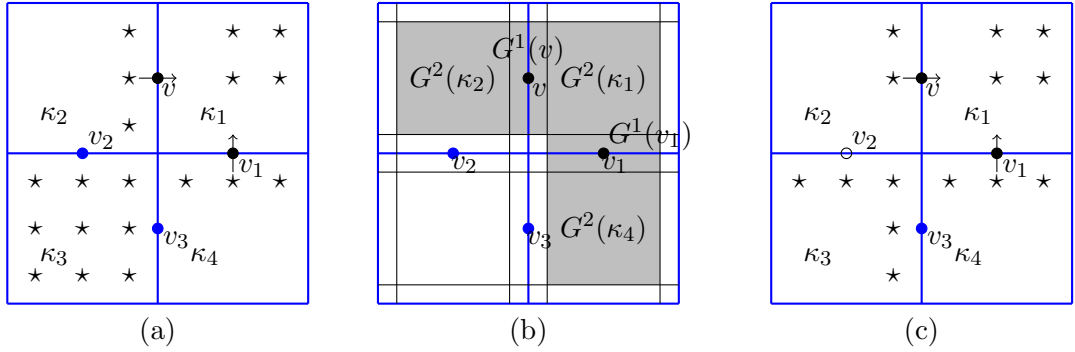
		
	\begin{enumerate}
	\item \emph{Assume $v_2\not\in\cN^0$.}  
		By Proposition~\ref{prop:cellAtt_Rev} $\kappa_2$ cannot be of type $SE$ and thus by Figure~\ref{fig:fourCellsNew}(a) must be  of type $E$ or $NE$.
		By Proposition~\ref{prop:adjacentCells}(i) if $\eta_j\in\sH^1$, then $\kappa_2$ being of type $E$ or $NE$ implies that $\kappa_3$ is of type $NW$, $N$ or $NE$.
		Similarly, by Proposition~\ref{prop:adjacentCells}(v) if $\eta_j\in\sH^2$, then $\kappa_3$ is of type $NE$, $E$, or $SE$.
		The possible cell types are indicated in Figure~\ref{fig:fourCellsNew}(c). 
		Since $v\not\in\cV_e(\kappa_2)$, {\bf Rule 3} does not apply to $\kappa_2$.
		
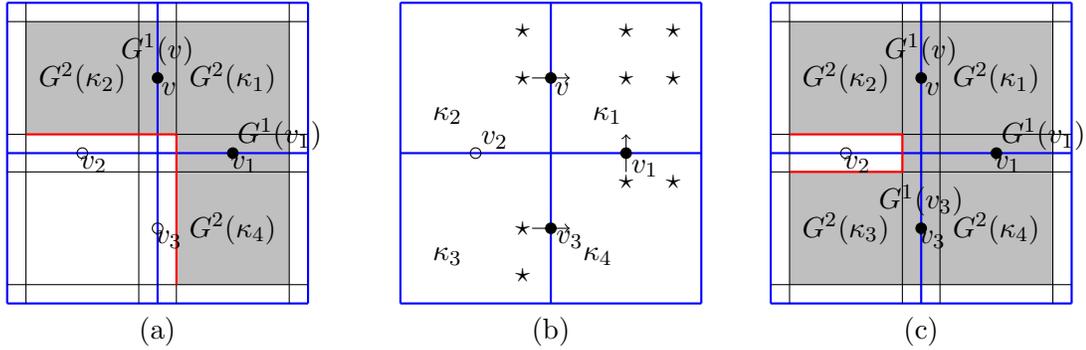
\begin{figure}
\begin{tikzpicture}
[scale=0.5]

	\fill[lightgray] (0.5,4.5) -- (3.5,4.5) -- (3.5,7.5) -- (0.5,7.5) -- (0.5,4.5);
	\fill[lightgray] (4.5,4.5) -- (7.5,4.5) -- (7.5,7.5) -- (4.5,7.5) -- (4.5,4.5);
	\fill[lightgray] (3.5,4.5) -- (4.5,4.5) -- (4.5,7.5) -- (3.5,7.5) -- (3.5,4.5);
	\fill[lightgray] (4.5,3.5) -- (7.5,3.5) -- (7.5,4.5) -- (4.5,4.5) -- (4.5,4.5);
	\fill[lightgray] (4.5,0.5) -- (7.5,0.5) -- (7.5,3.5) -- (4.5,3.5) -- (4.5,0.5);	
	
	\draw[blue, thick](0,0) -- (8,0);
	\draw[blue, thick](0,4) -- (8,4);
	\draw[blue, thick](0,0) -- (0,8);
	\draw[blue, thick](8,0) -- (8,8);
	\draw[blue, thick](4,0) -- (4,8);
	\draw[blue, thick](0,8) -- (8,8);
	
	\draw (0,0.5) -- (8,0.5);
	\draw (0,3.5) -- (8,3.5);
	\draw (0,4.5) -- (8,4.5);
	\draw (0,7.5) -- (8,7.5);
	\draw (0.5,0) -- (0.5,8);
	\draw (3.5,0) -- (3.5,8);
	\draw (4.5,0) -- (4.5,8);
	\draw (7.5,0) -- (7.5,8);
	
	 \draw[red, thick] (0.5,4.5) -- (4.5,4.5) -- (4.5,0.5);	
	
	\filldraw(4,6) circle (4pt);
	\draw (4.30,5.70) node {$v$};
	\draw(2,4) circle (4pt);
	\draw (2.30,3.70) node {$v_2$};
	\filldraw(6,4) circle (4pt);
	\draw (6.30,3.70) node {$v_1$};
	\draw(4,2) circle (4pt);
	\draw (4.30,1.70) node {$v_3$};
	
	\draw(2,6) node{$G^2(\kappa_2)$};
	\draw(6,6) node{$G^2(\kappa_1)$};
	\draw(4,6.75) node{$G^1(v)$};
	\draw(6,2) node{$G^2(\kappa_4)$};
	\draw(7.25,4.5) node{$G^1(v_1)$};
	
	\draw(4,-0.75) node{(a)};	
\end{tikzpicture}
\qquad
\begin{tikzpicture}
[scale=0.5]

	\draw[blue, thick](0,0) -- (8,0);
	\draw[blue, thick](0,4) -- (8,4);
	\draw[blue, thick](0,0) -- (0,8);
	\draw[blue, thick](8,0) -- (8,8);
	\draw[blue, thick](4,0) -- (4,8);
	\draw[blue, thick](0,8) -- (8,8);

	\draw (4.30,5.70) node {$v$};		
	\filldraw(4,6) circle (4pt);
	\draw[->] (3.5,6) -- (4.5,6);
	
	\draw (2.50,4.35) node {$v_2$};	
	\draw(2,4) circle (4pt);

	\draw (6.50,3.60) node {$v_1$};
	\filldraw(6,4) circle (4pt);
	\draw[->] (6,3.5) -- (6,4.5);
	
	\draw (4.5,1.70) node {$v_3$};
	\filldraw(4,2) circle (4pt);
	\draw[->] (3.5,2) -- (4.5,2);

	\draw(5.5,5) node{$\kappa_1$};
	
	\draw(6,7.25) node{$\star$};
	\draw(7.25,7.25) node{$\star$};
	\draw(6,6) node{$\star$};
	\draw(7.25,6) node{$\star$};
	
	\draw(1.25,5) node{$\kappa_2$};
	
	\draw(3.25,7.25) node{$\star$};
	\draw(3.25,6) node{$\star$};

	\draw(1.25,1.25) node{$\kappa_3$};
	
	\draw(3.25,2) node{$\star$};
	\draw(3.25,0.75) node{$\star$};
	
	\draw(5.25,1.25) node{$\kappa_4$};
	
	\draw(6,3.25) node{$\star$};
	\draw(7.25,3.25) node{$\star$};

	\draw(4,-0.75) node{(b)};	
\end{tikzpicture}
\qquad
\begin{tikzpicture}
[scale=0.5]

	\fill[lightgray] (0.5,4.5) -- (3.5,4.5) -- (3.5,7.5) -- (0.5,7.5) -- (0.5,4.5);
	\fill[lightgray] (4.5,4.5) -- (7.5,4.5) -- (7.5,7.5) -- (4.5,7.5) -- (4.5,4.5);
	\fill[lightgray] (3.5,4.5) -- (4.5,4.5) -- (4.5,7.5) -- (3.5,7.5) -- (3.5,4.5);
	\fill[lightgray] (3.5,3.5) -- (4.5,3.5) -- (4.5,4.5) -- (3.5,4.5) -- (3.5,3.5);
	\fill[lightgray] (4.5,3.5) -- (7.5,3.5) -- (7.5,4.5) -- (4.5,4.5) -- (4.5,4.5);
	\fill[lightgray] (4.5,0.5) -- (7.5,0.5) -- (7.5,3.5) -- (4.5,3.5) -- (4.5,0.5);
	\fill[lightgray] (3.5,3.5) -- (4.5,3.5) -- (4.5,0.5) -- (3.5,0.5) -- (3.5,3.5);
	\fill[lightgray] (3.5,3.5) -- (3.5,0.5) -- (0.5,0.5) -- (0.5,3.5) -- (3.5,3.5);

	\draw[blue, thick](0,0) -- (8,0);
	\draw[blue, thick](0,4) -- (8,4);
	\draw[blue, thick](0,0) -- (0,8);
	\draw[blue, thick](8,0) -- (8,8);
	\draw[blue, thick](4,0) -- (4,8);
	\draw[blue, thick](0,8) -- (8,8);
	
	\draw (0,0.5) -- (8,0.5);
	\draw (0,3.5) -- (8,3.5);
	\draw (0,4.5) -- (8,4.5);
	\draw (0,7.5) -- (8,7.5);
	\draw (0.5,0) -- (0.5,8);
	\draw (3.5,0) -- (3.5,8);
	\draw (4.5,0) -- (4.5,8);
	\draw (7.5,0) -- (7.5,8);
	
        \draw[red, thick](0.5,4.5) --  (3.5,4.5) -- (3.5,3.5) -- (0.5,3.5);
        
	
	\filldraw(4,6) circle (4pt);
	\draw (4.30,5.70) node {$v$};
	\draw(2,4) circle (4pt);
	\draw (2.30,3.70) node {$v_2$};
	\filldraw(6,4) circle (4pt);
	\draw (6.30,3.70) node {$v_1$};
	\filldraw(4,2) circle (4pt);
	\draw (4.30,1.70) node {$v_3$};
	
	\draw(2,6) node{$G^2(\kappa_2)$};
	\draw(6,6) node{$G^2(\kappa_1)$};
	\draw(4,6.75) node{$G^1(v)$};
	\draw(2,2) node{$G^2(\kappa_3)$};
	\draw(4,2.75) node{$G^1(v_3)$};
	\draw(6,2) node{$G^2(\kappa_4)$};
	\draw(7.25,4.5) node{$G^1(v_1)$};
	
	\draw(4,-0.75) node{(c)};	
\end{tikzpicture}
\caption{(a) Cells in $N_\cN$ associated with $E(i,j)$ under the assumptions of Case 1(a)(i).
(b) Possible cell types in the setting of Case 1(a)(ii).
(c) Tiles for $N_\cN$ associated with $E(i,j)$ in Case 1(a)(ii).}
\label{fig:case1(a)(i)}
\end{figure}		
		
		\begin{enumerate}
		\item  \emph{Assume $v_3\not\in\cN^0$} 
			Because of the cell type of $\kappa_1$, {\bf Rule 2} does not apply.
			Because $G^0(\pi)\not\subset N_\cN$, {\bf Rule 3}, {\bf Rule 4} and {\bf Rule 5} do not apply.
			We claim that we are in the setting of Figure~\ref{fig:case1(a)(i)}(a).
			In particular, we need to argue that $G^2(\kappa_3)\not\subset N_\cN$.
			Observe that only {\bf Rule 0} and {\bf Rule 1} require the introduction of a 2-tile. 
			By Figure~\ref{fig:fourCellsNew}(c) $\kappa_3$ is not of type $A$ and hence {\bf Rule 0} does not apply.
			Assume {\bf Rule 1} forces the introduction of $G^2(\kappa_3)$.
			Then there exists $u\in\cN^0\cap\cT$ such that $u\in \cV(\kappa_3)$.
			Again, by Figure~\ref{fig:fourCellsNew}(b)$v_2$ or $v_3$ belong to $\cV_a(\kappa_3)$ and hence $v_2$ or $v_3$ belong to $\cN^0$, a contradiction.
			Thus, we are in the setting of Figure~\ref{fig:case1(a)(i)}(a).
			Proposition~\ref{prop:TG2} guarantees the desired transversality on the edges of $G^2(\kappa_2)$ and $G^2(\kappa_4)$.
			Furthermore, Proposition~\ref{prop:TG1j}(i - ii) and Proposition~\ref{prop:TG1i}(i - ii) guarantees the desired transversality on the edges of $G^0(\pi)$.
			
		\item  \emph{Assume $v_3\in\cN^0$ and $v_3$ is transparent east}.  
		The assumption that $v_3$ is transparent east implies that $\kappa_4$ is of type $N$ or $NE$ and $\kappa_3$ is of type $NE$, $E$, or $SE$.
		By Proposition~\ref{prop:cellAtt_Rev} if $\kappa_3$ is of type $NE$ or $\kappa_2$ is of type $SE$, $v_2\in\cN^0$, a contradiction.
		Thus, the possible cell types are as shown in Figure~\ref{fig:case1(a)(i)}(b).
		
		By {\bf Rule 1} $G^2(\kappa_3)\cup G^1(v_3)\subset N_\cN$.
		Applying {\bf Rule 2} to $G^2(\kappa_3)\cup G^2(\kappa_4)\cup G^2(\kappa_1)\cup G^1(v_3)\cup G^1(v_1)$ implies that $G^0(\pi)\subset N_\cN$.		
		This implies that $v_3\not\in\cV_e(\kappa_3)$ and hence, {\bf Rule 3} does not apply to $\kappa_3$.
		{\bf Rule 4} and {\bf Rule 5} do not apply and hence we are in the setting of Figure~\ref{fig:case1(a)(i)}(c).
		Proposition~\ref{prop:TG2} guarantees the desired transversality on the edges of $G^2(\kappa_2)$ and $G^2(\kappa_3)$.
		Proposition~\ref{prop:TG0}(iv-v) guarantees the desired transversality on the edges of $G^0(\kappa_2)$.

		\item  \emph{Assume $v_3\in\cN^0$ and $v_3$ is transparent west}.  
		Then $\kappa_3$ is of type $NW$ or $N$. 
		See Figure~\ref{fig:CellTypes1(b)}(a).
		Since $v_3\in\cN^0$ Proposition~\ref{prop:cellAtt_Rev} implies that $v_2\in\cN^0$ a contradiction.
		Thus this case cannot occur.
		\end{enumerate}

\begin{figure}

\begin{tikzpicture}
[scale=0.5]

	\draw[blue, thick](0,0) -- (8,0);
	\draw[blue, thick](0,4) -- (8,4);
	\draw[blue, thick](0,0) -- (0,8);
	\draw[blue, thick](8,0) -- (8,8);
	\draw[blue, thick](4,0) -- (4,8);
	\draw[blue, thick](0,8) -- (8,8);

	\draw (4.30,5.70) node {$v$};		
	\filldraw(4,6) circle (4pt);
	\draw[->] (3.5,6) -- (4.5,6);
	\draw (2.50,4.35) node {$v_2$};
	\draw(2,4) circle (4pt);

	\draw (6.50,3.60) node {$v_1$};
	\filldraw(6,4) circle (4pt);
	\draw[->] (6,3.5) -- (6,4.5);
	\draw(4.5,1.70) node {$v_3$};
	\filldraw(4,2) circle (4pt);
	\draw[->] (4.5,2) -- (3.5,2);
		
	\draw(5.5,5) node{$\kappa_1$};
	
	\draw(6,7.25) node{$\star$};
	\draw(7.25,7.25) node{$\star$};
	\draw(6,6) node{$\star$};
	\draw(7.25,6) node{$\star$};
	
	\draw(1.25,5) node{$\kappa_2$};
	
	\draw(3.25,7.25) node{$\star$};
	\draw(3.25,6) node{$\star$};

	\draw(1.25,1.25) node{$\kappa_3$};
	
	\draw(0.75,3.25) node{$\star$};
	\draw(2,3.25) node{$\star$};
	
	\draw(5.25,1.25) node{$\kappa_4$};
	
	\draw(4.75,3.25) node{$\star$};
	\draw(6,3.25) node{$\star$};
	\draw(7.25,3.25) node{$\star$};
		
	\draw(4,-0.75) node{(a)};	
\end{tikzpicture}
\qquad
\begin{tikzpicture}
[scale=0.5]

	\draw[blue, thick](0,0) -- (8,0);
	\draw[blue, thick](0,4) -- (8,4);
	\draw[blue, thick](0,0) -- (0,8);
	\draw[blue, thick](8,0) -- (8,8);
	\draw[blue, thick](4,0) -- (4,8);
	\draw[blue, thick](0,8) -- (8,8);

	\draw (4.30,5.70) node {$v$};		
	\filldraw(4,6) circle (4pt);
	\draw[->] (3.5,6) -- (4.5,6);
	\draw (2.50,4.35) node {$v_2$};
	\filldraw(2,4) circle (4pt);
	\draw[->] (2,4.5) -- (2,3.5);	
	
	\draw (6.50,3.60) node {$v_1$};
	\filldraw(6,4) circle (4pt);
	\draw[->] (6,3.5) -- (6,4.5);
	\draw(4.5,1.70) node {$v_3$};
	\filldraw[blue](4,2) circle (4pt);
	
	\draw(5.5,5) node{$\kappa_1$};
	
	\draw(6,7.25) node{$\star$};
	\draw(7.25,7.25) node{$\star$};
	\draw(6,6) node{$\star$};
	\draw(7.25,6) node{$\star$};
	
	\draw(1.25,5) node{$\kappa_2$};
	
	\draw(3.25,4.75) node{$\star$};

	\draw(1.25,1.25) node{$\kappa_3$};
	
	\draw(0.75,2) node{$\star$};
	\draw(2,2) node{$\star$};
	\draw(3.25,2) node{$\star$};
	\draw(0.75,0.75) node{$\star$};
	\draw(2,0.75) node{$\star$};
	\draw(3.25,0.75) node{$\star$};
	
	\draw(5.25,1.25) node{$\kappa_4$};
	
	\draw(4.75,3.25) node{$\star$};
	\draw(6,3.25) node{$\star$};
	\draw(7.25,3.25) node{$\star$};

	\draw(4,-0.75) node{(b)};	
\end{tikzpicture}
\qquad
\begin{tikzpicture}
[scale=0.5]

	\fill[lightgray] (0.5,4.5) -- (3.5,4.5) -- (3.5,7.5) -- (0.5,7.5) -- (0.5,4.5);
	\fill[lightgray] (4.5,4.5) -- (7.5,4.5) -- (7.5,7.5) -- (4.5,7.5) -- (4.5,4.5);
	\fill[lightgray] (3.5,4.5) -- (4.5,4.5) -- (4.5,7.5) -- (3.5,7.5) -- (3.5,4.5);
	\fill[lightgray] (3.5,3.5) -- (4.5,3.5) -- (4.5,4.5) -- (3.5,4.5) -- (3.5,3.5);
	\fill[lightgray] (4.5,3.5) -- (7.5,3.5) -- (7.5,4.5) -- (4.5,4.5) -- (4.5,4.5);
	\fill[lightgray] (4.5,0.5) -- (7.5,0.5) -- (7.5,3.5) -- (4.5,3.5) -- (4.5,0.5);	
	\fill[lightgray] (0.5,0.5) -- (3.5,0.5) -- (3.5,3.5) -- (0.5,3.5) -- (0.5,0.5);
	\fill[lightgray] (0.5,4.5) -- (3.5,4.5) -- (3.5,3.5) -- (0.5,3.5) -- (0.5,4.5);
		
	\draw[blue, thick](0,0) -- (8,0);
	\draw[blue, thick](0,4) -- (8,4);
	\draw[blue, thick](0,0) -- (0,8);
	\draw[blue, thick](8,0) -- (8,8);
	\draw[blue, thick](4,0) -- (4,8);
	\draw[blue, thick](0,8) -- (8,8);
	
	\draw (0,0.5) -- (8,0.5);
	\draw (0,3.5) -- (8,3.5);
	\draw (0,4.5) -- (8,4.5);
	\draw (0,7.5) -- (8,7.5);
	\draw (0.5,0) -- (0.5,8);
	\draw (3.5,0) -- (3.5,8);
	\draw (4.5,0) -- (4.5,8);
	\draw (7.5,0) -- (7.5,8);
	
	\filldraw(4,6) circle (4pt);
	\draw (4.30,5.70) node {$v$};
	\filldraw (2,4) circle (4pt);
	\draw (2.30,3.70) node {$v_2$};
	\filldraw(6,4) circle (4pt);
	\draw (6.30,3.70) node {$v_1$};
	\filldraw[blue](4,2) circle (4pt);
	\draw (4.30,1.70) node {$v_3$};
	
	\draw(2,6) node{$G^2(\kappa_2)$};
	\draw(6,6) node{$G^2(\kappa_1)$};
	\draw(4,6.75) node{$G^1(v)$};
	\draw(6,2) node{$G^2(\kappa_4)$};
	\draw(2,2) node{$G^2(\kappa_3)$};
	\draw(7.25,4.5) node{$G^1(v_1)$};
	\draw(0.75,4.5) node{$G^1(v_2)$};
	
	\draw(4,-0.75) node{(c)};
	
\end{tikzpicture}
\caption{(a) Possible cell types in the setting of Case 1(a)(iii). Observe that the possible cell types of $\kappa_2$ and $\kappa_3$ cannot occur simultaneously. Thus this case cannot occur.
(b) Possible cell types in the setting of Case 1(b).
(c) Minimal set of tiles for Case 1(b).
}
\label{fig:CellTypes1(b)}
\end{figure}
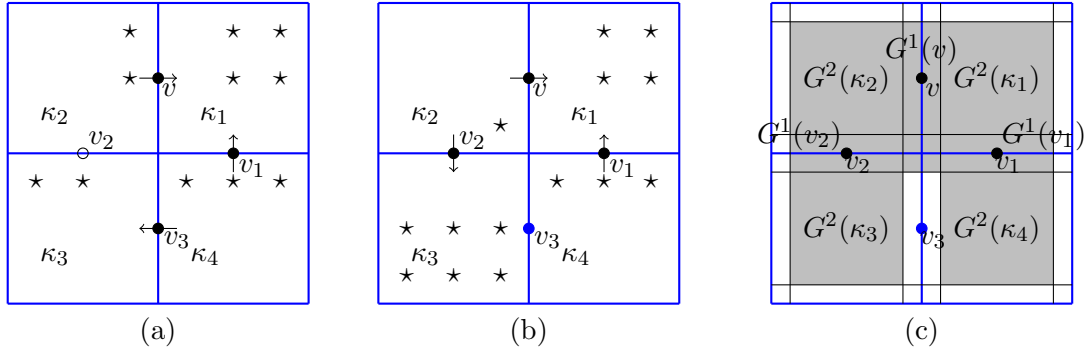
		
	\item \emph{Assume $v_2\in\cN^0$ and $v_2$ is transparent south}.
		This implies that $\kappa_2$ is of type $SE$ and $\kappa_3$ is of type $E$ or $SE$.
		See Figure~\ref{fig:CellTypes1(b)}(b).
		By {\bf Rule 1}, $G^2(\kappa_3)\cup G^1(v_2) \subset N_\cN$.
		By {\bf Rule 2}, $G^0(\pi)\subset N_\cN$.
		See Figure~\ref{fig:CellTypes1(b)}(c).
		
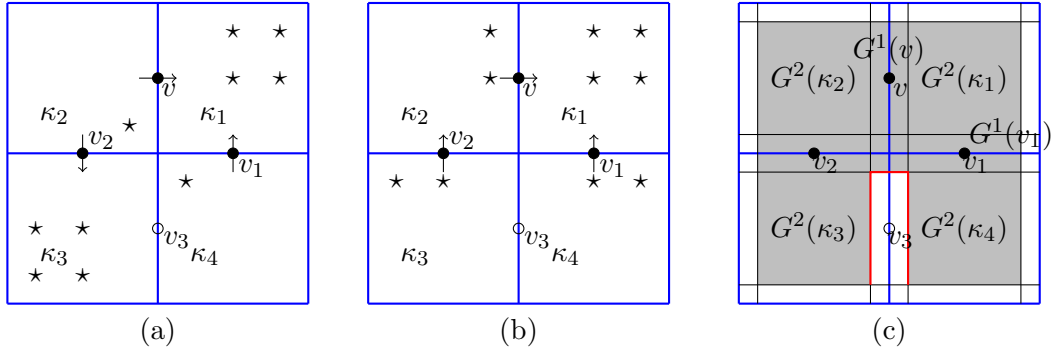
\begin{figure}
\begin{tikzpicture}
[scale=0.5]

	\draw[blue, thick](0,0) -- (8,0);
	\draw[blue, thick](0,4) -- (8,4);
	\draw[blue, thick](0,0) -- (0,8);
	\draw[blue, thick](8,0) -- (8,8);
	\draw[blue, thick](4,0) -- (4,8);
	\draw[blue, thick](0,8) -- (8,8);

	\draw (4.30,5.70) node {$v$};		
	\filldraw(4,6) circle (4pt);
	\draw[->] (3.5,6) -- (4.5,6);
	\draw (2.50,4.35) node {$v_2$};
	\filldraw(2,4) circle (4pt);
	\draw[->] (2,4.5) -- (2,3.5);	
	
	\draw (6.50,3.60) node {$v_1$};
	\filldraw(6,4) circle (4pt);
	\draw[->] (6,3.5) -- (6,4.5);
	\draw(4.5,1.70) node {$v_3$};
	\draw(4,2) circle (4pt);
	
	\draw(5.5,5) node{$\kappa_1$};
	
	\draw(6,7.25) node{$\star$};
	\draw(7.25,7.25) node{$\star$};
	\draw(6,6) node{$\star$};
	\draw(7.25,6) node{$\star$};
	
	\draw(1.25,5) node{$\kappa_2$};
	
	\draw(3.25,4.75) node{$\star$};

	\draw(1.25,1.25) node{$\kappa_3$};
	
	\draw(0.75,2) node{$\star$};
	\draw(2,2) node{$\star$};
	\draw(0.75,0.75) node{$\star$};
	\draw(2,0.75) node{$\star$};
	
	\draw(5.25,1.25) node{$\kappa_4$};
	
	\draw(4.75,3.25) node{$\star$};

	\draw(4,-0.75) node{(a)};	
\end{tikzpicture}%
\qquad
\begin{tikzpicture}
[scale=0.5]

	\draw[blue, thick](0,0) -- (8,0);
	\draw[blue, thick](0,4) -- (8,4);
	\draw[blue, thick](0,0) -- (0,8);
	\draw[blue, thick](8,0) -- (8,8);
	\draw[blue, thick](4,0) -- (4,8);
	\draw[blue, thick](0,8) -- (8,8);

	\draw (4.30,5.70) node {$v$};		
	\filldraw(4,6) circle (4pt);
	\draw[->] (3.5,6) -- (4.5,6);
	\draw (2.50,4.35) node {$v_2$};
	\filldraw(2,4) circle (4pt);
	\draw[->] (2,3.5) -- (2,4.5);	
	
	\draw (6.50,3.60) node {$v_1$};
	\filldraw(6,4) circle (4pt);
	\draw[->] (6,3.5) -- (6,4.5);
	\draw(4.5,1.70) node {$v_3$};
	\draw(4,2) circle (4pt);
	
	\draw(5.5,5) node{$\kappa_1$};
	
	\draw(6,7.25) node{$\star$};
	\draw(7.25,7.25) node{$\star$};
	\draw(6,6) node{$\star$};
	\draw(7.25,6) node{$\star$};
	
	\draw(1.25,5) node{$\kappa_2$};
	
	\draw(3.25,7.25) node{$\star$};
	\draw(3.25,6) node{$\star$};

	\draw(1.25,1.25) node{$\kappa_3$};
	
	\draw(0.75,3.25) node{$\star$};
	\draw(2,3.25) node{$\star$};
	
	\draw(5.25,1.25) node{$\kappa_4$};
	
	\draw(6,3.25) node{$\star$};
	\draw(7.25,3.25) node{$\star$};

	\draw(4,-0.75) node{(b)};	
\end{tikzpicture}
\qquad
\begin{tikzpicture}
[scale=0.5]

	\fill[lightgray] (0.5,4.5) -- (3.5,4.5) -- (3.5,7.5) -- (0.5,7.5) -- (0.5,4.5);
	\fill[lightgray] (4.5,4.5) -- (7.5,4.5) -- (7.5,7.5) -- (4.5,7.5) -- (4.5,4.5);
	\fill[lightgray] (3.5,4.5) -- (4.5,4.5) -- (4.5,7.5) -- (3.5,7.5) -- (3.5,4.5);
	\fill[lightgray] (3.5,3.5) -- (4.5,3.5) -- (4.5,4.5) -- (3.5,4.5) -- (3.5,3.5);
	\fill[lightgray] (4.5,3.5) -- (7.5,3.5) -- (7.5,4.5) -- (4.5,4.5) -- (4.5,4.5);
	\fill[lightgray] (4.5,0.5) -- (7.5,0.5) -- (7.5,3.5) -- (4.5,3.5) -- (4.5,0.5);
	\fill[lightgray] (3.5,3.5) -- (0.5,3.5) -- (0.5,4.5) -- (3.5,4.5) -- (3.5,3.5);
	\fill[lightgray] (3.5,3.5) -- (3.5,0.5) -- (0.5,0.5) -- (0.5,3.5) -- (3.5,3.5);

	\draw[blue, thick](0,0) -- (8,0);
	\draw[blue, thick](0,4) -- (8,4);
	\draw[blue, thick](0,0) -- (0,8);
	\draw[blue, thick](8,0) -- (8,8);
	\draw[blue, thick](4,0) -- (4,8);
	\draw[blue, thick](0,8) -- (8,8);
	
	\draw (0,0.5) -- (8,0.5);
	\draw (0,3.5) -- (8,3.5);
	\draw (0,4.5) -- (8,4.5);
	\draw (0,7.5) -- (8,7.5);
	\draw (0.5,0) -- (0.5,8);
	\draw (3.5,0) -- (3.5,8);
	\draw (4.5,0) -- (4.5,8);
	\draw (7.5,0) -- (7.5,8);
	
        \draw[red, thick](3.5,0.5) --  (3.5,3.5) -- (4.5,3.5) -- (4.5,0.5);
        
	
	\filldraw(4,6) circle (4pt);
	\draw (4.30,5.70) node {$v$};
	\filldraw(2,4) circle (4pt);
	\draw (2.30,3.70) node {$v_2$};
	\filldraw(6,4) circle (4pt);
	\draw (6.30,3.70) node {$v_1$};
	\draw(4,2) circle (4pt);
	\draw (4.30,1.70) node {$v_3$};
	
	\draw(2,6) node{$G^2(\kappa_2)$};
	\draw(6,6) node{$G^2(\kappa_1)$};
	\draw(4,6.75) node{$G^1(v)$};
	\draw(2,2) node{$G^2(\kappa_3)$};
	\draw(6,2) node{$G^2(\kappa_4)$};
	\draw(7.25,4.5) node{$G^1(v_1)$};
	
	\draw(4,-0.75) node{(c)};	
\end{tikzpicture}
\caption{(a) Possible cell types  in the setting of Case 1(b).
(b) Possible cell types in the setting of Case 1(c)(i).
(c) Tiles in $N_\cN$ that are related to $E(i,j)$ in Case 1(c)(i). 
The vector field is transverse in along the  interior edges indicated in red.
}
\label{fig:CellTypes1b}
\end{figure}
		
		\begin{enumerate}
		\item  \emph{Assume $v_3\not\in\cN^0$.} 
			Since $v_2\in\cN^0$, if $\kappa_3$ is of type $E$ or $SE$, then $v_3\in\cV_a(\kappa_3)$, a contradiction.
			Thus, $\kappa_3$ is of type $W$, $A$, $SW$ or $S$, and hence, by Proposition~\ref{prop:adjacentCells}(x) $\kappa_4$ is not of type $N$ or $NE$.
			Therefore, the set of possible types is as in Figure~\ref{fig:CellTypes1b}(a) from which we note that $v_3\in\cV_a(\kappa_4)$.
			Since 
			 $v_1\in\cN^0$, by Proposition~\ref{prop:cellAtt_Rev} $v_3\in \cN^0$, a contradiction.
			Therefore, this case cannot occur.
		\item  \emph{Assume $v_3\in\cN^0$}.  By {\bf Rule 1}, $G^1(v_3) \subset N_\cN$.
		Combining this with the information from Figure~\ref{fig:CellTypes1(b)}(c) we observe that there are no interior boundary edges to check.
		\end{enumerate}
	\item \emph{Assume $v_2\in\cN^0$ and $v_2$ is transparent north}.

		\begin{enumerate}
		\item  \emph{Assume $v_3\not\in\cN^0$}. 
		By Proposition~\ref{prop:cellAtt_Rev} the assumption that $v_3\not\in\cN^0$ implies that $\kappa_3$ is of type $NW$ or $N$ and $\kappa_4$ is of type $N$ or $NE$ (see Figure~\ref{fig:CellTypes1b}(b)).
		By {\bf Rule 1} $G^2(\kappa_3)\cup G^1(v_2)\subset N_\cN$.
		By {\bf Rule 2} $G^0(\pi)\subset N_\cN$.
		{\bf Rule 3}-{\bf 5} do not apply, thus we are in the setting of Figure~\ref{fig:CellTypes1b}(c).
		
		Thus, by Proposition~\ref{prop:TG2} the vector field is transverse in along the boundary edges of $G^2(\kappa_3)$ and $G^2(\kappa_4)$.
		By Proposition~\ref{prop:TG0}(vi-vii) the vector field is transverse in along the boundary edges of $G^0(\pi)$.
		
		\item  \emph{Assume $v_3\in\cN^0$.}  By {\bf Rule 1}, $G^1(v_3) \subset N_\cN$.
		Observe that there are no interior boundary edges to check.
		\end{enumerate}
	\end{enumerate}

\item {\em Assume $v_1\in\cN^0$ and $v_1$ is transparent south}. 
	Recall that Figure~\ref{fig:MinimalFourCells}(a) provides an upper bound of the types of cells.
	Since $v_1$ is transparent down, $v_1\in\cV_a(\kappa_1)$ and $v_1\in\cV_e(\kappa_4)$.
	Therefore, $\kappa_1$ is of type $S$ or $SE$ and $\kappa_4$ is of type $W$, $A$, $E$, $SW$, $S$ or $SE$.
	The possible types of $\kappa_4$ preclude the possibility of $\kappa_3$ being of type $NW$ or $N$.
	Thus we are in the setting of Figure~\ref{fig:Case2}(a).
	
	By {\bf Rule 1}, $G^2(\kappa_4)\cup G^1(v_1) \subset N_\cN$.
	By {\bf Rule 2}, $G^0(\pi) \subset N_\cN$.
	See Figure~\ref{fig:Case2}(b).
	
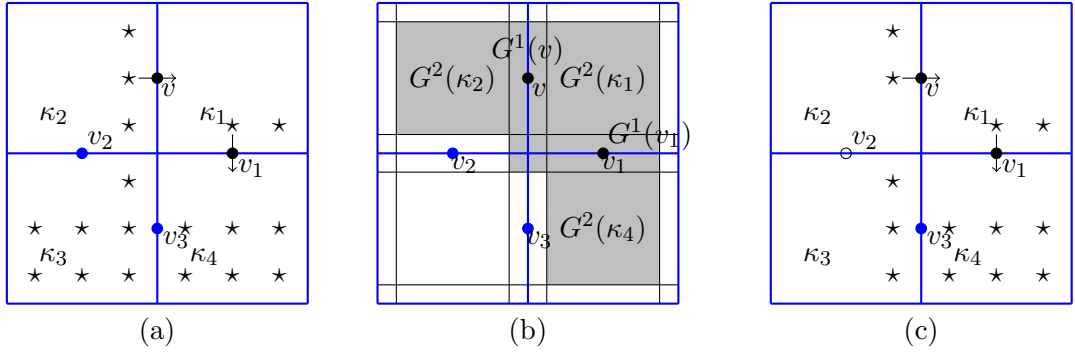
\begin{figure}
\begin{tikzpicture}
[scale=0.5]

	\draw[blue, thick](0,0) -- (8,0);
	\draw[blue, thick](0,4) -- (8,4);
	\draw[blue, thick](0,0) -- (0,8);
	\draw[blue, thick](8,0) -- (8,8);
	\draw[blue, thick](4,0) -- (4,8);
	\draw[blue, thick](0,8) -- (8,8);

	\draw (4.30,5.70) node {$v$};		
	\filldraw(4,6) circle (4pt);
	\draw[->] (3.5,6) -- (4.5,6);
	\draw (2.50,4.35) node {$v_2$};	
	\filldraw[blue](2,4) circle (4pt);
	\draw (6.50,3.60) node {$v_1$};
	\filldraw(6,4) circle (4pt);
	\draw[->] (6,4.5) -- (6,3.5);	
	\draw (4.5,1.70) node {$v_3$};
	\filldraw[blue](4,2) circle (4pt);

	\draw(5.5,5) node{$\kappa_1$};
	
	\draw(6,4.75) node{$\star$};
	\draw(7.25,4.75) node{$\star$};
	
	\draw(1.25,5) node{$\kappa_2$};
	
	\draw(3.25,7.25) node{$\star$};
	\draw(3.25,6) node{$\star$};
	\draw(3.25,4.75) node{$\star$};

	\draw(1.25,1.25) node{$\kappa_3$};
	
	\draw(3.25,3.25) node{$\star$};
	\draw(0.75,2) node{$\star$};
	\draw(2,2) node{$\star$};
	\draw(3.25,2) node{$\star$};
	\draw(0.75,0.75) node{$\star$};
	\draw(2,0.75) node{$\star$};
	\draw(3.25,0.75) node{$\star$};
	
	\draw(5.25,1.25) node{$\kappa_4$};
	
	\draw(4.75,2) node{$\star$};
	\draw(6,2) node{$\star$};
	\draw(7.25,2) node{$\star$};
	\draw(4.75,0.75) node{$\star$};
	\draw(6,0.75) node{$\star$};
	\draw(7.25,0.75) node{$\star$};

	\draw(4,-0.75) node{(a)};	
\end{tikzpicture}
\qquad
\begin{tikzpicture}
[scale=0.5]

	\fill[lightgray] (0.5,4.5) -- (3.5,4.5) -- (3.5,7.5) -- (0.5,7.5) -- (0.5,4.5);
	\fill[lightgray] (4.5,4.5) -- (7.5,4.5) -- (7.5,7.5) -- (4.5,7.5) -- (4.5,4.5);
	\fill[lightgray] (3.5,4.5) -- (4.5,4.5) -- (4.5,7.5) -- (3.5,7.5) -- (3.5,4.5);
	\fill[lightgray] (3.5,3.5) -- (4.5,3.5) -- (4.5,4.5) -- (3.5,4.5) -- (3.5,3.5);
	\fill[lightgray] (4.5,3.5) -- (7.5,3.5) -- (7.5,4.5) -- (4.5,4.5) -- (4.5,4.5);
	\fill[lightgray] (4.5,0.5) -- (7.5,0.5) -- (7.5,3.5) -- (4.5,3.5) -- (4.5,0.5);	
	
	\draw[blue, thick](0,0) -- (8,0);
	\draw[blue, thick](0,4) -- (8,4);
	\draw[blue, thick](0,0) -- (0,8);
	\draw[blue, thick](8,0) -- (8,8);
	\draw[blue, thick](4,0) -- (4,8);
	\draw[blue, thick](0,8) -- (8,8);
	
	\draw (0,0.5) -- (8,0.5);
	\draw (0,3.5) -- (8,3.5);
	\draw (0,4.5) -- (8,4.5);
	\draw (0,7.5) -- (8,7.5);
	\draw (0.5,0) -- (0.5,8);
	\draw (3.5,0) -- (3.5,8);
	\draw (4.5,0) -- (4.5,8);
	\draw (7.5,0) -- (7.5,8);
	
	\filldraw(4,6) circle (4pt);
	\draw (4.30,5.70) node {$v$};
	\filldraw[blue] (2,4) circle (4pt);
	\draw (2.30,3.70) node {$v_2$};
	\filldraw(6,4) circle (4pt);
	\draw (6.30,3.70) node {$v_1$};
	\filldraw[blue](4,2) circle (4pt);
	\draw (4.30,1.70) node {$v_3$};
	
	\draw(2,6) node{$G^2(\kappa_2)$};
	\draw(6,6) node{$G^2(\kappa_1)$};
	\draw(4,6.75) node{$G^1(v)$};
	\draw(6,2) node{$G^2(\kappa_4)$};
	\draw(7.25,4.5) node{$G^1(v_1)$};
	
	\draw(4,-0.75) node{(b)};
	
\end{tikzpicture}
\qquad
\begin{tikzpicture}
[scale=0.5]

	\draw[blue, thick](0,0) -- (8,0);
	\draw[blue, thick](0,4) -- (8,4);
	\draw[blue, thick](0,0) -- (0,8);
	\draw[blue, thick](8,0) -- (8,8);
	\draw[blue, thick](4,0) -- (4,8);
	\draw[blue, thick](0,8) -- (8,8);

	\draw (4.30,5.70) node {$v$};		
	\filldraw(4,6) circle (4pt);
	\draw[->] (3.5,6) -- (4.5,6);
	\draw (2.50,4.35) node {$v_2$};	
	\draw(2,4) circle (4pt);
	\draw (6.50,3.60) node {$v_1$};
	\filldraw(6,4) circle (4pt);
	\draw[->] (6,4.5) -- (6,3.5);	
	\draw (4.5,1.70) node {$v_3$};
	\filldraw[blue](4,2) circle (4pt);
	
	\draw(5.5,5) node{$\kappa_1$};
	
	\draw(6,4.75) node{$\star$};
	\draw(7.25,4.75) node{$\star$};
	
	\draw(1.25,5) node{$\kappa_2$};
	
	\draw(3.25,7.25) node{$\star$};
	\draw(3.25,6) node{$\star$};

	\draw(1.25,1.25) node{$\kappa_3$};
	
	\draw(3.25,3.25) node{$\star$};
	\draw(3.25,2) node{$\star$};
	\draw(3.25,0.75) node{$\star$};
	
	\draw(5.25,1.25) node{$\kappa_4$};
	
	\draw(4.75,2) node{$\star$};
	\draw(6,2) node{$\star$};
	\draw(7.25,2) node{$\star$};
	\draw(4.75,0.75) node{$\star$};
	\draw(6,0.75) node{$\star$};
	\draw(7.25,0.75) node{$\star$};

	\draw(4,-0.75) node{(c)};	
\end{tikzpicture}
\caption{(a) Possible cell types for Case 2.
(b) Minimal set of grid elements in $N_\cN$ related to $E(i,j)$ in Case 2(a).
(c) Possible cell types for Case 2(a).
}
\label{fig:Case2}
\end{figure}

	\begin{enumerate}
	\item \emph{Assume $v_2\not\in\cN^0$}. 
		Figure~\ref{fig:Case2}(a) provides an upper bound on the cell types.
		By Proposition~\ref{prop:cellAtt_Rev} $\kappa_2$ is of type $E$ or $NE$.
		This in turn, by  Proposition~\ref{prop:adjacentCells}, implies that $\kappa_3$ must be of type $E$, $NE$ or $SE$ (see Figure~\ref{fig:Case2}(c)). 
		
		Therefore, $v\in\cV_a(\kappa_2)$ and hence {\bf Rule 3} does not apply to $\kappa_2$.
		Thus we are still in the setting of Figure~\ref{fig:Case2}(b).
			
		\begin{enumerate}
		\item  \emph{Assume $v_3\not\in\cN^0$}. 
			By Proposition~\ref{prop:cellAtt_Rev} there cannot be an edge $v_1\to v_3$.
			Thus, $\kappa_4$ is of type $A$, $S$, $E$, or $SE$ (see Figure~\ref{fig:Case2*}(a)).  
			We claim that $G^2(\kappa_3)\not\subset N_\cN$.
			Observe that only {\bf Rule 0} and {\bf Rule 1} require the introduction of a 2-tile. 
			$\kappa_3$ is not of type $A$ and hence {\bf Rule 0} does not apply.
			Assume {\bf Rule 1} forces the introduction of $G^2(\kappa_3)$.
			Then there exists $u\in\cN^0\cap\cT$ such that $u\in \cV(\kappa_3)$.
			Since $\kappa_3$ is of type $E$, $NE$ or $SE$, $v_3\in\cV_a(\kappa_3)$ and hence $v_3\in\cN^0$, a contradiction.
			By Proposition~\ref{prop:looseG^1}, $G^1(v_2)$ and $G^1(v_3)$ are not in $N_\cN$.
			Furthermore, {\bf Rule 5} does not apply.
			Thus we remain in the setting of Figure~\ref{fig:Case2}(b).			
				
			\begin{enumerate}
			\item  \emph{Assume $\kappa_3$ is of type $NE$. }
				Then {\bf Rule 4} does not apply to $\kappa_3$.
				By {\bf Rule 3} applied to $\kappa_4$, $C^n(\kappa_4, v_3, \pi)\subset N_\cN$.
				See Figure~\ref{fig:Case2*}(b).
				We need to check that the interior edges of $\partial N_\cN$ are transverse in.
				Observe that $v_2\in\cV_e(\kappa_2)$ and $v_3\in\cV_e(\kappa_4)$, hence by Proposition~\ref{prop:TG2} we have the desire transversality along the edge of $G^2(\kappa_2)$ and $G^2(\kappa_4)$.
				By Proposition~\ref{prop:TG2}(iv-vi) we have the desired transversality for $G^0(\pi)$.
				By Proposition~\ref{prop:TC} we have the desired transversality for $C^n(\kappa_4, v_3, \pi)$.
			\item  \emph{Assume $\kappa_3$ is of type $E$ or $SE$. }
				{\bf Rule 4} applied to $\kappa_3$ implies that $C^w(\kappa_4,v_3,\pi)\subset N_\cN$.
				See Figure~\ref{fig:Case2*}(c).
				The desired transversality follows from the argument used in the previous case. 			
			\end{enumerate}
			
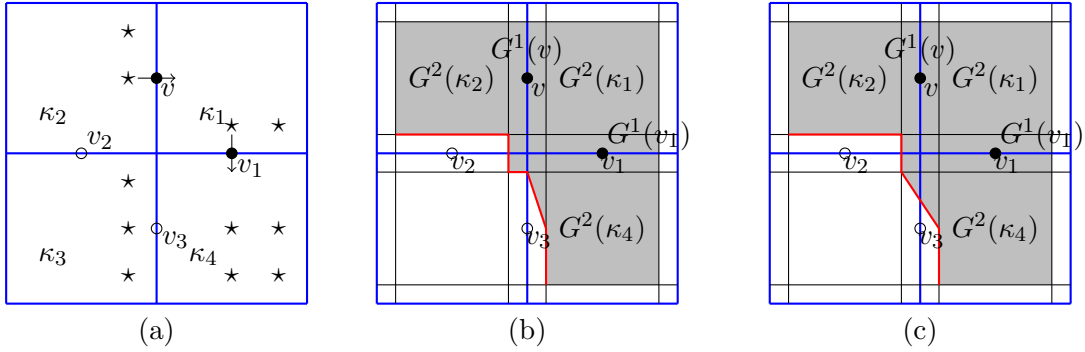
\begin{figure}
\begin{tikzpicture}
[scale=0.5]

	\draw[blue, thick](0,0) -- (8,0);
	\draw[blue, thick](0,4) -- (8,4);
	\draw[blue, thick](0,0) -- (0,8);
	\draw[blue, thick](8,0) -- (8,8);
	\draw[blue, thick](4,0) -- (4,8);
	\draw[blue, thick](0,8) -- (8,8);

	\draw (4.30,5.70) node {$v$};		
	\filldraw(4,6) circle (4pt);
	\draw[->] (3.5,6) -- (4.5,6);
	\draw (2.50,4.35) node {$v_2$};	
	\draw(2,4) circle (4pt);
	\draw (6.50,3.60) node {$v_1$};
	\filldraw(6,4) circle (4pt);
	\draw[->] (6,4.5) -- (6,3.5);	
	\draw (4.5,1.70) node {$v_3$};
	\draw(4,2) circle (4pt);
	
	\draw(5.5,5) node{$\kappa_1$};
	
	\draw(6,4.75) node{$\star$};
	\draw(7.25,4.75) node{$\star$};
	
	\draw(1.25,5) node{$\kappa_2$};
	
	\draw(3.25,7.25) node{$\star$};
	\draw(3.25,6) node{$\star$};

	\draw(1.25,1.25) node{$\kappa_3$};
	
	\draw(3.25,3.25) node{$\star$};
	\draw(3.25,2) node{$\star$};
	\draw(3.25,0.75) node{$\star$};
	
	\draw(5.25,1.25) node{$\kappa_4$};
	
	\draw(6,2) node{$\star$};
	\draw(7.25,2) node{$\star$};
	\draw(6,0.75) node{$\star$};
	\draw(7.25,0.75) node{$\star$};

	\draw(4,-0.75) node{(a)};	
\end{tikzpicture}
\qquad
\begin{tikzpicture}
[scale=0.5]

	\fill[lightgray] (0.5,4.5) -- (3.5,4.5) -- (3.5,7.5) -- (0.5,7.5) -- (0.5,4.5);
	\fill[lightgray] (4.5,4.5) -- (7.5,4.5) -- (7.5,7.5) -- (4.5,7.5) -- (4.5,4.5);
	\fill[lightgray] (3.5,4.5) -- (4.5,4.5) -- (4.5,7.5) -- (3.5,7.5) -- (3.5,4.5);
	\fill[lightgray] (3.5,3.5) -- (4.5,3.5) -- (4.5,4.5) -- (3.5,4.5) -- (3.5,3.5);
	\fill[lightgray] (4.5,3.5) -- (7.5,3.5) -- (7.5,4.5) -- (4.5,4.5) -- (4.5,4.5);
	\fill[lightgray] (4.5,0.5) -- (7.5,0.5) -- (7.5,3.5) -- (4.5,3.5) -- (4.5,0.5);
	\fill[lightgray] (4.0,3.5) -- (4.5,3.5) -- (4.5,2.0) -- (4.0,3.5);

	\draw[blue, thick](0,0) -- (8,0);
	\draw[blue, thick](0,4) -- (8,4);
	\draw[blue, thick](0,0) -- (0,8);
	\draw[blue, thick](8,0) -- (8,8);
	\draw[blue, thick](4,0) -- (4,8);
	\draw[blue, thick](0,8) -- (8,8);

	\draw (0,0.5) -- (8,0.5);
	\draw (0,3.5) -- (8,3.5);
	\draw (0,4.5) -- (8,4.5);
	\draw (0,7.5) -- (8,7.5);
	\draw (0.5,0) -- (0.5,8);
	\draw (3.5,0) -- (3.5,8);
	\draw (4.5,0) -- (4.5,8);
	\draw (7.5,0) -- (7.5,8);
	
        \draw[red, thick](0.5,4.5) --  (3.5,4.5) -- (3.5,3.5) -- (4.0,3.5) -- (4.5,2) -- (4.5,0.5);
        

	\filldraw(4,6) circle (4pt);
	\draw (4.30,5.70) node {$v$};
	\draw(2,4) circle (4pt);
	\draw (2.30,3.70) node {$v_2$};
	\filldraw(6,4) circle (4pt);
	\draw (6.30,3.70) node {$v_1$};
	\draw(4,2) circle (4pt);
	\draw (4.30,1.70) node {$v_3$};
	
	\draw(2,6) node{$G^2(\kappa_2)$};
	\draw(6,6) node{$G^2(\kappa_1)$};
	\draw(4,6.75) node{$G^1(v)$};
	\draw(6,2) node{$G^2(\kappa_4)$};
	\draw(7.25,4.5) node{$G^1(v_1)$};
	
	\draw(4,-0.75) node{(b)};
	
\end{tikzpicture}
\qquad
\begin{tikzpicture}
[scale=0.5]

	\fill[lightgray] (0.5,4.5) -- (3.5,4.5) -- (3.5,7.5) -- (0.5,7.5) -- (0.5,4.5);
	\fill[lightgray] (4.5,4.5) -- (7.5,4.5) -- (7.5,7.5) -- (4.5,7.5) -- (4.5,4.5);
	\fill[lightgray] (3.5,4.5) -- (4.5,4.5) -- (4.5,7.5) -- (3.5,7.5) -- (3.5,4.5);
	\fill[lightgray] (3.5,3.5) -- (4.5,3.5) -- (4.5,4.5) -- (3.5,4.5) -- (3.5,3.5);
	\fill[lightgray] (4.5,3.5) -- (7.5,3.5) -- (7.5,4.5) -- (4.5,4.5) -- (4.5,4.5);
	\fill[lightgray] (4.5,0.5) -- (7.5,0.5) -- (7.5,3.5) -- (4.5,3.5) -- (4.5,0.5);
	\fill[lightgray] (3.5,3.5) -- (4.5,3.5) -- (4.5,2.0) -- (3.5,3.5);

	\draw[blue, thick](0,0) -- (8,0);
	\draw[blue, thick](0,4) -- (8,4);
	\draw[blue, thick](0,0) -- (0,8);
	\draw[blue, thick](8,0) -- (8,8);
	\draw[blue, thick](4,0) -- (4,8);
	\draw[blue, thick](0,8) -- (8,8);

	\draw (0,0.5) -- (8,0.5);
	\draw (0,3.5) -- (8,3.5);
	\draw (0,4.5) -- (8,4.5);
	\draw (0,7.5) -- (8,7.5);
	\draw (0.5,0) -- (0.5,8);
	\draw (3.5,0) -- (3.5,8);
	\draw (4.5,0) -- (4.5,8);
	\draw (7.5,0) -- (7.5,8);
	
        \draw[red, thick](0.5,4.5) --  (3.5,4.5) -- (3.5,3.5) --  (4.5,2) -- (4.5,0.5);

	\filldraw(4,6) circle (4pt);
	\draw (4.30,5.70) node {$v$};
	\draw(2,4) circle (4pt);
	\draw (2.30,3.70) node {$v_2$};
	\filldraw(6,4) circle (4pt);
	\draw (6.30,3.70) node {$v_1$};
	\draw(4,2) circle (4pt);
	\draw (4.30,1.70) node {$v_3$};
	
	\draw(2,6) node{$G^2(\kappa_2)$};
	\draw(6,6) node{$G^2(\kappa_1)$};
	\draw(4,6.75) node{$G^1(v)$};
	\draw(6,2) node{$G^2(\kappa_4)$};
	\draw(7.25,4.5) node{$G^1(v_1)$};
	
	\draw(4,-0.75) node{(c)};
	
\end{tikzpicture}

\caption{(a) Possible cell types for Case 2(a)(i).
(b) Set of grid elements and chips in $N_\cN$ associated with $E(i,j)$ in Case 2(a)(i)(A).
(c) Set of grid elements and chips in $N_\cN$ associated with $E(i,j)$ in Case 2(a)(i)(B).
}
\label{fig:Case2*}
\end{figure}

		\item  \emph{Assume $v_3\in\cN^0$ and $v_3$ is transparent east}.
			If we take this configuration and rotate it counterclockwise by $90^\circ$ then we are in case 1(c)(i) for which we have already shown the desired transversality.
		
			

		\item  \emph{Assume $v_3\in\cN^0$ and $v_3$ is transparent west}. 
			Figure~\ref{fig:Case2}(c) provides an upper bound on the cell types.
			However, $v_3$ is transparent west implies that $\kappa_3$ cannot be of type $NE$, $E$,  or $SE$.
			Thus this case cannot occur.
	\end{enumerate}

\begin{figure}
\begin{tikzpicture}
[scale=0.5]

	\draw[blue, thick](0,0) -- (8,0);
	\draw[blue, thick](0,4) -- (8,4);
	\draw[blue, thick](0,0) -- (0,8);
	\draw[blue, thick](8,0) -- (8,8);
	\draw[blue, thick](4,0) -- (4,8);
	\draw[blue, thick](0,8) -- (8,8);

	\draw (4.30,5.70) node {$v$};		
	\filldraw(4,6) circle (4pt);
	\draw[->] (3.5,6) -- (4.5,6);
	\draw (2.50,4.35) node {$v_2$};	
	\filldraw(2,4) circle (4pt);
	\draw[->] (2,4.5) -- (2,3.5);	
	\draw (6.50,3.60) node {$v_1$};
	\filldraw(6,4) circle (4pt);
	\draw[->] (6,4.5) -- (6,3.5);	
	\draw (4.5,1.70) node {$v_3$};
	\filldraw[blue](4,2) circle (4pt);

	\draw(5.5,5) node{$\kappa_1$};
	
	\draw(6,4.75) node{$\star$};
	\draw(7.25,4.75) node{$\star$};
	
	\draw(1.25,5) node{$\kappa_2$};
	
	\draw(3.25,4.75) node{$\star$};

	\draw(1.25,1.25) node{$\kappa_3$};
	
	\draw(0.75,2) node{$\star$};
	\draw(2,2) node{$\star$};
	\draw(3.25,2) node{$\star$};
	\draw(0.75,0.75) node{$\star$};
	\draw(2,0.75) node{$\star$};
	\draw(3.25,0.75) node{$\star$};
	
	\draw(5.25,1.25) node{$\kappa_4$};
	
	\draw(4.75,2) node{$\star$};
	\draw(6,2) node{$\star$};
	\draw(7.25,2) node{$\star$};
	\draw(4.75,0.75) node{$\star$};
	\draw(6,0.75) node{$\star$};
	\draw(7.25,0.75) node{$\star$};

	\draw(4,-0.75) node{(a)};	
\end{tikzpicture}
\qquad
\begin{tikzpicture}
[scale=0.5]

	\fill[lightgray] (0.5,4.5) -- (3.5,4.5) -- (3.5,7.5) -- (0.5,7.5) -- (0.5,4.5);
	\fill[lightgray] (4.5,4.5) -- (7.5,4.5) -- (7.5,7.5) -- (4.5,7.5) -- (4.5,4.5);
	\fill[lightgray] (3.5,4.5) -- (4.5,4.5) -- (4.5,7.5) -- (3.5,7.5) -- (3.5,4.5);
	\fill[lightgray] (3.5,3.5) -- (4.5,3.5) -- (4.5,4.5) -- (3.5,4.5) -- (3.5,3.5);
	\fill[lightgray] (4.5,3.5) -- (7.5,3.5) -- (7.5,4.5) -- (4.5,4.5) -- (4.5,4.5);
	\fill[lightgray] (4.5,0.5) -- (7.5,0.5) -- (7.5,3.5) -- (4.5,3.5) -- (4.5,0.5);
	\fill[lightgray] (3.5,3.5) -- (0.5,3.5) -- (0.5,4.5) -- (3.5,4.5) -- (3.5,3.5);
	\fill[lightgray] (3.5,3.5) -- (3.5,0.5) -- (0.5,0.5) -- (0.5,3.5) -- (3.5,3.5);

	\draw[blue, thick](0,0) -- (8,0);
	\draw[blue, thick](0,4) -- (8,4);
	\draw[blue, thick](0,0) -- (0,8);
	\draw[blue, thick](8,0) -- (8,8);
	\draw[blue, thick](4,0) -- (4,8);
	\draw[blue, thick](0,8) -- (8,8);
	
	\draw (0,0.5) -- (8,0.5);
	\draw (0,3.5) -- (8,3.5);
	\draw (0,4.5) -- (8,4.5);
	\draw (0,7.5) -- (8,7.5);
	\draw (0.5,0) -- (0.5,8);
	\draw (3.5,0) -- (3.5,8);
	\draw (4.5,0) -- (4.5,8);
	\draw (7.5,0) -- (7.5,8);
	
	
	\filldraw(4,6) circle (4pt);
	\draw (4.30,5.70) node {$v$};
	\filldraw(2,4) circle (4pt);
	\draw (2.30,3.70) node {$v_2$};
	\filldraw(6,4) circle (4pt);
	\draw (6.30,3.70) node {$v_1$};
	\draw(4,2) circle (4pt);
	\draw (4.30,1.70) node {$v_3$};
	
	\draw(2,6) node{$G^2(\kappa_2)$};
	\draw(6,6) node{$G^2(\kappa_1)$};
	\draw(4,6.75) node{$G^1(v)$};
	\draw(2,2) node{$G^2(\kappa_3)$};
	\draw(6,2) node{$G^2(\kappa_4)$};
	\draw(7.25,4.5) node{$G^1(v_1)$};
	
	\draw(4,-0.75) node{(b)};
	
\end{tikzpicture}
\qquad
\begin{tikzpicture}
[scale=0.5]

	\draw[blue, thick](0,0) -- (8,0);
	\draw[blue, thick](0,4) -- (8,4);
	\draw[blue, thick](0,0) -- (0,8);
	\draw[blue, thick](8,0) -- (8,8);
	\draw[blue, thick](4,0) -- (4,8);
	\draw[blue, thick](0,8) -- (8,8);

	\draw (4.30,5.70) node {$v$};		
	\filldraw(4,6) circle (4pt);
	\draw[->] (3.5,6) -- (4.5,6);
	\draw (2.50,4.35) node {$v_2$};	
	\filldraw(2,4) circle (4pt);
	\draw[->] (2,4.5) -- (2,3.5);	
	\draw (6.50,3.60) node {$v_1$};
	\filldraw(6,4) circle (4pt);
	\draw[->] (6,4.5) -- (6,3.5);	
	\draw (4.5,1.70) node {$v_3$};
	\draw(4,2) circle (4pt);

	\draw(5.5,5) node{$\kappa_1$};
	
	\draw(6,4.75) node{$\star$};
	\draw(7.25,4.75) node{$\star$};
	
	\draw(1.25,5) node{$\kappa_2$};
	
	\draw(3.25,4.75) node{$\star$};

	\draw(1.25,1.25) node{$\kappa_3$};
	
	\draw(0.75,2) node{$\star$};
	\draw(2,2) node{$\star$};
	\draw(0.75,0.75) node{$\star$};
	\draw(2,0.75) node{$\star$};
	
	\draw(5.25,1.25) node{$\kappa_4$};
	
	\draw(6,2) node{$\star$};
	\draw(7.25,2) node{$\star$};
	\draw(6,0.75) node{$\star$};
	\draw(7.25,0.75) node{$\star$};

	\draw(4,-0.75) node{(c)};	
\end{tikzpicture}
\caption{(a) Possible cell types for Case 2(b).
(b) Necessary tiles in $N_\cN$ for Case 2(b)(i).
(c) Possible cell types for Case 2(c).
}
\label{fig:Case2b}
\end{figure}
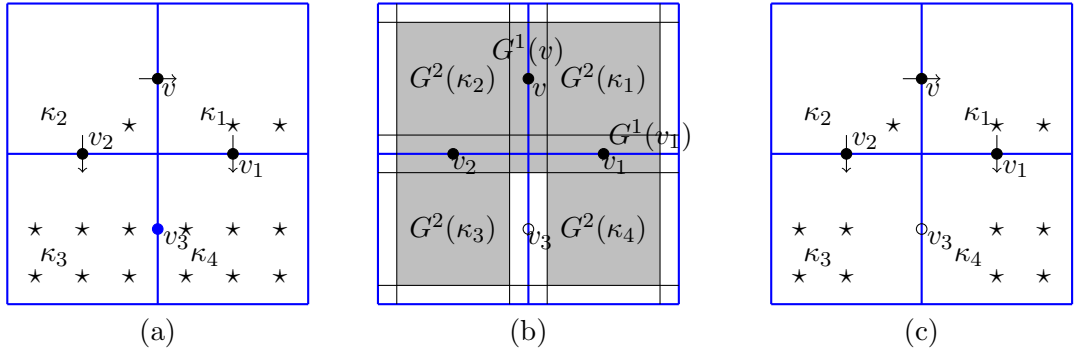	
	
	\item  \emph{Assume $v_2\in\cN^0$ and $v_2$ is transparent south}.
		Figure~\ref{fig:Case2}(a) provides an upper bound on the cell types.
		Since  $v_2$ is transparent south, $\kappa_2$ is of type $SE$ and $\kappa_3$ is of type $W$, $A$, $E$, $SW$, $S$, or $SE$ as indicated in Figure~\ref{fig:Case2b}(a).
		
		By {\bf Rule 1}, $G^2(\kappa_3)\cup G^1(v_2) \subset N_\cN$.
		By {\bf Rule 2}, $G^0(\pi)\subset N_\cN$.
		Thus $N_\cN$ must contain at least the tiles indicated in Figure~\ref{fig:Case2b}(b).
	
		\begin{enumerate}
		\item   \emph{Assume $v_3\not\in\cN^0$}.  
			This implies that $\kappa_3$ cannot be of type $E$ or $SE$ and $\kappa_4$ cannot be of type $W$ or $SW$ (see Figure~\ref{fig:Case2b}(c)).
			Therefore $v_3\in\cW$.
			Thus {\bf Rule 5} applies and $G^1(v_3)\subset N_\cN$.
			Observe that there are no internal edges against which to check internal tangencies.

		\item   \emph{Assume $v_3\in\cN^0$}.  By {\bf Rule 1}, $G^1(v_3) \subset N_\cN$.
		Observe that there are no interior boundary edges to check.
		\end{enumerate}

	\item  \emph{Assume $v_2\in\cN^0$ and $v_2$ is transparent north}.
		Figure~\ref{fig:Case2}(a) provides an upper bound on the cell types.
		Since $v_2$ is transparent north,  $\kappa_3$ must be of type $NE$ and $\kappa_2$ cannot be of type $SE$.
		The adjacency of $\kappa_3$ and $\kappa_4$ implies that $\kappa_4$ cannot be of type $W$ or $SW$.
		See Figure~\ref{fig:Case2c3}(a).
		
		By {\bf Rule 1}, $G^2(\kappa_3)\cup G^1(v_2) \subset N_\cN$.
		By {\bf Rule 2}, $G^0(\pi)\subset N_\cN$.
		The necessary tiles in $N_\cN$ are indicated in Figure~\ref{fig:Case2b}(b).
		
\begin{figure}
\begin{tikzpicture}
[scale=0.5]

	\draw[blue, thick](0,0) -- (8,0);
	\draw[blue, thick](0,4) -- (8,4);
	\draw[blue, thick](0,0) -- (0,8);
	\draw[blue, thick](8,0) -- (8,8);
	\draw[blue, thick](4,0) -- (4,8);
	\draw[blue, thick](0,8) -- (8,8);

	\draw (4.30,5.70) node {$v$};		
	\filldraw(4,6) circle (4pt);
	\draw[->] (3.5,6) -- (4.5,6);
	\draw (2.50,4.35) node {$v_2$};	
	\filldraw(2,4) circle (4pt);
	\draw[->] (2,3.5) -- (2,4.5);	
	\draw (6.50,3.60) node {$v_1$};
	\filldraw(6,4) circle (4pt);
	\draw[->] (6,4.5) -- (6,3.5);	
	\draw (4.5,1.70) node {$v_3$};
	\filldraw[blue](4,2) circle (4pt);

	\draw(5.5,5) node{$\kappa_1$};
	
	\draw(6,4.75) node{$\star$};
	\draw(7.25,4.75) node{$\star$};
	
	\draw(1.25,5) node{$\kappa_2$};
	
	\draw(3.25,7.25) node{$\star$};
	\draw(3.25,6) node{$\star$};

	\draw(1.25,1.25) node{$\kappa_3$};
	
	\draw(3.25,3.25) node{$\star$};
	
	\draw(5.25,1.25) node{$\kappa_4$};
	
	\draw(6,2) node{$\star$};
	\draw(7.25,2) node{$\star$};
	\draw(6,0.75) node{$\star$};
	\draw(7.25,0.75) node{$\star$};

	\draw(4,-0.75) node{(c)};	
\end{tikzpicture}
\qquad
\begin{tikzpicture}
[scale=0.5]

	\draw[blue, thick](0,0) -- (8,0);
	\draw[blue, thick](0,4) -- (8,4);
	\draw[blue, thick](0,0) -- (0,8);
	\draw[blue, thick](8,0) -- (8,8);
	\draw[blue, thick](4,0) -- (4,8);
	\draw[blue, thick](0,8) -- (8,8);

	\draw (4.30,5.70) node {$v$};		
	\filldraw(4,6) circle (4pt);
	\draw[->] (3.5,6) -- (4.5,6);
	\filldraw[blue](2,4) circle (4pt);
	\draw (2.50,4.35) node {$v_2$};
	\draw (6.50,3.60) node {$v_1$};
	\draw(6,4) circle (4pt);
	\draw (4.5,1.70) node {$v_3$};
	\filldraw[blue](4,2) circle (4pt);	
	
	\draw(5.5,5) node{$\kappa_1$};
	
	\draw(6,7.25) node{$\star$};
	\draw(7.25,7.25) node{$\star$};	
	\draw(6,6) node{$\star$};
	\draw(7.25,6) node{$\star$};	
	
	\draw(1.25,5) node{$\kappa_2$};
	
	\draw(3.25,7.25) node{$\star$};
	\draw(3.25,6) node{$\star$};
	\draw(3.25,4.75) node{$\star$};

	\draw(1.25,1.25) node{$\kappa_3$};
	
	\draw(0.75,3.25) node{$\star$};
	\draw(2,3.25) node{$\star$};
	\draw(3.25,3.25) node{$\star$};
	\draw(0.75,2) node{$\star$};
	\draw(2,2) node{$\star$};
	\draw(3.25,2) node{$\star$};
	\draw(0.75,0.75) node{$\star$};
	\draw(2,0.75) node{$\star$};
	\draw(3.25,0.75) node{$\star$};
	
	\draw(5.25,1.25) node{$\kappa_4$};
	
	\draw(4.75,3.25) node{$\star$};
	\draw(6,3.25) node{$\star$};
	\draw(7.25,3.25) node{$\star$};
	\draw(6,2) node{$\star$};
	\draw(7.25,2) node{$\star$};
	\draw(6,0.75) node{$\star$};
	\draw(7.25,0.75) node{$\star$};

	\draw(4,-0.75) node{(b)};	
\end{tikzpicture}
\quad
\begin{tikzpicture}
[scale=0.5]

	\draw[blue, thick](0,0) -- (8,0);
	\draw[blue, thick](0,4) -- (8,4);
	\draw[blue, thick](0,0) -- (0,8);
	\draw[blue, thick](8,0) -- (8,8);
	\draw[blue, thick](4,0) -- (4,8);
	\draw[blue, thick](0,8) -- (8,8);

	\draw(4.30,5.70) node {$v$};		
	\filldraw(4,6) circle (4pt);
	\draw[->] (3.5,6) -- (4.5,6);
	\draw(2.50,4.35) node {$v_2$};	
	\draw(2,4) circle (4pt);

	\draw (6.50,3.60) node {$v_1$};
	\draw(6,4) circle (4pt);
	\draw (4.5,1.70) node {$v_3$};
	\filldraw[blue](4,2) circle (4pt);	
	
	\draw(5.5,5) node{$\kappa_1$};
	
	\draw(6,7.25) node{$\star$};
	\draw(7.25,7.25) node{$\star$};	
	\draw(6,6) node{$\star$};
	\draw(7.25,6) node{$\star$};	
	
	\draw(1.25,5) node{$\kappa_2$};
	
	\draw(3.25,7.25) node{$\star$};
	\draw(3.25,6) node{$\star$};

	\draw(1.25,1.25) node{$\kappa_3$};
	
	\draw(0.75,3.25) node{$\star$};
	\draw(2,3.25) node{$\star$};
	\draw(3.25,3.25) node{$\star$};
	\draw(3.25,2) node{$\star$};
	\draw(3.25,0.75) node{$\star$};
	
	\draw(5.25,1.25) node{$\kappa_4$};
	
	\draw(4.75,3.25) node{$\star$};
	\draw(6,3.25) node{$\star$};
	\draw(7.25,3.25) node{$\star$};
	\draw(6,2) node{$\star$};
	\draw(7.25,2) node{$\star$};
	\draw(6,0.75) node{$\star$};
	\draw(7.25,0.75) node{$\star$};

	\draw(4,-0.75) node{(c)};	
\end{tikzpicture}
\caption{(a) Possible cell types for Case 2(c).
(b) Possible cell types for Case 3.
(c) Possible cell types for Case 3(a).
}
\label{fig:Case2c3}
\end{figure}
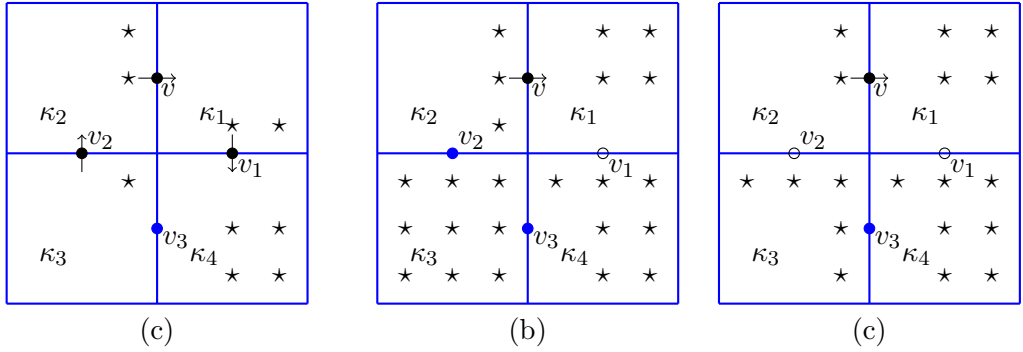	
		
		\begin{enumerate}
		\item   \emph{Assume $v_3\not\in\cN^0$}.  
		Because $v_2\in\cN^0$ and $\kappa_3$ is of type $NE$, by Proposition~\ref{prop:cellAtt_Rev} $v_3\in\cN^0$, a contradiction.
		Thus this case cannot occur.
		
		\item   \emph{Assume $v_3\in\cA$}.
		By {\bf Rule 1}, $G^1(v_3) \subset N_\cN$.
		Observe that there are no interior boundary edges to check.
		\end{enumerate}
	\end{enumerate}

\item {\em Assume $v_1\not\in\cN^0$.}  
Figure~\ref{fig:fourCellsNew}(a) provides an upper bound on the cell types.
Since $v_1\not\in\cN^0$ there is no edge $v\to v_1$.
Therefore  the cell $\kappa_1$  is of type $A$,  $NE$, $N$ or $E$ and $v_1\in \cV_e(\kappa_1)$.
Adjacency of $\kappa_1$ and $\kappa_4$ implies that $\kappa_4$ is not of type $W$ or $SW$.
See Figure~\ref{fig:Case2c3}(b).

	\begin{enumerate}
	\item \emph{Assume $v_2\not\in\cN^0$}. 
		Figure~\ref{fig:Case2c3}(b) provides an upper bound on the cell types.
		By Proposition~\ref{prop:cellAtt_Rev}  we conclude that $\kappa_2$ is of type $NE$ or $E$.
		By Proposition~\ref{prop:adjacentCells}(i) or (iv) the adjacency of $\kappa_2$ and $\kappa_3$ implies that $\kappa_3$ must be of type $NW$, $N$, $NE$, $E$ or $SE$.
		See Figure~\ref{fig:Case2c3}(c).
		The necessary set of tiles in $N_\cN$ is given by Figure~\ref{fig:MinimalFourCells}(b).
		
		\begin{enumerate}
		\item \emph{Assume $v_3\not\in\cN^0$.}
		\begin{enumerate}
		\item \emph{Assume $G^2(\kappa_3)\not\subset N_\cN$ and $G^2(\kappa_4)\not\subset N_\cN$.}
		By Proposition~\ref{prop:looseG^1}, $G^1(v_n)\not\subset N_\cN$, $n=1,2,3$.
		{\bf Rule 2} is not applicable thus, $G^0(\pi)\not\subset N_\cN$.
		Thus the set of tiles is not changed from that of Figure~\ref{fig:MinimalFourCells}(b).
		By Proposition~\ref{prop:TG2}  the lower face of $G^2_2$ and by Proposition~\ref{prop:TG1j} the lower left face of $G^1_v$ is transverse in.

		\item \emph{Assume $G^2(\kappa_3)\subset N_\cN$.}
		By Figure~\ref{fig:Case2c3}(c) $\kappa_3$ is not of type $A$.  
		Thus there exists $u\in \cV(\kappa_3)\cap \cA \cap \cT$.
		By Proposition~\ref{prop:uvwCell} there exists $u\in\cV_e(\kappa_3)\cap \cN^0$. 
		By assumption $u\in\cV(\kappa_3)\setminus\setof{v_2,v_3}$.
		We leave it to the reader to check that the given the possible types of $\kappa_3$ by Proposition~\ref{prop:cellAtt_Rev}, either  $v_2$ or $v_3$ is in $\cA$, a contradiction.
		Thus, it is not possible for $G^2(\kappa_3)\subset N_\cN$.
				
		\item \emph{Assume $G^2(\kappa_4)\subset N_\cN$.}
		By Case 3(a)(i)(B), $G^2(\kappa_3)\not\subset N_\cN$.
		Therefore, by Proposition~\ref{prop:looseG^1}, $G^1(v_3)\not\subset N_\cN$.
		We are in the setting of  Figure~\ref{fig:fourCells2}(a).

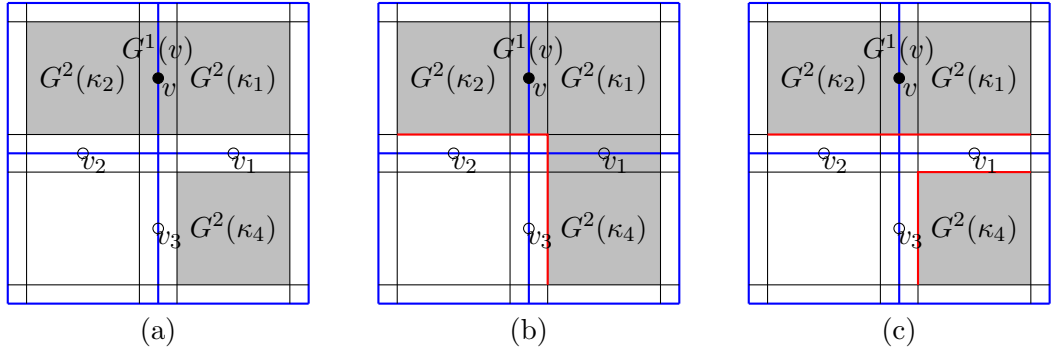
\begin{figure}
\begin{tikzpicture}
[scale=0.5]

	\fill[lightgray] (0.5,4.5) -- (3.5,4.5) -- (3.5,7.5) -- (0.5,7.5) -- (0.5,4.5);
	\fill[lightgray] (4.5,4.5) -- (7.5,4.5) -- (7.5,7.5) -- (4.5,7.5) -- (4.5,4.5);
	\fill[lightgray] (3.5,4.5) -- (4.5,4.5) -- (4.5,7.5) -- (3.5,7.5) -- (3.5,4.5);
	\fill[lightgray] (4.5,0.5) -- (4.5,3.5) -- (7.5,3.5) -- (7.5,0.5) -- (4.5,0.5);
		
	\draw[blue, thick](0,0) -- (8,0);
	\draw[blue, thick](0,4) -- (8,4);
	\draw[blue, thick](0,0) -- (0,8);
	\draw[blue, thick](8,0) -- (8,8);
	\draw[blue, thick](4,0) -- (4,8);
	\draw[blue, thick](0,8) -- (8,8);
		
	\draw (0,0.5) -- (8,0.5);
	\draw (0,3.5) -- (8,3.5);
	\draw (0,4.5) -- (8,4.5);
	\draw (0,7.5) -- (8,7.5);
	\draw (0.5,0) -- (0.5,8);
	\draw (3.5,0) -- (3.5,8);
	\draw (4.5,0) -- (4.5,8);
	\draw (7.5,0) -- (7.5,8);
	
	
		
	\filldraw(4,6) circle (4pt);
	\draw (4.30,5.70) node {$v$};
	\draw(2,4) circle (4pt);
	\draw (2.30,3.70) node {$v_2$};
	\draw(6,4) circle (4pt);
	\draw (6.30,3.70) node {$v_1$};
	\draw(4,2) circle (4pt);
	\draw (4.30,1.70) node {$v_3$};
	
	\draw(2,6) node{$G^2(\kappa_2)$};
	\draw(6,6) node{$G^2(\kappa_1)$};
	\draw(6,2) node{$G^2(\kappa_4)$};
	\draw(4,6.75) node{$G^1(v)$};
	
	\draw(4,-0.75) node{(a)};	
\end{tikzpicture}
\qquad
\begin{tikzpicture}
[scale=0.5]

	\fill[lightgray] (0.5,4.5) -- (3.5,4.5) -- (3.5,7.5) -- (0.5,7.5) -- (0.5,4.5);
	\fill[lightgray] (4.5,4.5) -- (7.5,4.5) -- (7.5,7.5) -- (4.5,7.5) -- (4.5,4.5);
	\fill[lightgray] (3.5,4.5) -- (4.5,4.5) -- (4.5,7.5) -- (3.5,7.5) -- (3.5,4.5);
	\fill[lightgray] (4.5,0.5) -- (4.5,3.5) -- (7.5,3.5) -- (7.5,0.5) -- (4.5,0.5);
	
	\fill[lightgray] (4.5,3.5) -- (4.5,4.5) -- (7.5,4.5) -- (7.5,3.5) -- (4.5,3.5);
		
	\draw[blue, thick](0,0) -- (8,0);
	\draw[blue, thick](0,4) -- (8,4);
	\draw[blue, thick](0,0) -- (0,8);
	\draw[blue, thick](8,0) -- (8,8);
	\draw[blue, thick](4,0) -- (4,8);
	\draw[blue, thick](0,8) -- (8,8);
		
	\draw (0,0.5) -- (8,0.5);
	\draw (0,3.5) -- (8,3.5);
	\draw (0,4.5) -- (8,4.5);
	\draw (0,7.5) -- (8,7.5);
	\draw (0.5,0) -- (0.5,8);
	\draw (3.5,0) -- (3.5,8);
	\draw (4.5,0) -- (4.5,8);
	\draw (7.5,0) -- (7.5,8);
	
	\draw[red, thick](0.5,4.5) --  (4.5,4.5) -- (4.5,0.5);

		
	\filldraw(4,6) circle (4pt);
	\draw (4.30,5.70) node {$v$};
	\draw(2,4) circle (4pt);
	\draw (2.30,3.70) node {$v_2$};
	\draw(6,4) circle (4pt);
	\draw (6.30,3.70) node {$v_1$};
	\draw(4,2) circle (4pt);
	\draw (4.30,1.70) node {$v_3$};
	
	\draw(2,6) node{$G^2(\kappa_2)$};
	\draw(6,6) node{$G^2(\kappa_1)$};
	\draw(6,2) node{$G^2(\kappa_4)$};
	\draw(4,6.75) node{$G^1(v)$};
	
	\draw(4,-0.75) node{(b)};	
\end{tikzpicture}
\qquad
\begin{tikzpicture}
[scale=0.5]

	\fill[lightgray] (0.5,4.5) -- (3.5,4.5) -- (3.5,7.5) -- (0.5,7.5) -- (0.5,4.5);
	\fill[lightgray] (4.5,4.5) -- (7.5,4.5) -- (7.5,7.5) -- (4.5,7.5) -- (4.5,4.5);
	\fill[lightgray] (3.5,4.5) -- (4.5,4.5) -- (4.5,7.5) -- (3.5,7.5) -- (3.5,4.5);
	\fill[lightgray] (4.5,0.5) -- (4.5,3.5) -- (7.5,3.5) -- (7.5,0.5) -- (4.5,0.5);
	
		
	\draw[blue, thick](0,0) -- (8,0);
	\draw[blue, thick](0,4) -- (8,4);
	\draw[blue, thick](0,0) -- (0,8);
	\draw[blue, thick](8,0) -- (8,8);
	\draw[blue, thick](4,0) -- (4,8);
	\draw[blue, thick](0,8) -- (8,8);
		
	\draw (0,0.5) -- (8,0.5);
	\draw (0,3.5) -- (8,3.5);
	\draw (0,4.5) -- (8,4.5);
	\draw (0,7.5) -- (8,7.5);
	\draw (0.5,0) -- (0.5,8);
	\draw (3.5,0) -- (3.5,8);
	\draw (4.5,0) -- (4.5,8);
	\draw (7.5,0) -- (7.5,8);
	
	\draw[red, thick](0.5,4.5) --  (7.5,4.5);
	\draw[red, thick](7.5,3.5) --  (4.5,3.5) -- (4.5,0.5);
	
		
	\filldraw(4,6) circle (4pt);
	\draw (4.30,5.70) node {$v$};
	\draw(2,4) circle (4pt);
	\draw (2.30,3.70) node {$v_2$};
	\draw(6,4) circle (4pt);
	\draw (6.30,3.70) node {$v_1$};
	\draw(4,2) circle (4pt);
	\draw (4.30,1.70) node {$v_3$};
	
	\draw(2,6) node{$G^2(\kappa_2)$};
	\draw(6,6) node{$G^2(\kappa_1)$};
	\draw(6,2) node{$G^2(\kappa_4)$};
	\draw(4,6.75) node{$G^1(v)$};
	
	\draw(4,-0.75) node{(c)};	
\end{tikzpicture}
\caption{(a) Necessary grid elements for Case 3(a)(i)(C).
(b) Tiles in Case  3(a)(i)(C)(I).
(c) Tiles in Case  3(a)(i)(C)(II). 
}
\label{fig:fourCells2}
\end{figure}
					 
			\begin{itemize}
			\item[(I)] \emph{Assume $G^1(v_1)\subset N_\cN$.}
			{\bf Rule 2} does not apply to $\kappa_1$, hence $G^0(\pi)\not\subset N_\cN$.
			We are in the setting of  Figure~\ref{fig:fourCells2}(b).
			The desired transversality follows from Propositions~\ref{prop:TG2}, \ref{prop:TG1j}, \ref{prop:TG1i}, and \ref{prop:TG0}.
			\item[(II)] \emph{Assume $G^1(v_1)\not\subset N_\cN$.}
			Assume $\kappa_4$ is not of type $A$.  
			Thus there exists $u\in \cV(\kappa_4)\cap \cN^0 \cap \cT$.
			By Proposition~\ref{prop:uvwCell} there exists $u\in\cV_e(\kappa_4)\cap \cN^0$. 
			By assumption 
			$u\in\cV(\kappa_4)\setminus\setof{v_1,v_3}$

			We leave it to the reader to check that if $\kappa_4$ in not of type $A$, then by Proposition~\ref{prop:cellAtt_Rev}, either  $v_1$ or $v_3$ is in $\cN^0$, a contradiction.
			Thus, if $G^2(\kappa_4)\subset N_\cN$, then $\kappa_4$ is of type $A$ and we are in the setting of Figure~\ref{fig:fourCells2}(c) and the desired transversality follows from Propositions~\ref{prop:TG2} and \ref{prop:TG1j}. 
			\end{itemize}
		\end{enumerate}

		\item \emph{Assume $v_3\in\cN^0$ and $v_3$ is transparent east}.
			Figure~\ref{fig:Case2c3}(c) provides an upper bound on the possible cell types.
			Observe that $\kappa_3$ cannot be of type $NW$ or $N$ and $\kappa_4$ cannot be of type $NW$.
			Furthermore, by Proposition~\ref{prop:cellAtt_Rev} $\kappa_3$ cannot be of type $NE$ and since the edge $v_3\to v_1$ cannot exist $\kappa_4$ cannot be of type $N$ or $NE$.
			See Figure~\ref{fig:Case3}(a).
			By {\bf Rule 1} $G^2(\kappa_3)\cup G^2(\kappa_4) \cup G^1(v_3)\subset N_\cN$.
			See Figure~\ref{fig:Case3}(b).			
			The existence of $G^0(\pi)$ is determined by {\bf Rule 2}, hence a necessary condition for $G^0(\pi)\subset N_\cN$ is the existence of $G^1(v_1)\subset N_\cN$ or $G^1(v_2)\subset N_\cN$.
			\begin{enumerate}
				\item \emph{Assume $G^1(v_1)\not\subset N_\cN$ and $G^1(v_2)\not\subset N_\cN$.}
				The desired transversality follows from Propositions~\ref{prop:TG2} and \ref{prop:TG1j}. 
								
				\item \emph{Assume $G^1(v_2)\subset N_\cN$.}
				Observe that $v\in\cV_a(\kappa_2)$. Thus {\bf Rule 2} applied to $\kappa_2$ implies that $G^0(\pi)\subset N_\cA$.
				{\bf Rule 5} applies and hence $G^1(v_1)\subset N_\cN$.
				There are no interior boundary edges to check.
								
				\item \emph{Assume $G^1(v_2)\not\subset N_\cN$ and $G^1(v_1)\subset N_\cN$.}
				From Figure~\ref{fig:Case3}(a) we deduce that {\bf Rule 2} does not apply to  $\kappa_1$ or $\kappa_4$.  
				Thus $G^0(\pi)\not\subset N_\cN$.
				This puts us into the setting shown in Figure~\ref{fig:Case3}(c) and the desired transversality follows from Propositions~\ref{prop:TG2}, \ref{prop:TG1j}, and \ref{prop:TG1i}. 
			\end{enumerate}

		\item \emph{Assume $v_3\in\cN^0$ and $v_3$ is transparent west}.
			Figure~\ref{fig:Case2c3}(c) provides an upper bound on the possible cell types.
			Observe that $\kappa_3$ cannot be of type $NE$, $E$, or $SE$ and $\kappa_4$ must be of type $NW$.
			Thus, we are in the setting of Figure~\ref{fig:Case3*}(a).
			By Proposition~\ref{prop:adjacentCells} the cell types of $\kappa_2$ and $\kappa_3$ are not compatible, a contradiction.
			Thus this case cannot occur.
		\end{enumerate}

\begin{figure}
\begin{tikzpicture}
[scale=0.5]

	\draw[blue, thick](0,0) -- (8,0);
	\draw[blue, thick](0,4) -- (8,4);
	\draw[blue, thick](0,0) -- (0,8);
	\draw[blue, thick](8,0) -- (8,8);
	\draw[blue, thick](4,0) -- (4,8);
	\draw[blue, thick](0,8) -- (8,8);

	\draw(4.30,5.70) node {$v$};		
	\filldraw(4,6) circle (4pt);
	\draw[->] (3.5,6) -- (4.5,6);
	\draw(2.50,4.35) node {$v_2$};	
	\draw(2,4) circle (4pt);

	\draw (6.50,3.60) node {$v_1$};
	\draw(6,4) circle (4pt);
	\draw (4.5,1.70) node {$v_3$};
	\filldraw(4,2) circle (4pt);	
	\draw[->] (3.5,2) -- (4.5,2);
		
	\draw(5.5,5) node{$\kappa_1$};
	
	\draw(6,7.25) node{$\star$};
	\draw(7.25,7.25) node{$\star$};	
	\draw(6,6) node{$\star$};
	\draw(7.25,6) node{$\star$};	
	
	\draw(1.25,5) node{$\kappa_2$};
	
	\draw(3.25,7.25) node{$\star$};
	\draw(3.25,6) node{$\star$};

	\draw(1.25,1.25) node{$\kappa_3$};
	
	\draw(3.25,2) node{$\star$};
	\draw(3.25,0.75) node{$\star$};
	
	\draw(5.25,1.25) node{$\kappa_4$};
	
	\draw(6,2) node{$\star$};
	\draw(7.25,2) node{$\star$};
	\draw(6,0.75) node{$\star$};
	\draw(7.25,0.75) node{$\star$};

	\draw(4,-0.75) node{(a)};	
\end{tikzpicture}
\qquad
\begin{tikzpicture}
[scale=0.5]

	\fill[lightgray] (0.5,4.5) -- (3.5,4.5) -- (3.5,7.5) -- (0.5,7.5) -- (0.5,4.5);
	\fill[lightgray] (4.5,4.5) -- (7.5,4.5) -- (7.5,7.5) -- (4.5,7.5) -- (4.5,4.5);
	\fill[lightgray] (3.5,4.5) -- (4.5,4.5) -- (4.5,7.5) -- (3.5,7.5) -- (3.5,4.5);
	\fill[lightgray] (4.5,0.5) -- (7.5,0.5) -- (7.5,3.5) -- (4.5,3.5) -- (4.5,0.5);
	\fill[lightgray] (3.5,3.5) -- (4.5,3.5) -- (4.5,0.5) -- (3.5,0.5) -- (3.5,3.5);
	\fill[lightgray] (3.5,3.5) -- (3.5,0.5) -- (0.5,0.5) -- (0.5,3.5) -- (3.5,3.5);

	\draw[blue, thick](0,0) -- (8,0);
	\draw[blue, thick](0,4) -- (8,4);
	\draw[blue, thick](0,0) -- (0,8);
	\draw[blue, thick](8,0) -- (8,8);
	\draw[blue, thick](4,0) -- (4,8);
	\draw[blue, thick](0,8) -- (8,8);
	
	\draw (0,0.5) -- (8,0.5);
	\draw (0,3.5) -- (8,3.5);
	\draw (0,4.5) -- (8,4.5);
	\draw (0,7.5) -- (8,7.5);
	\draw (0.5,0) -- (0.5,8);
	\draw (3.5,0) -- (3.5,8);
	\draw (4.5,0) -- (4.5,8);
	\draw (7.5,0) -- (7.5,8);
		
	\filldraw(4,6) circle (4pt);
	\draw (4.30,5.70) node {$v$};
	\draw(2.50,4.35) node {$v_2$};	
	\draw(2,4) circle (4pt);

	\draw (6.50,3.60) node {$v_1$};
	\draw(6,4) circle (4pt);
	\filldraw(4,2) circle (4pt);
	\draw (4.30,1.70) node {$v_3$};
	
	\draw(2,6) node{$G^2(\kappa_2)$};
	\draw(6,6) node{$G^2(\kappa_1)$};
	\draw(4,6.75) node{$G^1(v)$};
	\draw(2,2) node{$G^2(\kappa_3)$};
	\draw(4,2.75) node{$G^1(v_3)$};
	\draw(6,2) node{$G^2(\kappa_4)$};
	
	\draw(4,-0.75) node{(b)};	
\end{tikzpicture}
\qquad
\begin{tikzpicture}
[scale=0.5]

	\fill[lightgray] (0.5,4.5) -- (3.5,4.5) -- (3.5,7.5) -- (0.5,7.5) -- (0.5,4.5);
	\fill[lightgray] (4.5,4.5) -- (7.5,4.5) -- (7.5,7.5) -- (4.5,7.5) -- (4.5,4.5);
	\fill[lightgray] (3.5,4.5) -- (4.5,4.5) -- (4.5,7.5) -- (3.5,7.5) -- (3.5,4.5);
	\fill[lightgray] (4.5,3.5) -- (7.5,3.5) -- (7.5,4.5) -- (4.5,4.5) -- (4.5,4.5);
	\fill[lightgray] (4.5,0.5) -- (7.5,0.5) -- (7.5,3.5) -- (4.5,3.5) -- (4.5,0.5);
	\fill[lightgray] (3.5,3.5) -- (4.5,3.5) -- (4.5,0.5) -- (3.5,0.5) -- (3.5,3.5);
	\fill[lightgray] (3.5,3.5) -- (3.5,0.5) -- (0.5,0.5) -- (0.5,3.5) -- (3.5,3.5);

	\draw[blue, thick](0,0) -- (8,0);
	\draw[blue, thick](0,4) -- (8,4);
	\draw[blue, thick](0,0) -- (0,8);
	\draw[blue, thick](8,0) -- (8,8);
	\draw[blue, thick](4,0) -- (4,8);
	\draw[blue, thick](0,8) -- (8,8);
	
	\draw (0,0.5) -- (8,0.5);
	\draw (0,3.5) -- (8,3.5);
	\draw (0,4.5) -- (8,4.5);
	\draw (0,7.5) -- (8,7.5);
	\draw (0.5,0) -- (0.5,8);
	\draw (3.5,0) -- (3.5,8);
	\draw (4.5,0) -- (4.5,8);
	\draw (7.5,0) -- (7.5,8);
		
	\filldraw(4,6) circle (4pt);
	\draw (4.30,5.70) node {$v$};
	\draw(2.50,4.35) node {$v_2$};	
	\draw(2,4) circle (4pt);

	\draw (6.50,3.60) node {$v_1$};
	\draw(6,4) circle (4pt);
	\filldraw(4,2) circle (4pt);
	\draw (4.30,1.70) node {$v_3$};
	
	\draw[red, thick](0.5,4.5) --  (4.5,4.5) -- (4.5,3.5) -- (0.5,3.5);	
	
	\draw(2,6) node{$G^2(\kappa_2)$};
	\draw(6,6) node{$G^2(\kappa_1)$};
	\draw(4,6.75) node{$G^1(v)$};
	\draw(2,2) node{$G^2(\kappa_3)$};
	\draw(4,2.75) node{$G^1(v_3)$};
	\draw(6,2) node{$G^2(\kappa_4)$};
	
	\draw(4,-0.75) node{(c)};	
\end{tikzpicture}
\caption{(a) Possible cell types in Case 3(a)(ii).
(b) Possible cell types in Case 3(a)(iii).
(c) Possible cell types in Case 3(b).
}
\label{fig:Case3}
\end{figure}
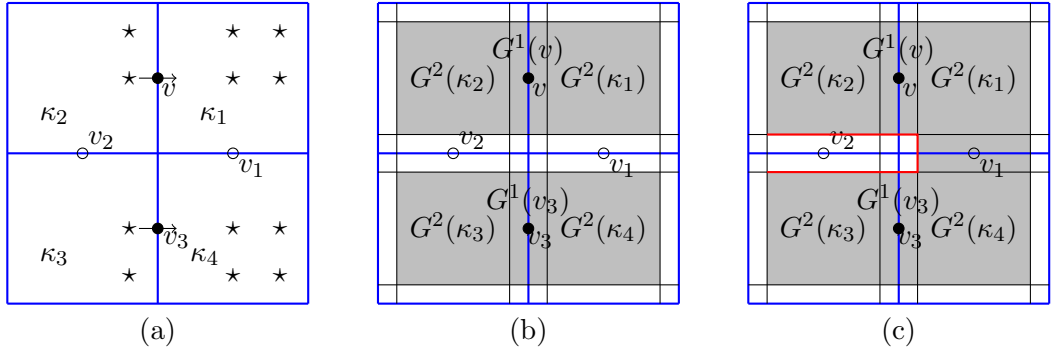	

\begin{figure}
\begin{tikzpicture}
[scale=0.5]

	\draw[blue, thick](0,0) -- (8,0);
	\draw[blue, thick](0,4) -- (8,4);
	\draw[blue, thick](0,0) -- (0,8);
	\draw[blue, thick](8,0) -- (8,8);
	\draw[blue, thick](4,0) -- (4,8);
	\draw[blue, thick](0,8) -- (8,8);

	\draw(4.30,5.70) node {$v$};		
	\filldraw(4,6) circle (4pt);
	\draw[->] (3.5,6) -- (4.5,6);
	\draw(2.50,4.35) node {$v_2$};	
	\draw(2,4) circle (4pt);

	\draw (6.50,3.60) node {$v_1$};
	\draw(6,4) circle (4pt);
	\draw (4.5,1.70) node {$v_3$};
	\filldraw(4,2) circle (4pt);	
	\draw[->] (4.5,2) -- (3.5,2);
	
	\draw(5.5,5) node{$\kappa_1$};
	
	\draw(6,7.25) node{$\star$};
	\draw(7.25,7.25) node{$\star$};	
	\draw(6,6) node{$\star$};
	\draw(7.25,6) node{$\star$};	
	
	\draw(1.25,5) node{$\kappa_2$};
	
	\draw(3.25,7.25) node{$\star$};
	\draw(3.25,6) node{$\star$};

	\draw(1.25,1.25) node{$\kappa_3$};
	
	\draw(0.75,3.25) node{$\star$};
	\draw(2,3.25) node{$\star$};
	
	\draw(5.25,1.25) node{$\kappa_4$};
	
	\draw(4.75,3.25) node{$\star$};

	\draw(4,-0.75) node{(a)};	
\end{tikzpicture}
\qquad
\begin{tikzpicture}
[scale=0.5]

	\draw[blue, thick](0,0) -- (8,0);
	\draw[blue, thick](0,4) -- (8,4);
	\draw[blue, thick](0,0) -- (0,8);
	\draw[blue, thick](8,0) -- (8,8);
	\draw[blue, thick](4,0) -- (4,8);
	\draw[blue, thick](0,8) -- (8,8);

	\draw (4.30,5.70) node {$v$};		
	\filldraw(4,6) circle (4pt);
	\draw[->] (3.5,6) -- (4.5,6);

	\draw (2.50,4.35) node {$v_2$};
	\filldraw(2,4) circle (4pt);	
	\draw[->] (2,3.5) -- (2,4.5);
	
	\draw (6.50,3.60) node {$v_1$};
	\draw(6,4) circle (4pt);
	\draw (4.5,1.70) node {$v_3$};
	\filldraw[blue](4,2) circle (4pt);	
	
	\draw(5.5,5) node{$\kappa_1$};
	
	\draw(6,7.25) node{$\star$};
	\draw(7.25,7.25) node{$\star$};	
	\draw(6,6) node{$\star$};
	\draw(7.25,6) node{$\star$};	
	
	\draw(1.25,5) node{$\kappa_2$};
	
	\draw(3.25,7.25) node{$\star$};
	\draw(3.25,6) node{$\star$};

	\draw(1.25,1.25) node{$\kappa_3$};
	
	\draw(0.75,3.25) node{$\star$};
	\draw(2,3.25) node{$\star$};
	\draw(3.25,3.25) node{$\star$};
	
	\draw(5.25,1.25) node{$\kappa_4$};
	
	\draw(4.75,3.25) node{$\star$};
	\draw(6,3.25) node{$\star$};
	\draw(7.25,3.25) node{$\star$};
	\draw(6,2) node{$\star$};
	\draw(7.25,2) node{$\star$};
	\draw(6,0.75) node{$\star$};
	\draw(7.25,0.75) node{$\star$};

	\draw(4,-0.75) node{(b)};	
\end{tikzpicture}
\qquad
\begin{tikzpicture}
[scale=0.5]

	\fill[lightgray] (0.5,4.0) -- (3.5,4.0) -- (3.5,7.5) -- (0.5,7.5) -- (0.5,4.0);
	\fill[lightgray] (4.5,4.5) -- (7.5,4.5) -- (7.5,7.5) -- (4.5,7.5) -- (4.5,4.5);
	\fill[lightgray] (0.5,0.5) -- (3.5,0.5) -- (3.5,3.5) -- (0.5,3.5) -- (0.5,0.5);
	\fill[lightgray] (0.5,3.5) -- (3.5,3.5) -- (3.5,4.5) -- (0.5,4.5) -- (0.5,3.5);
	\fill[lightgray] (3.5,4.5) -- (4.5,4.5) -- (4.5,7.5) -- (3.5,7.5) -- (3.5,4.5);
	\fill[lightgray] (3.5,3.5) -- (4.5,3.5) -- (4.5,4.5) -- (3.5,4.5) -- (3.5,3.5);
	
	\draw[blue, thick](0,0) -- (8,0);
	\draw[blue, thick](0,4) -- (8,4);
	\draw[blue, thick](0,0) -- (0,8);
	\draw[blue, thick](8,0) -- (8,8);
	\draw[blue, thick](4,0) -- (4,8);
	\draw[blue, thick](0,8) -- (8,8);
	
	\draw (0,0.5) -- (8,0.5);
	\draw (0,3.5) -- (8,3.5);
	\draw (0,4.5) -- (8,4.5);
	\draw (0,7.5) -- (8,7.5);
	\draw (0.5,0) -- (0.5,8);
	\draw (3.5,0) -- (3.5,8);
	\draw (4.5,0) -- (4.5,8);
	\draw (7.5,0) -- (7.5,8);
	
		
	\filldraw(4,6) circle (4pt);
	\draw (4.30,5.70) node {$v$};
	\filldraw(2,4) circle (4pt);
	\draw (2.50,4.35) node {$v_2$};
	\draw(6,4) circle (4pt);
	\draw (6.50,3.60) node {$v_1$};

	\draw (4.5,1.70) node {$v_3$};
	\filldraw[blue](4,2) circle (4pt);
	
	\draw(2,6) node{$G^2(\kappa_2)$};
	\draw(6,6) node{$G^2(\kappa_1)$};
	\draw(2,2) node{$G^2(\kappa_3)$};
	
	\draw(4,-0.75) node{(c)};	
\end{tikzpicture}
\caption{(a) Possible cell types in Case 3(a)(iii).
(b) Possible cell types in Case 3(b).
(c) Necessary grid elements in $N_\cN$ in Case 3(b).
}
\label{fig:Case3*}
\end{figure}
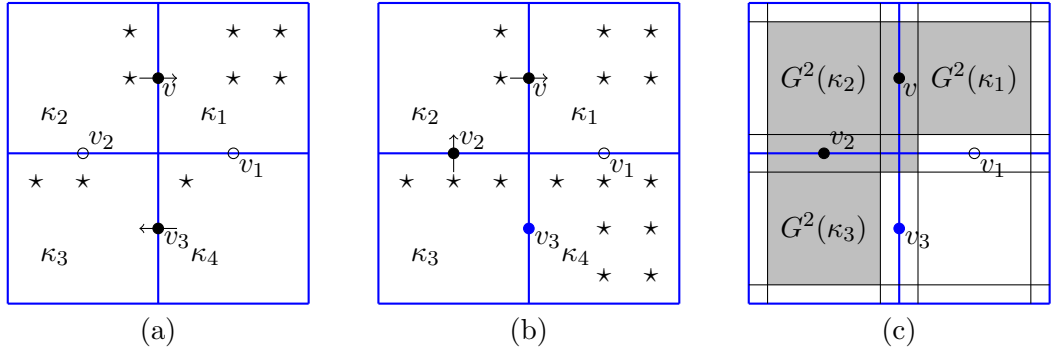

	\item \emph{Assume $v_2\in\cN^0$ and $v_2$ is transparent north}.
		Figure~\ref{fig:Case2c3}(b) provides an upper bound on the possible cell types.
		The assumption that $v_2\in\cN^0$ forces $\kappa_2$ to be of type $NE$ or $E$, and $\kappa_3$ to be of type $NE$, $N$ or $NE$.
		By {\bf Rule 1}, $G^2(\kappa_3) \cup G^1(v_2)\subset N_\cN$.
		By {\bf Rule 2} applied to $\kappa_2$, $G^0(\pi)\subset N_\cN$.
		Thus we are in the setting of Figure~\ref{fig:Case3*}(c).  
	
		\begin{enumerate}
		\item  \emph{Assume $v_3\not\in\cN^0$}.  
			This implies that $\kappa_3$ is of type $A$, $W$, $S$, or $SW$ and hence $v_2,v_3\in\cV_e(\kappa_3)$.
			Since $G^2(\kappa_3)\cup G^1(v_2) \cup G^0(i,j)\subset N_\cN$, {\bf Rule 3} applies and hence $C^n(\kappa_3,v_3,\pi)\subset N_\cN$.
			Similarly, $C^n(\kappa_1,v_1,\pi)\subset N_\cN$.
			By Proposition~\ref{prop:adjacentCells} $\kappa_4$ is of type $NW$.
			This implies that $v_1,v_3\in\cV_a(\kappa_4)$
			Thus {\bf Rule 4} does not apply to $\kappa_4$.
			Figure~\ref{fig:Case3*}(c) indicates $N_\cN\cap E(i,j)$.
			That the vector field is transverse in at all the interior boundary edges (beginning with the top right and moving down to the bottom left) follows from Propositions~\ref{prop:TG2}, \ref{prop:TC}, \ref{prop:TG0}, \ref{prop:TG0}, \ref{prop:TC}, and \ref{prop:TG2}.

		\item  \emph{Assume $v_3\in\cN^0$ and $v_3$ is transparent east}. 
			Observe that rotating clockwise by $90^\circ$ results in Case 2(b)(i) for which the desired transversality has been demonstrated.

												
		\item  \emph{Assume $v_3\in\cN^0$ and $v_3$ is transparent west}. 
			Observe that rotating clockwise by $90^\circ$ results in Case 2(c)(i), which has already been shown not to occur.
		\end{enumerate}		

	\item \emph{Assume $v_2\in\cN^0$ and $v_2$ is transparent south}.  
		We begin by rotating clockwise by $90^\circ$ and the performing a reflection in the east-west direction.
		Figure~\ref{fig:MinimalFourCells}(a) provides an upper bound on the possible cell types.
		Since $v_2$ is transparent south, $\kappa_2$ is of type $SE$ and $\kappa_3$ cannot be of type $NW$, $N$, or $NE$.
		The assumption that $v_3\not\in\cN^0$ implies that $\kappa_3$ cannot be of type $E$ or $SE$
		This results in cell types as shown in Figure~\ref{fig:case3c}(a).  
		{\bf Rule 1} implies that $G^2(\kappa_3)\cup G^1(v_2)\subset N_\cN$.
		{\bf Rule 2} implies that $G^0(\pi)\subset N_\cN$.
		Thus the necessary set of tiles in $N_\cN$ is shown in  Figure~\ref{fig:case3c}(b).  
		
		\begin{enumerate}
			\item \emph{Assume $v_1\in\cN^0$ and $v_1$ is transparent north}. 
			This is the same as Case 1(b)(i), hence the desired transversality has been demonstrated.
			\item \emph{Assume $v_1\in\cN^0$ and $v_1$ is transparent south}. 
			This is the same as Case 2(b)(i), hence the desired transversality has been demonstrated.
			\item \emph{Assume $v_1\not\in\cN^0$}. 
			Figure~\ref{fig:case3c}(a) provides an upper bound on the possible cell types. 
			Since $v_1\not\in\cN^0$, $\kappa_1$ cannot be of type $S$ or $SE$.
			Since $\kappa_2$ is of type $SE$, Proposition~\ref{prop:adjacentCells}(ii) and (ix) implies that $\eta_j\in \sH^1$ and $\xi_i\in \Xi^2$, respectively.
			With these restrictions Proposition~\ref{prop:adjacentCells}(i)  implies that $\kappa_4$ is of type $NW$, $N$, or $NE$, while Proposition~\ref{prop:adjacentCells}(x) implies that $\kappa_4$ is of type $NW$, $W$, or $SW$.
			Thus $\kappa_4$ is of type $NW$ as is indicated in Figure~\ref{fig:case3c}(c).			
		\end{enumerate}	
	\end{enumerate}

\begin{figure}
\begin{tikzpicture}
[scale=0.5]

	\draw[blue, thick](0,0) -- (8,0);
	\draw[blue, thick](0,4) -- (8,4);
	\draw[blue, thick](0,0) -- (0,8);
	\draw[blue, thick](8,0) -- (8,8);
	\draw[blue, thick](4,0) -- (4,8);
	\draw[blue, thick](0,8) -- (8,8);

	\draw (4.30,5.70) node {$v$};		
	\filldraw(4,6) circle (4pt);
	\draw[->] (3.5,6) -- (4.5,6);
	
	\draw (2.50,4.35) node {$v_2$};
	\filldraw(2,4) circle (4pt);
	\draw[->] (2,4.5) -- (2,3.5);

	\filldraw[blue](6,4) circle (4pt);
	\draw (6.50,3.60) node {$v_1$};
	
	\draw (4.5,1.70) node {$v_3$};
	\draw(4,2) circle (4pt);

	\draw(5.5,5) node{$\kappa_1$};
	
	\draw(6,7.25) node{$\star$};
	\draw(6,6) node{$\star$};
	\draw(6,4.75) node{$\star$};
	\draw(7.25,7.25) node{$\star$};
	\draw(7.25,6) node{$\star$};
	\draw(7.25,4.75) node{$\star$};
	
	\draw(1.25,5) node{$\kappa_2$};
	
	\draw(3.25,4.75) node{$\star$};

	\draw(1.25,1.25) node{$\kappa_3$};
	
	\draw(0.75,2) node{$\star$};
	\draw(2,2) node{$\star$};
	\draw(0.75,0.75) node{$\star$};
	\draw(2,0.75) node{$\star$};
	
	\draw(5.25,1.25) node{$\kappa_4$};
	
	\draw(4.75,3.25) node{$\star$};
	\draw(6,3.25) node{$\star$};
	\draw(7.25,3.25) node{$\star$};
	\draw(4.75,2) node{$\star$};
	\draw(6,2) node{$\star$};
	\draw(7.25,2) node{$\star$};
	\draw(4.75,0.75) node{$\star$};
	\draw(6,0.75) node{$\star$};
	\draw(7.25,0.75) node{$\star$};

	\draw(4,-0.75) node{(a)};	
\end{tikzpicture}
\qquad
\begin{tikzpicture}
[scale=0.5]

	\fill[lightgray] (0.5,4.5) -- (3.5,4.5) -- (3.5,7.5) -- (0.5,7.5) -- (0.5,4.5);
	\fill[lightgray] (4.5,4.5) -- (7.5,4.5) -- (7.5,7.5) -- (4.5,7.5) -- (4.5,4.5);
	\fill[lightgray] (3.5,4.5) -- (4.5,4.5) -- (4.5,7.5) -- (3.5,7.5) -- (3.5,4.5);
	\fill[lightgray] (0.5,0.5) -- (0.5,3.5) -- (3.5,3.5) -- (3.5,0.5) -- (0.5,0.5);	
	\fill[lightgray] (0.5,4.5) -- (3.5,4.5) -- (3.5,3.5) -- (0.5,3.5) -- (0.5,4.5);
	\fill[lightgray] (3.5,3.5) -- (4.5,3.5) -- (4.5,4.5) -- (3.5,4.5) -- (3.5,4.5);
		
	\draw[blue, thick](0,0) -- (8,0);
	\draw[blue, thick](0,4) -- (8,4);
	\draw[blue, thick](0,0) -- (0,8);
	\draw[blue, thick](8,0) -- (8,8);
	\draw[blue, thick](4,0) -- (4,8);
	\draw[blue, thick](0,8) -- (8,8);
		
	\draw (0,0.5) -- (8,0.5);
	\draw (0,3.5) -- (8,3.5);
	\draw (0,4.5) -- (8,4.5);
	\draw (0,7.5) -- (8,7.5);
	\draw (0.5,0) -- (0.5,8);
	\draw (3.5,0) -- (3.5,8);
	\draw (4.5,0) -- (4.5,8);
	\draw (7.5,0) -- (7.5,8);
	
	
		
	\filldraw(4,6) circle (4pt);
	\draw (4.30,5.70) node {$v$};
	\filldraw(2,4) circle (4pt);
	\draw (2.30,3.70) node {$v_2$};

	\draw (6.30,3.70) node {$v_1$};
	\filldraw[blue](6,4) circle (4pt);
	\draw (4.5,1.70) node {$v_3$};
	\draw(4,2) circle (4pt);
		
	\draw(2,6) node{$G^2(\kappa_2)$};
	\draw(6,6) node{$G^2(\kappa_1)$};
	\draw(4,6.75) node{$G^1(v)$};
	
	\draw(4,-0.75) node{(b)};	
\end{tikzpicture}
\qquad
\begin{tikzpicture}
[scale=0.5]

	\draw[blue, thick](0,0) -- (8,0);
	\draw[blue, thick](0,4) -- (8,4);
	\draw[blue, thick](0,0) -- (0,8);
	\draw[blue, thick](8,0) -- (8,8);
	\draw[blue, thick](4,0) -- (4,8);
	\draw[blue, thick](0,8) -- (8,8);

	\draw (4.30,5.70) node {$v$};		
	\filldraw(4,6) circle (4pt);
	\draw[->] (3.5,6) -- (4.5,6);
	
	\draw (2.50,4.35) node {$v_2$};
	\filldraw(2,4) circle (4pt);
	\draw[->] (2,4.5) -- (2,3.5);

	\draw(6,4) circle (4pt);
	\draw (6.50,3.60) node {$v_1$};
	
	\draw (4.5,1.70) node {$v_3$};
	\draw(4,2) circle (4pt);

	\draw(5.5,5) node{$\kappa_1$};
	
	\draw(6,7.25) node{$\star$};
	\draw(6,6) node{$\star$};
	\draw(7.25,7.25) node{$\star$};
	\draw(7.25,6) node{$\star$};
	
	\draw(1.25,5) node{$\kappa_2$};
	
	\draw(3.25,4.75) node{$\star$};

	\draw(1.25,1.25) node{$\kappa_3$};
	
	\draw(0.75,2) node{$\star$};
	\draw(2,2) node{$\star$};
	\draw(0.75,0.75) node{$\star$};
	\draw(2,0.75) node{$\star$};
	
	\draw(5.25,1.25) node{$\kappa_4$};
	
	\draw(4.75,3.25) node{$\star$};
		
	\draw(4,-0.75) node{(c)};	
\end{tikzpicture}
\caption{(a) 
The possible cell types for Case 3(c).
(b) Necessary tiles for $N_\cA$ in Case 3(c).
(c) The possible cell types for Case 3(c)(iii).
}
\label{fig:case3c}
\end{figure}
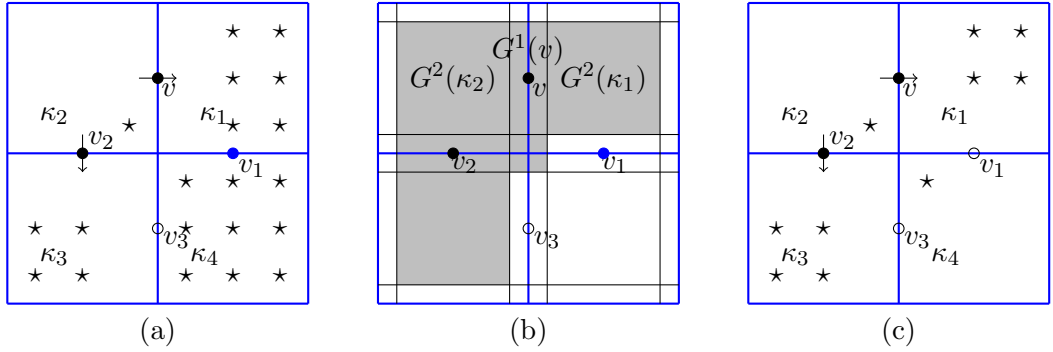

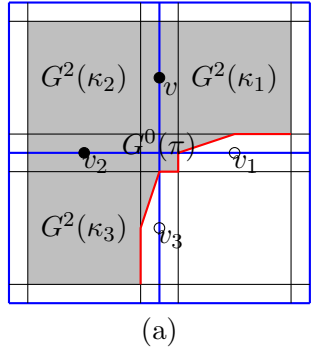
\begin{figure}
\begin{center}
\begin{tikzpicture}
[scale=0.5]

	\fill[lightgray] (0.5,4.0) -- (3.5,4.0) -- (3.5,7.5) -- (0.5,7.5) -- (0.5,4.0);
	\fill[lightgray] (4.5,4.5) -- (7.5,4.5) -- (7.5,7.5) -- (4.5,7.5) -- (4.5,4.5);
	\fill[lightgray] (0.5,0.5) -- (3.5,0.5) -- (3.5,3.5) -- (0.5,3.5) -- (0.5,0.5);
	\fill[lightgray] (0.5,3.5) -- (3.5,3.5) -- (3.5,4.5) -- (0.5,4.5) -- (0.5,3.5);
	\fill[lightgray] (3.5,4.5) -- (4.5,4.5) -- (4.5,7.5) -- (3.5,7.5) -- (3.5,4.5);
	\fill[lightgray] (3.5,3.5) -- (4.5,3.5) -- (4.5,4.5) -- (3.5,4.5) -- (3.5,3.5);
	\fill[lightgray] (3.5,3.5) -- (4.0,3.5) -- (3.5,2.0) -- (3.5,3.5);
	\fill[lightgray] (4.5,4.0) -- (4.5,4.5) -- (6.0,4.5) -- (4.5,4.0);
	
	\draw[blue, thick](0,0) -- (8,0);
	\draw[blue, thick](0,4) -- (8,4);
	\draw[blue, thick](0,0) -- (0,8);
	\draw[blue, thick](8,0) -- (8,8);
	\draw[blue, thick](4,0) -- (4,8);
	\draw[blue, thick](0,8) -- (8,8);
	
	\draw (0,0.5) -- (8,0.5);
	\draw (0,3.5) -- (8,3.5);
	\draw (0,4.5) -- (8,4.5);
	\draw (0,7.5) -- (8,7.5);
	\draw (0.5,0) -- (0.5,8);
	\draw (3.5,0) -- (3.5,8);
	\draw (4.5,0) -- (4.5,8);
	\draw (7.5,0) -- (7.5,8);
	
	\draw[red, thick](7.5,4.5) -- (6.0,4.5) -- (4.5,4.0) -- (4.5,3.5) -- (4.0,3.5) -- (3.5,2.0) -- (3.5,0.5);
	
	
	\filldraw(4,6) circle (4pt);
	\draw (4.30,5.70) node {$v$};
	\filldraw(2,4) circle (4pt);
	\draw (2.30,3.70) node {$v_2$};
	\draw(6,4) circle (4pt);
	\draw (6.30,3.70) node {$v_1$};
	\draw(4,2) circle (4pt);
	\draw (4.30,1.70) node {$v_3$};
	
	\draw(2,6) node{$G^2(\kappa_2)$};
	\draw(6,6) node{$G^2(\kappa_1)$};
	\draw(2,2) node{$G^2(\kappa_3)$};
	\draw(4,4.2) node{$G^0(\pi)$};
	
	\draw(4,-0.75) node{(a)};	
\end{tikzpicture}
\end{center}
\caption{Tiles and chips for Case 3(c)(iii)}
\label{fig:Case3ciii}
\end{figure}

\end{enumerate}	
\end{proof}

\section{Proof of Theorem}
\label{sec:theorem}

The focus of this section is on the proof of Theorem~\ref{thm:controlledPert}, the main theorem of this paper.
We begin with a preliminary result concerning the existence of global attractors and conclude with minor extensions
of Theorem~\ref{thm:controlledPert}.

\begin{prop}
\label{prop:globalAtt}
Consider 
\begin{equation}
\label{eq:general}
\dot{x} = -\Gamma x + f(x),\quad x\in (0,\infty)^n
\end{equation}
where $f\colon (0,\infty)^n \to (0,\infty)^n$ is Lipschitz continuous and satisfies the conditions that there exist positive constants $a_i^\pm$ such that
\[
0 < a_i^- < f_i(x) < a_i^+, \quad i=1,\ldots, n
\]
for all $x\in (0,\infty)^n$.
Furthermore, assume that $\Gamma$ is a diagonal matrix with diagonal elements $\gamma_i >0$.
Then, there exists a global compact attractor $X\subset (0,\infty)^n$ for the flow generated by \eqref{eq:general}.
\end{prop}

We leave the proof of Proposition~\ref{prop:globalAtt} to the reader noting that for positive $a^-$ sufficiently small and $a^+$ sufficiently large, the vector field \eqref{eq:general} is transverse in on the boundary of $[a^-,a^+]^n$.

\begin{thm}
\label{thm:controlledPert}
Let $\Sigma = \Sigma(\Gamma,\Lambda,\Xi^1,\Xi^2,\sH^1,\sH^2)$ be a switching system as defined by Definition~\ref{defn:switchingsystem}.
Let $\cF$ be the associated state transition diagram  given by Definition~\ref{defn:stateTransitionGraph}.
Let $\sMG(\cF) = \setof{\cM(p)\mid p\in (\sP,\leq_\sP)}$ be the associated Morse graph.
Choose $0< \delta < \delta^*$, where $\delta^*$ satisfies \eqref{eq:mindelta},
and let 
\begin{equation}
\label{eq:dconstrained2}
\dot{x} = -\Gamma x + f^{(\delta)}(x), \quad x\in (0,\infty)^2
\end{equation}
be an associated $\delta$-constrained continuous switching system as defined in Section~\ref{sec:controlP}. 
Let $\varphi$ be the flow associated with \eqref{eq:dconstrained2} and let $X\subset (0,\infty)^2$ be the associated global attractor.
Let $(\bar{\sP},\leq_\sP)$ be the poset given by Definition~\ref{defn:barP}.
Then, there exists a Morse decomposition for $X$ under $\varphi$ with Morse sets $M(p)$, $p\in\bar{\sP}$, and an admissible order $\leq_{\bar{\sP}}$.
\end{thm}

As indicated in the introduction Theorem~\ref{thm:controlledPert} is obtained as an application of Theorem~\ref{thm:translate}.
Thus, the main task of this section is to establish that desired  finite distributive lattice $(\sN,\wedge,\vee,{\bf 0},{\bf 1})$ that satisfies hypothesis (1)-(3) of  Theorem~\ref{thm:translate}.
With this in mind we introduce the following lemmas and proposition.

\begin{lem}
\label{lem:2bi}
Let $\cA\in\sAtt(\cF)$.
Let  $E(i,j)$ be an elementary domain with the local vector field in Case 2(b)(i) of the proof of  Proposition~\ref{prop:trapN}.
If $v\in \cA$, then $\setof{v,v_1,v_2}\subset \cA$.
\end{lem}

\begin{proof}
By Figure~\ref{fig:Case2b}(a), $\kappa_1$ is of type $S$ or $SE$. Thus $v\to v_1$ and hence $v_1\in\cA$.
Furthermore, $\kappa_2$ is of type $SE$ and hence by Proposition~\ref{prop:cellAtt_Rev} $v_2\in \cA$.
\end{proof}

We now extend Lemma~\ref{lem:2bi} to a sequence of elementary domains.
For this need to be able to consider symmetric versions of Case 2(b)(i) as indicated in Figure~\ref{fig:Case2bi*}.
Recall that given an elementary domain $E(i,j)$, $\setof{v_{i,\overline{j}}, v_{i,\overline{j-1}}, v_{\overline{i},j}, v_{\overline{i-1},j}}\subset \cE(i,j)$.

\begin{figure}
\begin{tikzpicture}
[scale=0.5]

	\draw[blue, thick](0,0) -- (8,0);
	\draw[blue, thick](0,4) -- (8,4);
	\draw[blue, thick](0,0) -- (0,8);
	\draw[blue, thick](8,0) -- (8,8);
	\draw[blue, thick](4,0) -- (4,8);
	\draw[blue, thick](0,8) -- (8,8);

	\draw (4.30,5.70) node {$v_1$};		
	\filldraw(4,6) circle (4pt);
	\draw[->] (3.5,6) -- (4.5,6);
	\draw (2.50,4.35) node {$v$};	
	\filldraw(2,4) circle (4pt);
	\draw[->] (2,3.5) -- (2,4.5);	
	\draw (6.50,3.60) node {$v_3$};
	\draw(6,4) circle (4pt);
	\draw (4.5,1.70) node {$v_2$};
	\filldraw(4,2) circle (4pt);
	\draw[->] (3.5,2) -- (4.5,2);
	
	\draw(5.5,5) node{$\kappa_4$};
	
	\draw(6,7.25) node{$\star$};
	\draw(6,6) node{$\star$};
	\draw(7.25,7.25) node{$\star$};
	\draw(7.25,6) node{$\star$};
	
	\draw(1.25,5) node{$\kappa_1$};
	
	\draw(3.25,7.25) node{$\star$};
	\draw(3.25,6) node{$\star$};

	\draw(1.25,1.25) node{$\kappa_2$};
	
	\draw(3.25,3.25) node{$\star$};
	
	\draw(5.25,1.25) node{$\kappa_3$};
	
	\draw(6,2) node{$\star$};
	\draw(7.25,2) node{$\star$};
	\draw(6,0.75) node{$\star$};
	\draw(7.25,0.75) node{$\star$};

	\draw(4,-0.75) node{(a)};	
\end{tikzpicture}
\qquad
\begin{tikzpicture}
[scale=0.5]

	\draw[blue, thick](0,0) -- (8,0);
	\draw[blue, thick](0,4) -- (8,4);
	\draw[blue, thick](0,0) -- (0,8);
	\draw[blue, thick](8,0) -- (8,8);
	\draw[blue, thick](4,0) -- (4,8);
	\draw[blue, thick](0,8) -- (8,8);

	\draw (4.30,5.70) node {$v_1$};		
	\filldraw(4,6) circle (4pt);
	\draw[->] (3.5,6) -- (4.5,6);
	\draw (2.50,4.35) node {$v$};	
	\filldraw(2,4) circle (4pt);
	\draw[->] (2,4.5) -- (2,3.5);	
	\draw (6.50,3.60) node {$v_3$};
	\draw(6,4) circle (4pt);
	\draw (4.5,1.70) node {$v_2$};
	\filldraw(4,2) circle (4pt);
	\draw[->] (3.5,2) -- (4.5,2);
	
	\draw(5.5,5) node{$\kappa_4$};
	
	\draw(6,7.25) node{$\star$};
	\draw(6,6) node{$\star$};
	\draw(7.25,7.25) node{$\star$};
	\draw(7.25,6) node{$\star$};
	
	\draw(1.25,5) node{$\kappa_1$};
	
	\draw(3.25,4.75) node{$\star$};

	\draw(1.25,1.25) node{$\kappa_2$};
	
	\draw(3.25,2) node{$\star$};
	\draw(3.25,0.75) node{$\star$};
	
	\draw(5.25,1.25) node{$\kappa_3$};
	
	\draw(6,2) node{$\star$};
	\draw(7.25,2) node{$\star$};
	\draw(6,0.75) node{$\star$};
	\draw(7.25,0.75) node{$\star$};

	\draw(4,-0.75) node{(b)};	
\end{tikzpicture}
\qquad
\begin{tikzpicture}
[scale=0.5]

	\draw[blue, thick](0,0) -- (8,0);
	\draw[blue, thick](0,4) -- (8,4);
	\draw[blue, thick](0,0) -- (0,8);
	\draw[blue, thick](8,0) -- (8,8);
	\draw[blue, thick](4,0) -- (4,8);
	\draw[blue, thick](0,8) -- (8,8);

	\draw (4.75,5.70) node {$v_{i,\overline{j}}$};		
	\filldraw[blue](4,6) circle (4pt);
	\draw (3.0,3.55) node {$v_{\overline{i-1},j}$};	
	\filldraw[blue](2,4) circle (4pt);
	\draw (6.60,3.55) node {$v_{\overline{i},j}$};
	\filldraw[blue](6,4) circle (4pt);
	\draw (5.0,1.70) node {$v_{i,\overline{j-1}}$};
	\filldraw[blue](4,2) circle (4pt);
	
	\draw(5.5,5) node{$\kappa_4$};
	
	\draw(7.25,7.25) node{$\star$};
	\draw(7.25,6) node{$\star$};
	
	\draw(1.25,5) node{$\kappa_1$};
	
	\draw(3.25,7.25) node{$\star$};
	\draw(3.25,6) node{$\star$};

	\draw(1.25,1.25) node{$\kappa_2$};
	
	\draw(3.25,2) node{$\star$};
	\draw(3.25,0.75) node{$\star$};
	
	\draw(5.75,1.0) node{$\kappa_3$};
	
	\draw(7.25,2) node{$\star$};
	\draw(7.25,0.75) node{$\star$};
		
	\draw(4,-0.75) node{(c)};	
\end{tikzpicture}

\caption{(a) Possible cell types for Case 2(b)(i)* obtained from Case 2(b)(i) by counterclockwise rotation of $90^\circ$.
(b) Possible cell types for Case 2(b)(i)** obtained from Case 2(b)(i)* by a vertical reflection.
(c) Possible cell types in elementary domain $E(i,j)$ in Lemma~\ref{lem:case2bi*}.
}
\label{fig:Case2bi*}
\end{figure}
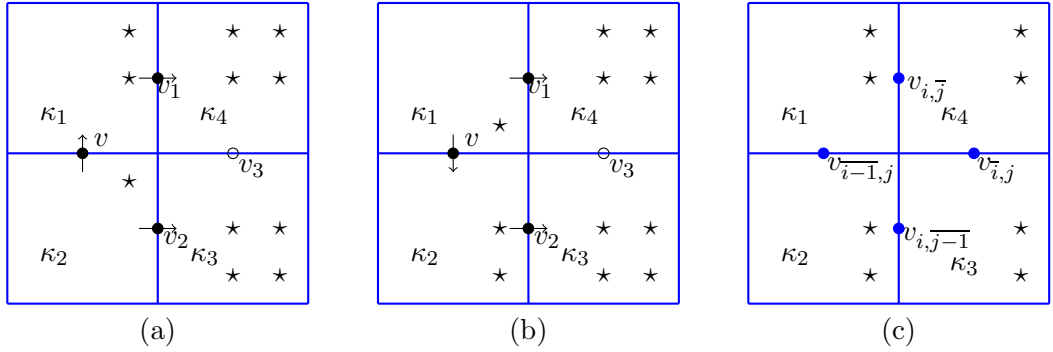

\begin{lem}
\label{lem:case2bi*}
Fix $L\geq 1$.
Consider elementary domains $E(i-\ell,j)$, $\ell=0,\ldots, L$.
Assume that in region regions $E(i-\ell,j)$, $\ell=1,\ldots, L-1$, the possible cell types are of the form indicated in Figure~\ref{fig:Case2bi*}(c). 
Assume that in region $E(i-L,j)$ the possible cell types are of the form  indicated in Figure~\ref{fig:Case2bi*}(a).
Assume that in region $E(i,j)$ the possible cell types are of the form associated with Case 3(a)(ii)(B) of the proof of  Proposition~\ref{prop:trapN}.
Let $\cL = \setof{v_{i-\ell,\overline{j}}, v_{i-\ell,\overline{j-1}}\mid \ell =0,\ldots, L}$.
If $v_{\overline{i-L-1},j}\in\cA_0$, then $\cL\subset \cA_0$.
\end{lem}

\begin{proof}
By Lemma~\ref{lem:2bi}, $v_{i-L,\overline{j}}, v_{i-L,\overline{j-1}}\in \cA_0$.
From Figure~\ref{fig:Case2bi*}(c), $v_{i-L,\overline{j}}\in \cV_e(\kappa_4)$ and $v_{i-L+1,\overline{j}}\in \cV_a(\kappa_4)$ for $\kappa_4$ associated with $E(i-L+1,j)$.  
Similarly, $v_{i-L,\overline{j-1}}\in \cV_e(\kappa_3)$ and $v_{i-L+1,\overline{j-1}}\in \cV_a(\kappa_3)$ for $\kappa_3$ associated with $E(i-L+1,j)$. 
Thus by the definition of $\cF$, $v_{i-L,\overline{j}} \to v_{i-L+1,\overline{j}}$ and $v_{i-L,\overline{j-1}} \to v_{i-L+1,\overline{j-1}}$.
Therefore, $v_{i-L+1,\overline{j}}, v_{i-L+1,\overline{j-1}}\in \cA_0$.
The result now follows by induction.
\end{proof}

\begin{prop}
\label{prop:N}
Let $\Sigma = \Sigma(\Gamma,\Lambda,\Xi^1,\Xi^2,\sH^1,\sH^2)$ and $\cF$ be as in Theorem~\ref{thm:controlledPert}.
The set 
\[
\sN = \setof{N_\cA\mid \cA\in \sAtt(\cF)},
\]
where $N_\cA$ is given by Definition~\ref{defn:latticeMorphism}, is a finite distributed lattice with lattice operations
\begin{align}
\label{eq:wedge}
N_{\cA_0}\wedge N_{\cA_1} &:= N_{\cA_0\wedge\cA_1} \subset N_{\cA_0}\cap N_{\cA_1} \\
\label{eq:vee}
N_{\cA_0}\vee N_{\cA_1} &:= N_{\cA_0\vee\cA_1} = N_{\cA_0}\cup N_{\cA_1}
\end{align}
with minimal element ${\bf 0} = \emptyset$.
\end{prop}

\begin{proof}
Observe that $\sN$ is well defined and \eqref{eq:wedge} and \eqref{eq:vee} are well defined lattice operations since $\sAtt(\cF)$ is a lattice with $\vee$ and $\wedge$ defined by \eqref{eq:veeAtt} and \eqref{eq:wedgeAtt}, respectively.

By Lemma~\ref{lem:trivialN} ${\bf 0} = \emptyset$.
Thus all that remains to be shown is that the inclusion and equality of \eqref{eq:wedge} and \eqref{eq:vee} hold.

To prove that the inclusion of \eqref{eq:wedge} we consider several representative cases leaving the rest to the reader.
Let $\cA' := \cA_0\wedge \cA_1 \subset \cA_0\cap \cA_1$.

Assume that a wide chip $C^w$ is contained in $N_{\cA'}$. 
Without loss of generality we can assume that this implies that there is an elementary region $E(i,j)$ such that $N_{\cA'}\cap E(i,j)$ gives rise to Case 2(a)(i)(B) in the proof of Proposition~\ref{prop:trapN}. 
Observe that $v_2\in\cW$ and hence $v_2\not\in \cA_0 \cup \cA_1$.
If $v_3\not\in \cA_0 \cup \cA_1$, then $C^w$ belongs to both $N_{\cA_0}$ and $N_{\cA_1}$.
If $v_3\in \cA_0$ or $v_3 \in \cA_1$, then $G^1(v_3)$ belongs to $N_{\cA_0}$ or $N_{\cA_1}$.
Thus in all these cases  $G^1(v_3)$ belongs to $N_{\cA_0}\cap N_{\cA_1}$.

Assume that a narrow chip $C^n$ is contained in $N_{\cA'}$. 
Again, without loss of generality we can assume that this implies that there is an elementary region $E(i,j)$ such that $N_{\cA'}\cap E(i,j)$ gives rise to Case 2(a)(i)(A) or Case 3(c)(iii).
Recall that in Case  2(a)(i)(A), $\kappa_3$ is of type $NE$.
Thus, if $v_2\in\cA_i$ or $v_3\in\cA_i$, $i=0$ or $1$, then as indicated in Figure~\ref{fig:AcapT_Rev} by {\bf Rule 1} $G^2(\kappa_3)\cup G^1(v_3)\subset N_{\cA_i}$.
Therefore, $C^n$ is contained in $N_{\cA_0}\cap N_{\cA_1}$.
A similar argument applies to Case 3(c)(iii).

Assume that a $0$-tile $G^0(\pi)$ is contained in $N_{\cA'}$. 
If $\pi = (\xi_i,\eta_j)$, then {\bf Rule 2} is applicable to $E(i,j)$.
Reviewing the proof of Proposition~\ref{prop:trapN} we see that this happens in Cases 1(a)(ii), 1(b), 1(c)(i), 2(a), 2(b), 3(b), and 3(c).
In the subcases of Cases 1 and 2, $v,v_1\in \cT\cap\cA'$ and thus $v,v_1\in \cT\cap\cA_i$, $i=0,1$.
Therefore, after applying {\bf Rule 1} in these cases to $N_{\cA_i}$, $i=0,1$,  {\bf Rule 2} is applicable, and hence $G^0(\pi)$ is contained in $N_{\cA_i}$, $i=0,1$.
The subcases of Case 3 follows from a similar argument based on $v,v_2\in \cT\cap\cA'$.

Assume that a $1$-tile $G^1(v_\alpha)$ is contained in $N_{\cA'}$. 
This implies that {\bf Rule 1} or {\bf Rule 5} applies.
Observe that $G^1(v_\alpha)$ is introduced by {\bf Rule 1} if $v_\alpha\in \cT\cap\cA'\cap \cE(i,j)$.
But $v_\alpha\in \cT\cap\cA'\cap \cE(i,j)$ implies that $v_\alpha\in \cT\cap(\cA_0\cap \cA_1)\cap \cE(i,j)$ and
hence by {\bf Rule 1} $G^1(v_\alpha)$ belongs to $N_{\cA_i}$, $i=0,1$.
If $G^1(v_\alpha)$ is introduced by {\bf Rule 5}, then without loss of generality we can assume that we are in the setting of Case 2(b)(i) or Case 3(a)(ii)(B).
We leave it to the reader to check these cases using similar arguments as above.

Finally, assume that a $2$-tile $G^2(\kappa_\alpha)$  is contained in $N_{\cA'}$. 
This implies that {\bf Rule 0} and/or {\bf Rule 1} applies. In both cases the fact that $\cA'\subset \cA_0\cap \cA_1$
implies that {\bf Rule 0} and/or {\bf Rule 1} also applies to $N_{\cA_i}$, $i=0,1$.

We now turn to the proof of equality in \eqref{eq:vee}.
In general, since $\cA' := \cA_0\vee \cA_1 = \cA_0\cup \cA_1$, the same arguments that were used to verify the inclusion of \eqref{eq:wedge} lead to equality in the setting of \eqref{eq:vee}.
The exception is the introduction of $1$-tile $G^1(v_\alpha)$ due to {\bf Rule 5}.
This occurs in two cases Case 2(b)(i) or Case 3(a)(ii)(B).

Consider Case 2(b)(i).  Since $\cA' = \cA_0\cup \cA_1$ and $v\in\cA'$, we can without loss of generality assume that $v\in \cA_0$.
Applying Lemma~\ref{lem:2bi}  we see that $v,v_1,v_2\in \cA_0$ and hence {\bf Rule 5} applies to $\cA_0$.
Hence $\cA'\cap E(i,j) = (\cA_0\cup \cA_1)\cap E(i,j)$.

Finally, consider Case 3(a)(ii)(B). The assumption is that $G^1(v_2)\subset N_{\cA'}$, but $v_2\not\in \cA'$.
Thus, $G^1(v_2)$ was introduced by  {\bf Rule 5} applied to $E(i-1,j)$.
If $E(i-1,j)$ is of the form Case 3(a)(ii)(B), then we repeat the argument to conclude that $G^1(v_2)$ for $E(i-1,j)$ was introduced by 
{\bf Rule 5} applied to $E(i-2,j)$.
Since there are only a finite number of cases we can assume that there exists $L\geq 1$ such that $E(i-L,j)$ is of the form Case 2(b)(i)* or Case 2(b)(i)** (see Figure~\ref{fig:Case2bi*}(a) and (b)) and $E(i-\ell,j)$ is of the form indicated in Figure~\ref{fig:Case2bi*}(c) for $\ell = 1,\ldots, L-1$.
The desired equality in \eqref{eq:vee} now follows from Lemma~\ref{lem:case2bi*}.
\end{proof}

\begin{proof}[Proof of Theorem~\ref{thm:controlledPert}]
Proposition~\ref{prop:trapN} guarantees that for all $\cA\in\sAtt(\cF)$, $N_\cA \in \sInvset^+(\varphi)$.
Thus hypothesis (1) of Theorem~\ref{thm:translate} is satisfied.
Lemma~\ref{lem:trivialN} indicates that hypothesis (2) of Theorem~\ref{thm:translate} is satisfied.
Finally, Proposition~\ref{prop:N} guarantees the validity of hypothesis (3).
\end{proof}

\begin{rem}
\label{rem:locationM}
Because Theorem~\ref{thm:controlledPert} is obtained as an application of Theorem~\ref{thm:translate} we also have information concerning the location of the Morse sets $M(p)$ of the $\delta$-constrained continuous switching system.
In particular, each $p\in \bar{\sP}$  is associated to a unique a join irreducible element $N_\cA$ of the lattice $\sN$ and hence a unique $\cA\in\sAtt(\cF)$ which is join irreducible in $\sAtt(\cF)$.
Theorem~\ref{thm:translate} guarantees that
\[
M(p) \subset N_\cA\setminus N_{\pred{\cA}} \subset (0,\infty)^2.
\]
\end{rem}

We conclude this section by noting that the assumption that $f^{(\delta)}$ be constant on the tiles is not necessary.
Applying the classical continuation theorem for Morse decompositions \cite{cbms} to Theorem~\ref{thm:controlledPert} gives rise to the following result.
\begin{thm}
\label{thm:continuousPert}
Let 
\begin{equation}
\label{eq:dconstrained3}
\dot{x} = -\Gamma x + f^{(\delta)}(x), \quad x\in (0,\infty)^2
\end{equation}
be a $\delta$-constrained continuous switching system with an associated Morse graph $\sMG$ for the global attractor of \eqref{eq:dconstrained3}.
If 
\[
\sup_{x\in (0,\infty)^2} \| f(x) - f^{(\delta)}(x)\| <\epsilon
\]
for sufficiently small $\epsilon >0$, then $\sMG$ is a Morse graph for the global attractor of 
\[
\dot{x} = -\Gamma x + f(x), \quad x\in (0,\infty)^2.
\]
\end{thm}

\section{Conclusion}

We conclude  with a few comments concerning the results of this paper and potential future directions.

We begin by reviewing the conclusion of Theorem~\ref{thm:controlledPert} and in particular the Morse set $M(\bar{p})$.
Extending the discussion of Remark~\ref{rem:locationM} note that 
\[
{\bf 1} = \bigvee_{\cA\in \sAtt(\cF)}\cA \in \sN
\]
and hence
\[
M(\bar{p})\subset X \setminus N_{\bf 1}.
\]
Whether or not $M(\bar{p})\neq\emptyset$ is depends on the system being considered.
However, given the construction of $\cF$ this cannot be determined given the  calculations presented in this paper.
This is essentially due to the fact that we do not identify any dynamics that is associated with white vertices.
We are currently working on algorithms for generating an alternative state transition diagrams that are capable of capturing the dynamics imposed on $\delta$-constrained continuous switching system by white vertices.
However, the computational cost of working with these alternative state transition diagrams is greater.
Whether the additional dynamical information that can be gained is worth the additional cost in the context of the analysis of biologically motivated networks remains to be seen. 

With regard to $\delta$-constrained continuous switching systems there are three points worth emphasizing.
First, within the $\delta$ collars we impose minimal assumptions on the form of $f^{(\delta)}$.
Thus we do not need to assume that $f^{(\delta)}$ is based on a particular nonlinearity, e.g.\  a Hill function, and thus the dynamics we are recovering is valid for an extremely wide set of potential models.
Second, the larger $\delta$ is the less steep $f^{(\delta)}$ need be, and hence, the less switch-like the system needs to be.
Third, we give explicit bounds on $\delta^*$ in terms of the parameters of $\Lambda$ and $\Gamma$. 
A possible consequence of this is that the computational tools developed for switching systems can be used to guide the study of the local and global dynamics of arbitrary systems of the form \eqref{eq:abstract} for fixed families of nonlinearities.
To be more precise, assume that the $f_n$ are given in terms of Hill functions and we are interested in particular dynamical structures.
Letting the Hill coefficients $k\to\infty$ produces a switching system, the global dynamics of which can be analyzed over all of parameter space using $\cF$~\cite{paper2}. 
We then can identify  parameter values at which the desired nonlinear dynamics is exhibit and determine the maximal size of perturbation  $\delta^*$.
Given $\delta^*$ one can choose $k$ sufficiently large so that the Hill function approximates a $\delta$-constrained continuous switching system for $\delta < \delta^*$.
More standard numerical methods can then be used to identify the desired Morse set for \eqref{eq:abstract} with this large Hill coefficient $k$.
Finally, numerical continuation techniques can be employed to determine if the dynamics continues to  lower biologically motivated values of $k$.

The focus of this paper is on translating information obtained from piecewise constant models in the form of switching systems that are motivated by regulatory networks, to information about the dynamics generated by Lipschitz continuous differential equations. 
However, as is mentioned in the introduction part of the motivation for this paper is our interest in the mathematically rigorous analysis of global dynamics for multiparameter systems over large regions of parameter space and the question of whether for these purposes it is computationally efficient to use the techniques described here.
With this in mind the results presented in Section~\ref{sec:proof} are much more general than those required to study regulatory networks with two nodes.
Observe that Theorem~\ref{thm:continuousPert}  suggests that given a systems of ordinary differential equations of the form \eqref{eq:abstract} one could try to compute the associated dynamics by approximating $f$ via a linear term $\Gamma$  and a piecewise constant function $\Lambda$, and then, computing the associated Morse graph for the associated switching system.
We plan to exploring the effectiveness of such a procedure. 
However, there are at least two related issues that need to be addressed.
First, we need explicit results for bounds on $\epsilon$, the acceptable size of perturbation in Theorem~\ref{thm:continuousPert}.
Second, we need to understand how to determine the threshold values used to defined the domains of the piecewise constant functions.

\vskip0.3in

{\bf Acknowledgement:} TG was partially supported by DMS-1361240, DARPA D12AP200025 and NIH grant 1R01AG040020-01 and  SH and KM have been  partially supported by NSF grants NSF-DMS-0835621, 0915019, 1125174, 1248071, and contracts from AFOSR and DARPA. HK was partially supported by Grant-in-Aid for Scientific Research (No. 25287029, No. 26310208), Ministry of Education, Science, Technology, Culture and Sports, Japan, and JST-CREST.
HO was partially supported by Grant-in-Aid for Scientific Research (No. 24540222), Ministry of Education, Science, Technology, Culture and Sports, Japan.

\bibliographystyle{plainnat}
\bibliography{switching_bib}

\end{document}